\documentclass[a4paper]{amsart}
\usepackage{amscd,amsmath, amsthm,
	amssymb,amsfonts,verbatim, enumerate}
\usepackage[cmtip, all]{xy}

\newtheorem{thm}[equation]{Theorem}
\makeatletter
\let\c@subsubsection\c@equation
\makeatother
\newtheorem{prop}[equation]{Proposition}
\newtheorem{lem}[equation]{Lemma}
\newtheorem{cor}[equation]{Corollary}

\theoremstyle{definition}

\newtheorem{defn}[equation]{Definition}
\theoremstyle{remark}
\newtheorem{remk}[equation]{Remark}
\newtheorem{remks}[equation]{Remarks}

\newtheorem{exm}[equation]{Example}
\newtheorem{exms}[equation]{Examples}
\newtheorem{notat}[equation]{Notation}

\numberwithin{equation}{subsection}

{\hfill$\square$\end{defn}}
{\hfill$\square$\end{remk}}
{\hfill$\square$\end{remks}}
{\hfill$\square$\end{exm}}
{\hfill$\square$\end{exms}}
{\hfill$\square$\end{notat}}

\newcommand{\CH}{{\rm CH}}
\newcommand{\BGH}{\mathbf{C}\mathbf{H}}

\newcommand{\red}{{\rm red}}

\newcommand{\Hom}{{\rm Hom}}

\newcommand{\Spec}{{\rm Spec \,}}
\newcommand{\Spf}{{\rm Spf \,}}

\newcommand{\Sch}{{\operatorname{\mathbf{Sch}}}}

\newcommand{\holim}{\mathop{{\rm holim}}}

\newcommand{\op}{{\text{\rm op}}}

\newcommand{\Spt}{{\mathbf{Spt}}}

\newcommand{\hocolim}{\mathop{{\rm hocolim}}}

\newcommand{\ds}{{/\kern-3pt/}}

\newcommand{\colim}{\mathop{\text{\rm colim}}}

\newcommand{\un}{\underline}

\renewcommand{\dim}{\text{\rm dim}}

\newcommand{\tuborg}{\left\{\begin{array}{ll}}
\newcommand{\sluttuborg}{\end{array}\right.}

\newcommand{\perf}{{\rm perf}}
\newcommand{\coh}{{\rm coh}}
\newcommand{\cnv}{{\rm cnv}}
\newcommand{\mcnv}{{\rm m.cnv}}

\begin{document}
\title[On the coniveau filtration on algebraic $K$-theory]{On the coniveau filtration on algebraic $K$-theory of singular schemes}
\author{Jinhyun Park and Pablo Pelaez}
\address{Department of Mathematical Sciences, KAIST, 291 Daehak-ro Yuseong-gu, Daejeon, 34141, Republic of Korea (South)}
\email{jinhyun@mathsci.kaist.ac.kr; jinhyun@kaist.edu}
\address{Instituto de Matem\'aticas, Universidad Nacional Auton\'oma de M\'exico, 04510, CDMX, M\'exico}
\email{pablo.pelaez@im.unam.mx}




\begin{abstract}
We construct two functorial filtrations on the algebraic $K$-theory of schemes of finite type over a field $k$ that may admit arbitrary singularities and may be non-reduced, one called the coniveau filtration, and the other called the motivic coniveau filtration. 
Restricting to the subcategory of smooth $k$-schemes,
our coniveau filtration coincides with the classical coniveau (also known as the topological) filtration on algebraic $K$-theory of D. Quillen, whereas our motivic coniveau filtration coincides with the homotopy coniveau filtration for algebraic $K$-theory of M. Levine.
\end{abstract}

\maketitle

\setcounter{tocdepth}{1} 

\tableofcontents

\section{Introduction}

The objective of this article is to propose two functorial filtrations on higher algebraic $K$-theory \cite{Quillen} of schemes of finite type over an arbitrary field, where the schemes can be arbitrarily singular or non-reduced, extending two types of decreasing filtrations on the algebraic $K$-theory of smooth schemes.

\medskip

We know that (SGA VI \cite{SGA6}) when $Y$ is a smooth irreducible quasi-projective $k$-scheme, there is the natural isomorphism $K_0 (Y) \overset{\sim}{ \to} G_0 (Y)$ from the Grothendieck group of locally free sheaves of finite type to that of coherent sheaves. The key idea behind it is given by projective resolutions of finite lengths of coherent sheaves backed by the Auslander-Buchsbaum theorem \cite{AB}. Since the group $G_0 (Y)$ has the natural coniveau (also called topological, or codimension) filtration, the isomorphism induces a filtration on $K_0 (Y)$.

We have the cycle class map
\begin{equation}\label{eqn:cycle class intro}
\CH^q (Y) \to gr _{top} ^q K_0 (Y)
\end{equation}
from the Chow group of codimension $q$-cycles to the associated graded of $K_0 (Y)$ with respect to the filtration, and it is an isomorphism after tensoring with $\mathbb{Q}$. This result is often referred to as (a part of) the Grothendieck-Riemann-Roch theorem. It was further extended to higher Chow groups and higher algebraic $K$-theory of smooth $k$-schemes by S. Bloch combining \cite{Bloch HC} and \cite{Bloch moving}, or M. Levine \cite{Levine K}.

\medskip

When one tries to generalize the result to a singular $Y$, unfortunately a few things in the above go wrong. For instance, the Chow groups $\CH^q (Y)$ in \eqref{eqn:cycle class intro} are not contravariant functorial in $Y$, while $K_0 (Y)$ is. The homomorphism $K_0 (Y) \to G_0 (Y)$ is no longer an isomorphism so that the coniveau filtration on $G_0 (Y)$ does not induce a natural filtration on $K_0 (Y)$. Given these circumstances, so far an analogue of the Grothendieck-Riemann-Roch theorem for $K$-theory of singular schemes has not been available. 

This is where this article aims to make a contribution by building some stepping stones. In this article, we define two functorial decreasing filtrations on the algebraic $K$-theory on the category of $k$-schemes of finite type. 

\medskip

The following first main theorem of the article is on the first filtration (see \eqref{final uscnv fil}, Definition \ref{defn:cnv cech}, Proposition \ref{prop:sm cnv}, Theorem \ref{thm:cech functoriality}):

\begin{thm}\label{thm:main intro-0}
Let $n\geq 0$ be an integer. For a $k$-scheme $Y$ of finite type with possibly arbitrary singularities, there exists a tower in the 
homotopy category of spectra $\widehat{\mathcal{G}} ^\bullet _{Y,0} \rightarrow \mathcal{K}(Y)$, which induces
a decreasing filtration $F_{\cnv} ^{\bullet} K_n (Y)$ on $K_n(Y)$, such that
\begin{enumerate}
\item the tower and the filtration are functorial in $Y$, and
\item if $Y$ is smooth and equidimensional, the tower coincides with Quillen's classical coniveau tower (recalled in \eqref{Quillen tower}) up to weak-equivalence and the filtration coincides with the classical coniveau filtration on $K_n (Y)$. 
\end{enumerate}

We will call $F_{\cnv} ^{\bullet}$, the \emph{coniveau filtration} on $K_n (Y)$.
\end{thm}

Theorem \ref{thm:main intro-0} implies in particular that the pull-back $g^*: K_i (Y_2) \to K_i (Y_1)$ induced by a map between smooth $k$-schemes
$g:Y_1 \rightarrow Y_2$ respects the classical coniveau filtration. This is already known, e.g. see H. Gillet \cite[Theorem 83, Lemma 84, p.283]{Gillet}, which is based on the technique of deformation to the normal cone and $\mathbb{A}^1$-invariance of the algebraic $K$-theory of regular schemes. However, the proof given in this article offers a new argument even for the classical smooth case, not directly relying on the $\mathbb{A}^1$-invariance.

In a future work, we plan to investigate the ``Brown-Gersten-Quillen spectral sequence" determined by the tower in Theorem \ref{thm:main intro-0} and its potential application for a Bloch-Quillen-type formula for the cycle groups $\mathbf{CH}^q(Y,0)$ defined by the first author
\cite{P general}.

\medskip

The coniveau filtration, as in the above, is known to be unsuitable for constructing a motivic analogue of the Atiyah-Hirzebruch spectral sequence (\cite{AH}) even on smooth schemes. To resolve this issue on smooth schemes, different filtrations were considered by Friedlander-Suslin \cite{FS} and M. Levine \cite{Levine hcnv}, where the latter filtration is called the \emph{homotopy coniveau} filtration. 

The second filtration we define in this paper, is aimed to be an analogue of such filtrations on the category of schemes of finite type. This is a cubical enrichment of the above coniveau tower in a sense (see \eqref{final mcnv fil}, Definition \ref{defn:cnv cech}, Proposition \ref{prop:sm m.cnv}, Theorem \ref{thm:cech functoriality}):

\begin{thm}\label{thm:main intro}
Let $n\geq 0$ be an integer. For a $k$-scheme $Y$ of finite type with possibly arbitrary singularities, there exists a tower in the homotopy category of spectra 
$\widehat{\mathcal{G}} ^\bullet _{Y} \rightarrow \mathcal{K}(Y),$
 which induces a decreasing filtration $F_{\mcnv} ^{\bullet} K_n (Y)$ on $K_n(Y)$, such that 
 \begin{enumerate}
\item the tower and the filtration are functorial in $Y$, and
 \item if $Y$ is smooth and equidimensional, the tower coincides with Levine's homotopy coniveau tower  \cite[\S 2.1]{Levine hcnv}, up to weak-equivalence.
\end{enumerate}

We will call $F_{\mcnv} ^{\bullet}$, the \emph{motivic coniveau filtration} on $K_n (Y)$.

In particular, the spectral sequence associated to the tower $\widehat{\mathcal{G}} ^\bullet _{Y}$ restricts to the motivic Atiyah-Hirzebruch spectral sequence, when $Y$ is smooth and equidimensional.
\end{thm}

For a general scheme of finite type, we plan to study the spectral sequence determined by the tower in Theorem \ref{thm:main intro} in a future work. We expect there may be natural isomorphisms from the $E_2$-terms to the new higher Chow groups $\BGH^q(Y,n)$ of \cite{P general}.

\medskip

A basic idea in our constructions of the filtrations of Theorems \ref{thm:main intro-0} and \ref{thm:main intro} is to use formal neighborhoods $\widehat{X}$ of $Y$ associated to a closed immersion $Y \hookrightarrow X$ into an equidimensional smooth $k$-scheme, if such an immersion exists. This idea was used in Grothendieck's studies of \'etale fundamental groups in SGA I \cite{SGA1} as well as Hartshorne's studies of the algebraic de Rham cohomology of singular varieties in \cite{Hartshorne DR}. This idea is also a cornerstone in \cite{Park Tate}, \cite{P general} on studies of the motivic cohomology of singular schemes.

\medskip

To illustrate part of the strategy of the construction, for the sake of discussion, for a while suppose $Y$ is affine so that it does admit such a closed immersion $Y \hookrightarrow X$. From the regularity of $X$, we deduce that $\widehat{X}$ is a regular formal scheme (Lemma \ref{lem:exoskeleton}). For such $\widehat{X}$, the natural functor $\mathcal{D}_{\perf} (\widehat{X}) \to \mathcal{D}_{\coh} (\widehat{X})$ from the perfect complexes to pseudo-coherent complexes is an equivalence (Lemma \ref{lem:coh=perfect}). To have a higher structure, we propose to replace $\widehat{X}$ by the co-cubical formal scheme $\widehat{X} \times \square^{\bullet}$, where $\square:= \mathbb{P}^1 \setminus \{ 0, \infty \}$, bringing in some flavors of cubical higher Chow cycles on formal schemes of \cite{Park Tate} and \cite{P general}. Each $\widehat{X} \times \square^n$ is still regular (see \cite[Theorem 8.10, p.39]{GS} or \cite[Th\'eor\`eme 7, p.398]{Salmon}).

For integers $q \geq 0$, we can consider the triangulated subcategory $\mathcal{D}^q_{\coh} (\widehat{X},n) \subset \mathcal{D}_{\coh} (\widehat{X} \times \square^n)$ generated by the coherent sheaves whose associated cycles are ``higher Chow cycles over $\widehat{X}$" in $z^{\geq q} (\widehat{X}, n)$ (see Definition \ref{defn:HCG}), i.e. they have the codimension $\geq q$ on $\widehat{X} \times \square^n$, intersect properly with all faces $\widehat{X} \times F$ for faces $F \subset \square^n$ as well as their ``special fibers" $\widehat{X}_{\red} \times F$, defined by the ideals of definition, being cycles over formal schemes. The associated Waldhausen $K$-spaces (\cite{TT} and \cite{Waldhausen}) 
$$
\cdots \to \mathcal{G}^q (\widehat{X},n) \to \cdots \to \mathcal{G}^0 (\widehat{X},n) \to \mathcal{G} (\widehat{X} \times \square^n)
$$
define a tower of the cubical spaces $(\un{n} \mapsto \mathcal{G}^{q} (\widehat{X}, n))$ over $q \geq 0$, where \emph{spaces} here mean spectra. The geometric realizations $ \mathcal{G}^q (\widehat{X}):= | \un{n} \mapsto \mathcal{G}^q (\widehat{X}, n) |$ give a tower of spaces
$$
\cdots \to \mathcal{G}^q (\widehat{X}) \to \cdots \to \cdots \mathcal{G}^0 (\widehat{X}) \to | \un{n} \mapsto \mathcal{G}(\widehat{X} \times \square^n)|,$$
where the last space is weak-equivalent to $\mathcal{G} (\widehat{X})$. 

Taking the images of their higher homotopy groups, we have a decreasing filtration $F^{\bullet}$ on $G_n (\widehat{X})= \pi_n \ \mathcal{G} (\widehat{X})$ for $n \geq 0$. This induces a filtration on $K_n (\widehat{X})$ via the isomorphism $K_n (\widehat{X}) \overset{\sim}{\to} G_n (\widehat{X})$.

\medskip

As the first attempt, consider the induced filtration on $K_n (Y)$ by taking the images under the natural map $K_n (\widehat{X}) \to K_n (Y)$. More precisely, let
\begin{equation}\label{eqn:primitive cnv}
\tuborg 
F_{X} ^q K_n (Y): = K_n (Y), & \mbox{ for } q \leq 0,
 \\
F_{ X} ^q K_n (Y) := {\rm Im} ( F ^q K_n (\widehat{X}) \to K_n (Y)), & \mbox{ for } q \geq 1.\sluttuborg
\end{equation}

There are some apparent problems in the filtration of \eqref{eqn:primitive cnv} on $K_n (Y)$. Firstly, in general a finite type $k$-scheme $Y$ may not admit a global closed immersion $Y \hookrightarrow X$ into an equidimensional smooth $k$-scheme. Secondly, even if one has it, the filtration described in \eqref{eqn:primitive cnv} depends on the embedding. An additional subtle problem is that the homotopy cofiber of $\mathcal{G}^{q+1} (\widehat{X}) \to \mathcal{G}^q (\widehat{X})$ fails to distinguish $Y$ and $Y_{\red}$, because both of them have the same formal neighborhood $\widehat{X}$.

\medskip

We get around these difficulties by generalizing the \v{C}ech hypercohomology machine of R. Thomason \cite[\S 1]{Thomason etale} via the structured \v{C}ech covers, called ``systems of local embeddings" $\mathcal{U}=\{ (U_i, X_i)\}_{i \in \Lambda}$ (see Definition \ref{defn:system}). This is a minor variant of the one first used by R. Hartshorne \cite{Hartshorne DR} in his studies of the algebraic de Rham cohomology of singular varieties. This $\mathcal{U}$ consists of a (quasi-)affine open cover $|\mathcal{U}|= \{ U_i \}_{i \in \Lambda}$ of $Y$ and closed immersions $U_i \hookrightarrow X_i$ into equidimensional smooth $k$-schemes. For each $I= (i_0, \cdots, i_p) \in \Lambda^{p+1}$, $p \geq 0$, let $U_I:= U_{i_0} \cap \cdots \cap U_{i_p}$ and $X_I:= X_{i_0} \times \cdots \times X_{i_p}$. Consider the diagonal closed immersion $U_I \hookrightarrow X_I$, and the completion $\widehat{X}_I$ of $X_I$ along $U_I$. 

The spaces $\mathcal{G}^q (\widehat{X}_I) $ over all $I\in \Lambda^{p+1}$ and $p \geq 0$ still do not tell the difference between $Y$ and $Y_{\red}$. Here, we devise a $K$-theoretic analogue of the mod equivalence for cycles developed in \cite{Park Tate} and \cite{P general}. Namely, we construct (see \S \ref{sec:mod Y spaces}) a space $\mathcal{S}^{q} (D_{\widehat{X}_I}, U_I, n)$, where $D_{\widehat{X}_I} = \widehat{X}_I \coprod_Y \widehat{X}_I$ with two morphisms
\begin{equation}\label{eqn:intro coeq 0}
\mathcal{S}^{q} (D_{\widehat{X}_I}, U_I, n) \rightrightarrows \mathcal{G}^q (\widehat{X}_I, n),
\end{equation}
arising from the derived Milnor patching of S. Landsburg \cite{Landsburg Duke}. 
Then take the homotopy coequalizer of \eqref{eqn:intro coeq 0} to define the space $\mathcal{G}^q (\widehat{X}_I, U_I, n)$, and take the geometric realization $\mathcal{G}^q (\widehat{X}_I, U_I):=|\un{n} \mapsto \mathcal{G} ^q (\widehat{X}_I, U_I, n)|$ of the cubical object. It can be seen as ``the $K$-space of $\widehat{X}_I$ mod $U_I$, with the coniveau $\geq q$."

Using the spaces $\mathcal{G}^q (\widehat{X}_I, U_I)$ over all $I\in \Lambda^{p+1}$ and $p \geq 0$, we want to apply a \v{C}ech-like machine. This requires a technical tool that is similar in spirit to moving lemmas for algebraic cycles, thus we also call it a moving lemma in this paper. The important case is the following, stated in Proposition \ref{prop:D ess} and Corollary \ref{cor:Gysin cnv}:

\begin{prop}\label{prop:moving intro}
Let $Y$ be a quasi-projective $k$-scheme. Suppose we have closed immersions $Y \hookrightarrow X_1 \hookrightarrow X_2$, where $X_1$ and $X_2$ are equidimensional smooth $k$-schemes. Let $\widehat{X}_i$ be the completion of $X_i$ along $Y$. 

Let $\mathcal{D}^q _{\coh, \{ \widehat{X}_1 \}} (\widehat{X}_2,n)$ be the triangulated subcategory of $\mathcal{D}^q _{\coh} (\widehat{X}_2,n)$ generated by coherent sheaves whose associated cycles intersect properly with $\widehat{X}_1$ (see Definition \ref{defn:cnv X1}).

Then the inclusion 
$$
\mathcal{D}^q _{\coh, \{ \widehat{X}_1 \}} (\widehat{X}_2,n) \hookrightarrow \mathcal{D}^q _{\coh} (\widehat{X}_2,n)
$$ is essentially surjective, and the induced map 
$$
\mathcal{G}^q _{\{\widehat{X}_1\}} (\widehat{X}_2,n) \to \mathcal{G}^q (\widehat{X}_2,n)
$$
of their Waldhausen $K$-spaces is a weak-equivalence.
\end{prop}

One important consequence of Proposition \ref{prop:moving intro} is the following, stated in Theorems \ref{thm:moving pull-back} and \ref{thm:moving pull-back mod Y}:

\begin{thm}\label{thm:moving pull-back intro}
Suppose we have a commutative diagram
$$
\xymatrix{
X_1 \ar[r] ^f & X_2 \\
Y_1 \ar[r]^g \ar@{^{(}->}[u] & Y_2, \ar@{^{(}->}[u] }
$$
where $Y_1, Y_2$ are affine $k$-schemes of finite type and the vertical maps are closed immersions into equidimensional smooth $k$-schemes. Let $\widehat{X}_i$ be the completion of $X_i$ along $Y_i$ for $i=1,2$. Then for $q, n \geq 0$, we have the induced morphisms
$$\tuborg \widehat{f}^*:\mathcal{G}^q (\widehat{X}_2,n) \to \mathcal{G}^q (\widehat{X}_1,n), \\
\widehat{f}^*:\mathcal{G}^q (\widehat{X}_2, Y_2, n) \to \mathcal{G}^q (\widehat{X}_1,Y_1, n)\sluttuborg
$$
in the homotopy category. In particular, after geometric realizations of the cubical spaces over $n \geq 0$, we have the induced morphisms in the homotopy category
$$
\tuborg \widehat{f}^*:\mathcal{G}^q (\widehat{X}_2) \to \mathcal{G}^q (\widehat{X}_1), \\
\widehat{f}^*:\mathcal{G}^q (\widehat{X}_2, Y_2) \to \mathcal{G}^q (\widehat{X}_1, Y_1).\sluttuborg
$$
\end{thm}

\medskip

Coming back to the generalizations of the \v{C}ech-machine, how does Theorem \ref{thm:moving pull-back intro} help? For a $k$-scheme $Y$ of finite type and a system $\mathcal{U}$ of local embeddings for $Y$, we construct a cosimplicial object in the homotopy category
\begin{equation}\label{eqn:g-hat intro}
\widehat{\mathcal{G}} ^{q, \bullet} _{\mathcal{U}} := \left \{ \prod_{i_0 \in \Lambda} \mathcal{G}^q (\widehat{X}_{i_0}, U_{i_0}) \overset{\leftarrow}{ \mathrel{\substack{\textstyle\rightarrow\\[-0.6ex]
                      \textstyle\rightarrow }}}   \prod_{I \in \Lambda^2} \mathcal{G}^q (\widehat{X}_I, U_I)
 \overset{\leftarrow}{ \mathrel{\substack{\textstyle\rightarrow\\[-0.6ex]
                      \textstyle\rightarrow \\[-0.6ex]
                      \textstyle\rightarrow}} } \prod_{I \in \Lambda^3} \mathcal{G}^q (\widehat{X}_I, U_I) \cdots \right \},
\end{equation}
where the cofaces and the codegeneracies are not solid morphisms of spaces: they are defined with helps of Theorem \ref{thm:moving pull-back intro} in the homotopy category in general. Taking the homotopy limit over $\Delta$, we define a version of \v{C}ech hypercohomology space over the system $\mathcal{U}$
$$
\check{\mathbb{H}}^{\cdot} (\mathcal{U}, \widehat{\mathcal{G}} ^q):= \underset{\Delta}{\holim} \ \widehat{\mathcal{G}}_{\mathcal{U}} ^{q, \bullet}.
$$

For systems $\mathcal{U}$ of local embeddings, we introduce a notion of refinements (see Definition \ref{defn:refine}). This notion is more flexible than the one considered by R. Hartshorne \cite{Hartshorne DR}. With helps of Theorem \ref{thm:moving pull-back intro} again, we check that for each refinement system $\mathcal{V}$ of $\mathcal{U}$, there exists the induced morphism in the homotopy category (see Lemma \ref{lem:refinement induce})
$$
\check{\mathbb{H}}^{\cdot} (\mathcal{U} , \widehat{\mathcal{G}} ^q) \to \check{\mathbb{H}}^{\cdot} (\mathcal{V} , \widehat{\mathcal{G}} ^q),
$$
while for any two systems $\mathcal{U}_1$ and $\mathcal{U}_2$ for $Y$, we have a common refinement (see Lemma \ref{lem:comref}). Hence we have the homotopy colimit
$$
\widehat{\mathcal{G}} ^q_Y:= \underset{\mathcal{U}}{\hocolim} \ \check{\mathbb{H}}^{\cdot} (\mathcal{U}, \widehat{\mathcal{G}}^q).
$$

\medskip

On the other hand, for each $\mathcal{U}$, the closed immersions $U_I \hookrightarrow X_I $ induce the closed immersions $U_I \hookrightarrow \widehat{X}_I$. We deduce morphisms in the homotopy category
$$
\mathcal{G}^q (\widehat{X}_I, U_I) \to \mathcal{G} (\widehat{X}_I, U_I)  \to \mathcal{K} (U_I),
$$
which in turn induces a morphism in the homotopy category
$$
\check{\mathbb{H}}^{\cdot} (\mathcal{U}, \widehat{\mathcal{G}} ^q) \to \check{\mathbb{H}}^{\cdot} (|\mathcal{U}|, \mathcal{K}),
$$
where the right hand side is the original \v{C}ech hypercohomology space of $K$-theory in the sense of R. Thomason \cite[\S 1]{Thomason etale}. This is weak-equivalent to the $K$-space $\mathcal{K} (Y)$ of $Y$ (see Remark \ref{remk:Cech Zariski descent}). Thus, after taking the homotopy colimits over $\mathcal{U}$, we deduce a tower of morphisms in the homotopy category
$$
\cdots \to \widehat{\mathcal{G}}_Y ^q \to \widehat{\mathcal{G}} _Y ^{q-1}  \to \cdots \to \widehat{\mathcal{G}}_Y^0 \to  \mathcal{K} (Y),
$$
Taking the images of the higher homotopy groups, we finally obtain the desired motivic coniveau filtration on $K_n (Y)$ of the main result, Theorem \ref{thm:main intro}. 

On the other hand, if we used $\mathcal{G}^q (\widehat{X}_I, U_I, 0)$ instead of $\mathcal{G}^q (\widehat{X}_I, U_I)$ in \eqref{eqn:g-hat intro} and repeated the above, then we obtain the coniveau filtration on $K_n (Y)$ of Theorem \ref{thm:main intro-0}. 

\medskip

We remark that, in the article \cite{P general} concurrently developed by the first named author, there is a new cycle class group denoted by $\BGH^q (Y,n)$ in bold face letters for each $Y$, possibly arbitrarily singular. This is in general distinct from the higher Chow groups $\CH^q (Y, n)$ of S. Bloch \cite{Bloch HC}, while they coincide when $Y$ is smooth. This new cycle theory also utilizes a \v{C}ech machine via the systems of local embeddings and the regular formal neighborhoods $\widehat{X}_I$ of $U_I$. We wonder whether we have a cycle class map
\begin{equation}\label{eqn:cycle class}
\BGH^q (Y,n) \to gr_{\mcnv} ^q K_n (Y)
\end{equation}
and whether this map is an isomorphism up to tensoring with $\mathbb{Q}$. Should this hold, we may regard it as the generalization of the Grothendieck-Riemann-Roch theorem for singular schemes.

\medskip

\textbf{Conventions.}
In this paper, a given base field $k$ is arbitrary. All noetherian formal schemes are assumed to have finite Krull dimensions unless said otherwise. We let $\Sch_k$ be the category separated $k$-schemes of finite type, maybe singular, even non-reduced.

\section{Some recollections}\label{sec:recollection}

In \S \ref{sec:recollection}, we recall some basic notions of noetherian formal schemes, and discuss some materials needed for our studies of algebraic $K$-theory of noetherian formal schemes, such as perfect complexes and pseudo-coherent complexes. Most of the materials are from SGA VI  \cite{SGA6} and EGA I \cite{EGA1}, but one new observation (Proposition \ref{prop:regularity0}) on regularity of formal schemes is made.

\subsection{Formal schemes}

We recall the notion of adic rings. The basic references are EGA I \cite[Ch 0, \S 7, p.60]{EGA1} and \cite[\S 10, p.180]{EGA1}. 

\subsubsection{Adic rings}
Recall (\cite[Ch 0, D\'efinition (7.1.2), p.60]{EGA1}) that a topological ring $A$ is said to be \emph{linearly topologized} if there is a fundamental system of neighborhoods of $0$ in $A$ given by ideals. In a linearly topologized ring $A$, we say that an ideal $I \subset A$ is an \emph{ideal of definition} if $I$ is open, and for each neighborhood $V$ of $0$, there is an integer $n>0$ such that $I^n \subset V$. If an ideal of definition does exist, then we say $A$ is \emph{preadmissible}. An \emph{admissible ring} $A$ is a preadmissible ring that is also separated and complete.

When $A$ is a noetherian admissible ring, there exists the largest ideal of definition $I_0 \subset A$, such that $A/ I_0$ is reduced. See \cite[Ch 0, Corollaires (7.1.6), (7.1.7), pp.61-62]{EGA1}. 

Recall (\cite[Ch 0, D\'efinition (7.1.9), p.62]{EGA1}) that a preadmissible ring $A$ is called \emph{preadic} if there is an ideal $I$ of definition, and the powers $I^n$ for $n>0$ form a fundamental system of neighborhoods of $0$. It is called \emph{adic}, if this preadic ring is separated and complete.

 Let $A$ be an admissible ring. Let $J \subset A$ be an ideal contained in an ideal of definition. The topology given by the fundamental system $J^n$ for $n>0$ of neighborhoods of $0$ is called the $J$-preadic topology. In this case $A$ is separated and complete with respect to the $J$-preadic topology. (See \cite[Ch 0, (7.2.3),  Proposition (7.2.4), p.63]{EGA1}.)

\subsubsection{Completed rings of fractions}

Recall from \cite[Ch 0, (7.6.1), (7.6.5), pp.72-72]{EGA1} the following. Let $A$ be a linearly topologized ring with a fundamental system of neighborhoods of $0$ given by ideals $\{ I_{\lambda}\}$. Let $S \subset A$ be a multiplicative subset. Let $u_{\lambda}: A \to A_{\lambda}:= A/ I_{\lambda}$ be the natural map. If $I_{\mu} \subset I_{\lambda}$, let $u_{\lambda \mu}: A_{\mu} \to A_{\lambda}$ be the natural map. Let $S_{\lambda}:= u_{\lambda} (S)$.

The maps $u_{\lambda \mu}$ then canonically induce surjective homomorphisms $S_{\mu} ^{-1} A_{\mu} \to S_{\lambda} ^{-1} A_{\lambda}$, and the data form a projective system. We define 
$$
A\{ S^{-1} \} := \varprojlim_{\lambda} S_{\lambda}^{-1} A_{\lambda},
$$
called the \emph{completed ring of fractions} of $A$ with the denominators in $S$.

There is another version of localization from \cite[Ch 0, (7.6.15), p.74]{EGA1}, that we recall. Let $A$ be a linearly topologized ring and let $f \in A$. Let $S_f:= \{ 1, f, f^2, \cdots \}$, which is multiplicative in $A$. Then we let
\begin{equation}\label{eqn:comp loc f}
A_{\{ f \}} := A \{ S_f ^{-1} \}.
\end{equation}

If $g \in A$ is another element, then we have a canonical continuous homomorphism $A_{ \{ f \}} \to A_{ \{ fg \}}$.  (See \cite[Ch 0, (7.6.7), p.73]{EGA1}.) So, when $S \subset A$ is a multiplicative subset, and $f$ runs over $S$, the system $A_{\{ f \}}$ gives a filtered inductive system of rings. We define
$$
A_{\{ S \}} := \varinjlim_{f \in S} A_{ \{ f \}}.
$$
Since we have a natural homomorphism $A_{\{ f \}} \to A \{ S ^{-1} \}$, this induces a natural flat (\cite[Ch 0, Proposition (7.6.16), p.75]{EGA1}) homomorphism $A_{ \{S \}} \to A\{ S^{-1} \}$. 
We recall the following:

\begin{prop}[{\cite[Ch 0, Prop. (7.6.17), Cor. (7.6.18), p.75]{EGA1}}] \label{prop:two localizations}
Let $A$ be an admissible ring and let $P \subset A$ be an open prime ideal. Let $S:= A \setminus P$. Then
\begin{enumerate}
\item The rings $A_{\{ S \}}$ and $A\{ S^{-1} \}$ are local rings, and the homomorphism $A_{ \{ S \}} \to A \{ S^{-1} \}$ is a flat local homomorphism.
\item The residue fields of both rings are canonically isomorphic to ${\rm Frac} (A/P)$.
\end{enumerate}
Furthermore, in case $A$ is a noetherian adic ring, then the above local rings are both noetherian, and the homomorphism $A_{ \{ S \}} \to A \{ S^{-1} \}$ is faithfully flat.
\end{prop}

\subsubsection{Affine formal schemes}
Recall that (\cite[D\'efinition (10.1.2), p.181]{EGA1}), for an admissible ring $A$, the formal spectrum $\mathfrak{X}=\Spf (A)$ of $A$ is the closed subspace of $\Spec (A)_{\rm Zar}$ given by the \emph{open} prime ideals of $A$, together with the structure sheaf $\mathcal{O}_{\mathfrak{X}}$ given by the projective limit of the sheaves $(\tilde{A}/ \tilde{ I}_{\lambda} )|_{\mathfrak{X}}$ over the fundamental system of neighborhoods $\{ I_{\lambda} \}$, so that $\mathfrak{X}$ is a ringed space. Here the tilde means the sheaves associated to the modules. We denote the underlying topological space by $|\mathfrak{X}|$.

An affine formal scheme is a ringed space that is isomorphic to an affine formal spectrum as ringed spaces.

The rings $A_{\{ f \}}$ of \eqref{eqn:comp loc f} give basic open subsets for affine formal schemes: Let $A$ be an admissible ring and let $\mathfrak{X} = \Spf (A)$. For $f \in A$, let $\mathfrak{D}(f) := D(f) \cap | \mathfrak{X}|$, where $D(f)= \{ P \in \Spec (A) \ | \ f \not \in P \}$. 
Then \cite[Proposition (10.1.4), p.181]{EGA1} shows that the induced ringed space $(\mathfrak{D}(f) , \mathcal{O}_{\mathfrak{X}}|_{\mathfrak{D} (f)})$ is isomorphic to the affine formal spectrum $\Spf (A_{\{ f \}})$.

\begin{defn}[{\cite[(10.1.5, p.182]{EGA1}}]\label{defn:formal stalk}
Let $A$ be an admissible ring and let $\mathfrak{X} = \Spf (A)$. Let $x \in |\mathfrak{X}|$. Define the stalk $\mathcal{O}_{\mathfrak{X}, x} $ at $x$ to be
$$
\mathcal{O}_{\mathfrak{X}, x} := A_{\{ S_x \}} = \varinjlim_{f \in S_x} A_{\{ f \}},
$$
where $S_x:= A \setminus P_x$ and $P_x \subset A$ is the prime ideal corresponding to the point $x$. Since $x \in |\mathfrak{X}|$, this $P$ is an open prime ideal of $A$, with respect to the topology of $A$. This stalk is indeed a local ring by Proposition \ref{prop:two localizations}-(1). \qed 
\end{defn}

\subsubsection{Regularity for affine formal schemes}

The following regularity result plays an important role in this article:

\begin{prop}\label{prop:regularity0}
Let $A$ be a noetherian adic ring and $\mathfrak{X}= \Spf (A)$. Let $x \in | \mathfrak{X} |$, which is a regular scheme point of $\Spec (A)$. Then the local ring $\mathcal{O}_{\mathfrak{X}, x}$ at $x$ of the affine formal scheme is a regular local ring.
\end{prop}

To prove it, we need a few results. First recall the flat descent of regularity:

\begin{lem}[{EGA ${\rm IV}_2$ \cite[Proposition (6.5.1)-(i), p.143]{EGA4-2}}]\label{lem:flat_reg}
Let $R \to S$ be a flat local homomorphism of noetherian local rings. Under this assumption, if $S$ is regular, then so is $R$.
\end{lem}

We also need the following:

\begin{lem}\label{lem:loc comp lem}
Let $A$ be a preadic ring, $S \subset A$ be a multiplicative subset, and $I \subset A$ be an ideal of definition. Then there is a canonical isomorphism of rings
$$A \{ S^{-1} \} \overset{\sim}{\to} \widehat{ S^{-1} A },$$
where the completion notation of the latter is the $S^{-1} I$-adic completion.
\end{lem}
\begin{proof}
When $\{ I_{\lambda} \}_{\lambda}$ is a fundamental system of neighborhoods of $0$, we have isomorphisms of rings
$$
A\{ S^{ -1} \} \overset{\sim}{\to} \varprojlim_{\lambda} S^{-1} A/ S^{-1} I_{\lambda} \simeq \varprojlim_n S^{-1} A / (S^{-1}I)^n = \widehat{S^{-1} A},$$
which is also a homeomorphism. 
\end{proof}

\bigskip

\begin{proof}[Proof of Proposition \ref{prop:regularity0}]

The point $x \in |\mathfrak{X}|$ gives a regular scheme point of $\Spec (A)$. Let $P_x$ be the prime ideal corresponding to $x$ (thus $P_x$ is an open prime ideal in $A$), and let $S_x := A \setminus P_x$. By Definition \ref{defn:formal stalk}, we have $\mathcal{O}_{\mathfrak{X}, x} = A_{ \{ S_x \}}$. 

By Proposition \ref{prop:two localizations}, we have a flat local homomorphism $A_{\{S_x \}} \to A \{ S_x^{-1} \}$ of noetherian local rings. So, by Lemma \ref{lem:flat_reg}, it is enough to show that $A\{ S_x^{-1} \}$ is a regular local ring.

By the given assumption, $S^{-1}_x A =: A_x$ is a regular local ring, with the unique maximal ideal given by $S_x^{-1} P_x$. Here, $S^{-1}_x I$ gives an ideal of definition of $S^{-1}_x A$, and it is contained in the maximal ideal $S^{-1}_x P_x$. 

Since $S^{-1}_x A$ is a regular local ring, so is the completion $\widehat{S^{-1}_x A}'$ of $S_x ^{-1} A$ by the maximal ideal $S^{-1}_x P_x$ (see H. Matsumura \cite[Theorem 19.5, p.157]{Matsumura}). On the other hand, the local homomorphism $\widehat{S^{-1}_x A} \to \widehat{S^{-1}_x A}'$ is flat (see EGA I \cite[Corollaire (10.8.9), p.197]{EGA1}), so by Lemma \ref{lem:flat_reg}, the regularity of $\widehat{S^{-1}_x A}'$ implies the regularity of $\widehat{S^{-1}_x A}$. The latter ring is isomorphic to $A \{ S^{-1}_x \}$ by Lemma \ref{lem:loc comp lem}, so, $A\{ S_x ^{-1} \}$ is a regular local ring. This completes the proof.
\end{proof}

\subsubsection{Formal schemes}

We recall some basic definitions around noetherian formal schemes from EGA I \cite[(10.4.2), p.185]{EGA1}. A noetherian formal scheme is a locally ringed space such that its underlying topological space is quasi-compact, and each point has an open neighborhood that is an affine open formal spectrum given by a noetherian adic ring. 
All formal schemes $\mathfrak{X}$ will be assumed to be noetherian, unless said otherwise.

For a noetherian formal scheme $\mathfrak{X}$, the \emph{dimension} of $\mathfrak{X}$ is defined to be the supremum of the Krull dimensions of the local rings $\mathcal{O}_{\mathfrak{X}, x}$ over all $x \in |\mathfrak{X}|$. We have $\dim \ \mathfrak{X} \geq \dim \ |\mathfrak{X}|$, where the latter is the dimension of the noetherian topological space $|\mathfrak{X}|$. This becomes an equality if $\mathfrak{X}$ is a scheme, but otherwise it is a strict inequality in general.
Each open subset $U \subset | \mathfrak{X}|$ gives an open formal subscheme $(U, \mathcal{O}_{\mathfrak{X}} | _U)$, which we often denote by $\mathfrak{X}|_U$. 

We say that $\mathfrak{X}$ is \emph{equidimensional} if the dimensions of $\mathfrak{X}|_U$ are uniform over all nonempty affine open subsets $U \subset | \mathfrak{X}|$. We say $\mathfrak{X}$ is \emph{integral}, if every nonempty affine open formal subscheme $U \subset \mathfrak{X}$ is given by $U= \Spf (A)$ for some integral domain $A$.

A \emph{closed formal subscheme} $\mathfrak{Y}$ of $\mathfrak{X}$ is a ringed space $ (|\mathfrak{Y}|, \mathcal{O}_{\mathfrak{Y}})$ given by an ideal sheaf $\mathcal{I} \subset \mathcal{O}_{\mathfrak{X}}$; namely, we have $\mathcal{O}_{\mathfrak{Y}} = \mathcal{O}_{\mathfrak{X}}/ \mathcal{I}$ and $|\mathfrak{Y}|={\rm Supp} (\mathcal{O}_{\mathfrak{Y}})= | \mathcal{O}_{\mathfrak{Y}}|$. We often informally write that $\mathfrak{Y} \subset \mathfrak{X}$ is a closed formal subscheme. (See EGA I \cite[D\'efinition I-(10.14.2), p.210]{EGA1}) 

When $\mathfrak{X}$ is an equidimensional and $\mathfrak{Z} \subset \mathfrak{X}$ is a closed nonempty formal subscheme, we define the codimension to be
$$
{\rm codim}_{\mathfrak{X}} \mathfrak{Z} := \dim \ \mathfrak{X} - \dim \ \mathfrak{Z}.
$$
As in the case of schemes, when $\mathfrak{Z}_1, \mathfrak{Z}_2 \subset \mathfrak{X}$ are two closed formal subschemes, with their ideal sheaves $\mathcal{I}_{\mathfrak{Z}_1}, \mathcal{I}_{\mathfrak{Z}_2} \subset \mathcal{O}_{\mathfrak{X}}$, respectively, the formal scheme theoretic intersection $\mathfrak{Z}_1 \cap \mathfrak{Z}_2$ is defined to be the closed formal subscheme of $\mathfrak{X}$ associated to the sum $\mathcal{I}_{\mathfrak{Z}_1} + \mathcal{I}_{\mathfrak{Z}_2}$. In case the intersection $\mathfrak{Z}_1 \cap \mathfrak{Z}_2$ is either empty, or nonempty and its codimension in $\mathfrak{X}$ is greater than or equal to the sum of the codimensions of $\mathfrak{Z}_1, \mathfrak{Z}_2$, we say that the intersection is proper, or that $\mathfrak{Z}_1$ intersects properly with $\mathfrak{Z}_2$.

We say $\mathfrak{X}$ is \emph{regular}, if for each point $x \in |\mathfrak{X}|$, the local ring $\mathcal{O}_{\mathfrak{X}, x}$, as in Definition \ref{defn:formal stalk}, is a regular local ring. The notion of regular formal scheme is important in this paper. 

\subsection{Perfect complexes and pseudo-coherent complexes}

Recall the notion of perfect complexes on ringed topoi, especially on noetherian formal schemes. 
(See e.g. \cite[\S 2.4, 2.5]{Leo AT} or more generally SGA VI \cite[Expos\'e I, \S 4, p.119]{SGA6} for ringed topoi.)

For noetherian formal schemes, the perfect complexes can be described more precisely as follows: Let $\mathfrak{X}$ be a noetherian formal scheme. Let $\mathbf{A}(\mathfrak{X})$ be the category of all $\mathcal{O}_{\mathfrak{X}}$-modules. Let $\mathcal{D} (\mathfrak{X})$ be the derived category of $\mathbf{A}(\mathfrak{X})$.

A complex $\mathcal{E} \in \mathcal{D} (\mathfrak{X})$ is called a \emph{perfect complex} if each point $x \in | \mathfrak{X}|$ has an open neighborhood $U \subset |\mathfrak{X}|$ (which gives the open formal subscheme $\mathfrak{X}|_U$) and a bounded complex $\mathcal{F}$ of locally free finite type $\mathcal{O}_{\mathfrak{X}|_{U}}$-modules together with an isomorphism $\mathcal{F} \overset{\sim}{\to} \mathcal{E}|_{U}$ in $\mathcal{D}(\mathfrak{X}|_{U})$. We let $\mathcal{D}_{\perf} (\mathfrak{X})$ be the triangulated subcategory of $\mathcal{D}(\mathfrak{X})$ of perfect complexes on $\mathfrak{X}$.

A bit more generally, a complex $\mathcal{E}$ is called a \emph{pseudo-coherent complex} if every point $x \in |\mathfrak{X}|$ has an open neighborhood $U$ over which $\mathcal{E}|_U$ is quasi-isomorphic to a bounded above complex of locally free finite type $\mathcal{O}_{\mathfrak{X}|_U}$-modules. We let $\mathcal{D}_{\coh} (\mathfrak{X})$ be the triangulated subcategory of $\mathcal{D} (\mathfrak{X})$ of pseudo-coherent complexes on $\mathfrak{X}$ with bounded cohomologies.

\bigskip

An important result we need is the following: 

\begin{lem}\label{lem:coh=perfect}
Let $\mathfrak{X}$ be a noetherian formal scheme such that for each point $x \in |\mathfrak{X}|$, the stalk $\mathcal{O}_{\mathfrak{X}, x}$ is a regular local ring. 

Then the functor $\mathcal{D}_{\perf} (\mathfrak{X}) \to \mathcal{D}_{\coh} (\mathfrak{X})$ from the perfect complexes to pseudo-coherent complexes with bounded cohomologies is an equivalence.

 In particular, the natural induced homomorphism
$K_i (\mathfrak{X})  \to G_i (\mathfrak{X})$
is an isomorphism. 
\end{lem}

\begin{proof}
This is an immediate consequence of SGA VI \cite[Expos\'e IV, \S 2.5, pp.280-281]{SGA6} and \cite[Expos\'e I, Corollaire 5.10, p.138]{SGA6}; here there are two things required to check by \cite[Expos\'e IV, \S 2.5, pp.280-281]{SGA6}. 

First of all, for the requirement of having sufficiently many points in the sense of SGA ${\rm IV}_1$ \cite[Expos\'e IV, D\'efinition 6.4.1, p.389]{SGA4-1}, this is satisfied by P. Deligne's completeness theorem SGA ${\rm IV}_2$ \cite[Expos\'e VI, Proposition (9.0), p.336]{SGA4-2} because all (formal) schemes have locally coherent underlying topological spaces. (A kind modern reference would be Fujiwara-Kato \cite[\S 2.2-2.7]{FK}.) Secondly, the finite Tor dimension requirement is automatic by Auslander-Buchsbaum \cite{AB} because $\mathcal{O}_{\mathfrak{X}, x}$ is a regular local ring.

The last follows from the constructions of the Waldhausen $K$-spaces of the derived categories (see Thomason-Trobaugh \cite{TT}, Waldhausen \cite{Waldhausen}).
\end{proof}

\subsection{Examples of embeddable schemes}

The technique of embedding a scheme $Y \hookrightarrow X$ was used by Grothendieck at a few places in the literature, e.g. in SGA I \cite{SGA1}, in SGA II \cite{SGA2}, as well as in R. Hartshorne \cite{Hartshorne DR}.

This is also widely used in this article. Here is one large class of examples:

\begin{lem}\label{lem:exoskeleton}
Let $Y$ be a quasi-projective $k$-scheme. Then $Y$ is embeddable, in that there is a closed immersion $Y \hookrightarrow X$ into a regular scheme. Furthermore, we can find a closed immersion $Y \hookrightarrow X$ into a regular noetherian scheme $X$ that is smooth over $k$.

If desired, one can choose $X$ to be equidimensional.

Given an embedding $Y \hookrightarrow X$, let $\widehat{X}$ be the completion of $X$ along $Y$. Then $\widehat{X}$ is a regular noetherian formal $k$-scheme. If $X$ is equidimensional, then so is $\widehat{X}$.
\end{lem}

\begin{proof}
That there is a closed immersion $Y \hookrightarrow X$ for a regular scheme $X$ that is smooth and quasi-projective over $k$, is apparent because $Y$ is quasi-projective over $k$.

If $X$ is not equidimensional, being quasi-compact and smooth over $k$, it has at most finitely many smooth connected components of possibly various dimensions. We can then find a big enough projective space $\mathbb{P}_k ^N$ in which all connected components of $X$ are closed subschemes of an open subset $U$ of $\mathbb{P}_k ^N$. Such $U$ is smooth over $k$, and it gives a closed immersion $Y \hookrightarrow X \hookrightarrow U$. Replacing $X$ by $U$, we may assume $X$ is equidimensional. 

By taking the completion, the formal scheme $\widehat{X}$ is an equidimensional (EGA ${\rm IV}_2$ \cite[Corollaire (7.1.5), p.184]{EGA4-2}) noetherian formal $k$-scheme. For the regularity, we need to see that for each $x \in |\widehat{X}|$, the local ring $\mathcal{O}_{\widehat{X}, x}$ is regular. This follows by first covering $\widehat{X}$ by open affine formal subschemes, and then applying Proposition \ref{prop:regularity0}.
\end{proof}

\subsection{On cycles and associated cycles}

We recall a few definitions and facts on cycles on affine formal schemes from \cite{Park Tate}. 

\begin{defn}\label{defn:CG}
Let $\mathfrak{X}= \Spf (A)$ be an equidimensional noetherian affine formal scheme of finite Krull dimension for an equidimensional ring $A$. 

\begin{enumerate} 
\item Let $\un{z}_d (\mathfrak{X})$ be the free abelian group on the set of integral closed formal subschemes of $\mathfrak{X}$ of dimension $d$. This is equal to the classical group $z_d (\Spec (A))$ of $d$-dimensional cycles by EGA ${\rm III}_1$ \cite[Corollaire (5.1.8), p.495]{EGA3-1}. 

This \emph{naive group $\un{z}_d (\mathfrak{X})$ of cycles} may contain some undesirable cycles that have poor behaviors with respect to ideals of definition of $\mathfrak{X}$. So, we consider the following subgroup.

\item Let $z_d (\mathfrak{X})$ be the subgroup of $\un{z}_d (\mathfrak{X})$ generated by the integral closed formal subschemes that intersect properly with the subscheme $\mathfrak{X}_{\red}$ defined by the largest ideal of definition of $\mathfrak{X}$. The largest ideal of definition exists by EGA I \cite[Proposition (10.5.4), p.187]{EGA1}. 

\item If $d_{\mathfrak{X}}$ is the dimension of $\mathfrak{X}$, and $0 \leq q \leq d_{ \mathfrak{X}}$, then we define the group of codimension $q$ cycles by $z^q (\mathfrak{X}) := z_{d_{\mathfrak{X}} - q} (\mathfrak{X})$. 

\item Let $\mathfrak{X}' \subset \mathfrak{X}$ be a fixed closed formal subscheme. The subgroup $z^q _{\{ \mathfrak{X}'\}} (\mathfrak{X})$ of $z^q (\mathfrak{X})$ is generated by the integral closed formal subschemes in $z^q (\mathfrak{X})$ that intersect properly with $\mathfrak{X}'$. 
\qed
\end{enumerate}
\end{defn}

Since the naive group of cycles on a noetherian affine formal scheme $\mathfrak{X}= \Spf (A)$ is essentially the group of cycles on the scheme $\Spec (A)$, for a coherent sheaf $\mathcal{F}$ on $\mathfrak{X}$, we have its associated cycle $[\mathcal{F}] \in \un{z}_* (\mathfrak{X})$. 

\begin{remk}
When $\mathfrak{X}$ is not (quasi-)affine, we do not know whether we can define the associated cycle $[\mathcal{F}]$ on $\mathfrak{X}$. But this situation does not appear in this article. 
\qed
\end{remk}

Recall from \cite{Park Tate}, \cite{P general} the following version of higher Chow cycles on noetherian affine formal schemes. While \cite{P general} discusses more generally cycles on quasi-affine formal schemes, we recall only those relevant to this article.

\begin{defn}\label{defn:HCG}
Let $\mathfrak{X}= \Spf (A)$ be an equidimensional noetherian affine formal $k$-scheme of finite Krull dimension.

Let $\square:= \mathbb{P}^1 \setminus \{ 1 \}$. Consider the $k$-rational points $\{ 0, \infty \} \subset \square$. Let $\square^n$ be the $n$-fold self fiber-product of $\square$ over $k$. A face $F \subset \square^n$ is a closed subscheme defined by a finite set of equations of the form
$$\{ y_{i_1} = \epsilon_1, \cdots, y_{i_s} = \epsilon_s \},$$
where $1 \leq i_1 < \cdots< i_s \leq n$ is an increasing sequence of indices and $\epsilon_j \in \{ 0, \infty \}$. We allow the set to be empty, in which case $F=\square^n$.

Consider the fiber product $\mathfrak{X} \times_k \square^n$
in the category of formal schemes, which exists by EGA I \cite[Proposition (10.7.3), p.193]{EGA1}. It is equidimensional by Greco-Salmon \cite[Theorem 7.6-(b), p.35]{GS}.  In what follows we will simply write
$\mathfrak{X} \times \square^n$.

Let $\mathcal{I}_0 \subset \mathcal{O}_{\mathfrak{X}}$ be the largest ideal of definition of $\mathfrak{X}$ (EGA I \cite[Proposition (10.5.4), p.187]{EGA1}), and let $\mathfrak{X}_{\red}$ be the closed subscheme defined by the ideal.

For integers $n, q \geq 0$, let $\un{z}^q (\mathfrak{X}, n)$ be the free abelian group on the set of integral closed formal subschemes $\mathfrak{Z} \subset \mathfrak{X} \times \square^n$ of codimension $q$, subject to the following conditions:
\begin{enumerate}
\item [(\textbf{GP})] (General position) The cycle $\mathfrak{Z}$ intersects properly with $\mathfrak{X} \times F$ for each face $F \subset \square^n$.

\item [(\textbf{SF})] (Special fiber) For each face $F \subset \square^n$, we have
$$
{\rm codim}_{\mathfrak{X}_{\red} \times F} ( \mathfrak{Z} \cap (\mathfrak{X}_{\red} \times F)) \geq q.
$$
\end{enumerate}

The cycles in $\un{z}^q (\widehat{X}, n)$ are called \emph{admissible} for simplicity. We let 
$$z^q (\mathfrak{X}, n):= \frac{ \un{z} ^q (\widehat{X}, n)}{\un{z} ^q (\widehat{X}, n)_{\rm degn}},$$
which is the group of non-degenerate cycles. 
\qed
\end{defn}

For the codimension $1$ face maps $\iota_i ^{\epsilon}: \mathfrak{X} \times \square^{n-1} \hookrightarrow \mathfrak{X} \times \square^n$ given by $\{y_i = \epsilon \}$, where $\epsilon \in \{ 0, \infty\}$, we have seen in \cite{Park Tate} (or \cite{P general}) that it induces the face map $\partial_i ^{\epsilon} = (\iota_i ^{\epsilon})^*: z^q (\mathfrak{X}, n) \to z^q (\mathfrak{X}, n-1)$. For $\partial:= \sum_{i=1} ^n (-1)^i (\partial_i ^{\infty} - \partial_i ^0)$, one checks that $\partial \circ \partial = 0$ and it defines a complex $(z^q (\mathfrak{X}, \bullet), \partial)$.

\section{A moving lemma for some sheaves on formal schemes}\label{sec:cnv 1}

The central goal in \S \ref{sec:cnv 1} is to prove a moving lemma 
in
Proposition \ref{prop:D ess} on the level of triangulated subcategories of certain pseudo-coherent complexes over the regular formal schemes of the form $\widehat{X} \times \square^n$, and to deduce a general pull-back in Theorem \ref{thm:moving pull-back}. They will be fundamental tools in \S \ref{sec:mod Y spaces} and \S \ref{sec:Cech}.

\subsection{A tower for formal schemes}\label{sec:cnv 0}

\begin{defn}\label{defn:codim n sheaf}
Let $\mathfrak{X}$ be an equidimensional noetherian formal scheme of finite Krull dimension. We say that a coherent $\mathcal{O}_{\mathfrak{X}}$-module $\mathcal{F}$ on $\mathfrak{X}$ is \emph{of codimension $q$} if there is a finite affine open cover $\mathcal{U}$ of $|\mathfrak{X}|$ such that for each $U \in \mathcal{U}$, we have $[\mathcal{F}|_U] \in z^q (\mathfrak{X}|_U)$ for the cycle group in Definition \ref{defn:CG}-(2),(3).
\qed
\end{defn}

\begin{lem}\label{lem:codim indep}
The notion of codimension in Definition \ref{defn:codim n sheaf} is independent of the choice of the cover.
\end{lem}

\begin{proof}
We may assume $|\mathfrak{X}|$ is connected. 

When $\mathcal{F}= 0$, there is nothing to prove. So, suppose $\mathcal{F}\not = 0$ in what follows.

\medskip

Let $\mathcal{U}$ and $\mathcal{V}$ be two finite affine open covers of $|\mathfrak{X}|$ for which $\mathcal{F}$ is of codimension $q$ and $q'$ with respect to $\mathcal{U}$ and $\mathcal{V}$, respectively, in the sense of Definition \ref{defn:codim n sheaf}. Since $\mathcal{F} \not = 0$, there is some $U_0 \in \mathcal{U}$ such that $\mathcal{F}|_{U_0} \not = 0$.

\medskip

Toward contradiction, suppose $q \not = q'$. For any $V \in \mathcal{V}$ such that $U_0 \cap V \not = \emptyset$, we are given that 
$$
[ \mathcal{F}|_{U_0} ] \in z^q (\mathfrak{X}|_{U_0}), \ \ \mbox{ and } \ \ [\mathcal{F}|_{V} ] \in z^{q'} (\mathfrak{X}|_V).
$$

Then by restriction, we have $[ \mathcal{F}|_{U_0 \cap V}] \in z^q (\mathfrak{X}|_{U_0 \cap V})$ and $[ \mathcal{F}|_{U_0 \cap V}] \in z^{q'} (\mathfrak{X}|_{U_0 \cap V})$ at the same time. Since $q \not = q'$, this implies that $[\mathcal{F}|_{U_0 \cap V } ] = 0$, which is possible only when $\mathcal{F}_{U_0 \cap V} = 0$.

Since the above holds for all $V \in \mathcal{V}$ such that $U_0 \cap V \not = \emptyset$, and $\mathcal{V}$ is an open cover of $|\mathfrak{X}|$, this holds only when $\mathcal{F}|_{U_0} = 0$. This is a contradiction.
\end{proof}

\begin{defn}\label{defn:cnv} 
Let $\mathfrak{X}$ be an equidimensional noetherian formal scheme of finite Krull dimension. 

Let $\mathcal{D}_{\coh}^q (\mathfrak{X}) \subset \mathcal{D}_{\coh} (\mathfrak{X})$ be the triangulated subcategory generated by coherent $\mathcal{O}_{\mathfrak{X}}$-modules $\mathcal{F}$ of codimension $\geq q$ in the sense of Definition \ref{defn:codim n sheaf}. Over all $ q \geq 0$, they form a decreasing sequence of triangulated subcategories of $\mathcal{D}_{\coh} (\mathfrak{X})$.
\qed
\end{defn}

\begin{remk}\label{remk:F0}
Here the inclusion $\mathcal{D}_{\coh} ^0 (\mathfrak{X}) \hookrightarrow \mathcal{D}_{\coh} (\mathfrak{X})$ may not be essentially surjective in general: the category $\mathcal{D}_{\coh} ^0 (\mathfrak{X})$ of Definition \ref{defn:cnv} is generated by coherent sheaves whose associated cycles have the proper intersection condition with $\mathfrak{X}_{\red}$, while the category $\mathcal{D}_{\coh} (\mathfrak{X})$ has no such a requirement.

However, in the special case when $\mathfrak{X}_{\red}=\mathfrak{X}$, clearly $\mathcal{D}_{\coh} ^0 (\mathfrak{X}) = \mathcal{D}_{\coh} (\mathfrak{X})$.  This situation occurs when $\mathfrak{X}= \widehat{X}$, where $Y$ is a smooth equidimensional scheme and $X=Y$ in Lemma \ref{lem:exoskeleton}. In particular, $\mathfrak{X} = \widehat{X} = X = Y$ in this case.
 \qed
\end{remk}

We need the following higher level version of Definition \ref{defn:cnv}. This is analogous to the cubical version of higher Chow cycles on schemes (see \cite{Bloch HC}) and such cycles on formal schemes from \cite{Park Tate}, \cite{P general}, recalled in Definition \ref{defn:HCG}.

\begin{defn}\label{defn:cnv higher}
Let $\mathfrak{X}$ be an equidimensional noetherian formal $k$-scheme of finite Krull dimension.

We say that a coherent $\mathcal{O}_{\mathfrak{X} \times \square^n}$-module $\mathcal{F}$ on $\mathfrak{X} \times  \square^n$ is \emph{admissible of codimension $ \geq q$} if there is a finite affine open cover $\mathcal{U}$ of $|\mathfrak{X}|$ such that for each $U \in \mathcal{U}$, we have $[\mathcal{F}|_{\mathfrak{X}|_U \times \square^n}] \in z^{\geq q} (\mathfrak{X}|_U, n)$ for the higher cycle group in Definition \ref{defn:HCG}. Using an argument similar to that in Lemma \ref{lem:codim indep}, one checks that this codimension is independent of $\mathcal{U}$.

Let $\mathcal{D}_{\coh} ^q (\mathfrak{X}, n) \subset \mathcal{D}_{\coh} (\mathfrak{X} \times  \square^n)$ be the full
triangulated subcategory generated by coherent $\mathcal{O}_{\mathfrak{X} \times \square^n}$-modules $\mathcal{F}$ that are admissible of codimension $\geq q$ in the above sense. 
\qed
\end{defn}

\begin{remk} \label{D0-n equal}
When $n=0$, for all $q \geq 0$, we have the equalities $\mathcal{D}_{\coh} ^q (\mathfrak{X},0)= \mathcal{D}_{\coh} ^q (\mathfrak{X})$ of the categories in Definitions \ref{defn:cnv} and \ref{defn:cnv higher} in this case. \qed
\end{remk}

\begin{remk} \label{cubicalfaces}
For the closed immersion of a codimension $1$ face $\iota_i ^{\epsilon}: F= \square ^{n-1} \hookrightarrow \square^n$ given by the equation $\{ y_i = \epsilon\}$ for $1 \leq i \leq n$ and $\epsilon \in \{0, \infty\}$, we have the induced derived functor
$$
\mathbf{L} (\iota_i ^{\epsilon})^*: \mathcal{D}_{\coh} ^q (\mathfrak{X}, n) \to \mathcal{D}_{\coh} ^{q} (\mathfrak{X}, n-1).
$$
Note that we have the free resolution
$$
0 \to \mathcal{O}_{\mathfrak{X} \times \square^n} \overset{ \times (y_i - \epsilon)}{\longrightarrow} \mathcal{O}_{\mathfrak{X} \times \square^n} \longrightarrow \mathcal{O}_{\mathfrak{X} \times F} \to 0.
$$
Using it together with the proper intersection with faces, one checks immediately that for each coherent sheaf $\mathcal{F} \in \mathcal{D}_{\coh} ^q (\mathfrak{X}, n)$, we have ${\rm Tor}_j ^{\mathcal{O}_{\mathfrak{X} \times \square^n}} (\mathcal{O}_{\mathfrak{X} \times F}, \mathcal{F}) = 0$ for $j > 0$. Hence the derived functor $\mathbf{L} (\iota_i ^{\epsilon})^*$ at $\mathcal{F}$ is in fact given just by the usual pull-back $(\iota_i ^{\epsilon})^*$ of $\mathcal{F}$.
\qed
\end{remk}

Together with the apparent degeneracy functors, the above face functors define the cubical triangulated category
$$
\mathcal{D}_{\coh} ^q (\mathfrak{X}, \bullet):= (\un{n}\mapsto \mathcal{D}_{\coh}^q (\mathfrak{X}, n)).
$$
Over $q \geq 0$, they give morphisms of cubical triangulated categories
$$
\mathcal{D}_{\coh}^{q+1} (\mathfrak{X}, \bullet) \hookrightarrow \mathcal{D}_{\coh}^{q} (\mathfrak{X}, \bullet) \hookrightarrow \cdots \hookrightarrow \mathcal{D}_{\coh}^{0} (\mathfrak{X}, \bullet) \hookrightarrow \mathcal{D}_{\coh} (\mathfrak{X} \times \square^{\bullet}).
$$

For each $n\geq 0$, we have the Waldhausen $K$-spaces
\begin{equation}\label{eqn:318-1}
\mathcal{G}^q (\mathfrak{X}, n):= \mathcal{K} (\mathcal{D}_{\rm coh} ^q (\mathfrak{X}, n))
\end{equation}
(see Thomason-Trobaugh \cite[\S 1]{TT} and Waldhausen \cite{Waldhausen}), that form the tower of the Waldhausen $K$-spaces 
\begin{equation} \label{eqn:higher n-tower}
\cdots \to \mathcal{G}^{q+1} (\mathfrak{X}, n) \to  \mathcal{G}^q (\mathfrak{X}, n) \to \cdots \to \mathcal{G}^0 (\mathfrak{X}, n) \to \mathcal{G} (\mathfrak{X} \times \square^n),
\end{equation}
which induces the tower of the cubical spaces
\begin{equation}\label{eqn:higher tower}
\cdots \to \mathcal{G}^{q+1} (\mathfrak{X}, \bullet) \to  \mathcal{G}^q (\mathfrak{X}, \bullet) \to \cdots \to \mathcal{G}^0 (\mathfrak{X}, \bullet) \to \mathcal{G} (\mathfrak{X} \times \square^{\bullet}).
\end{equation}

\begin{defn}\label{defn:g geometric}
For a noetherian equidimensional formal scheme $\mathfrak{X}$, we define
$$\mathcal{G} (\mathfrak{X}) := \mathcal{K} (\mathcal{D}_{\coh} (\mathfrak{X})).$$

For $q \geq 0$, define 
$$
\mathcal{G}^q (\mathfrak{X}):= | \un{n} \mapsto \mathcal{G}^q (\mathfrak{X}, n) |,
$$
$$
\mathcal{G} ^{\square} (\mathfrak{X}):=  | \un{n} \mapsto \mathcal{G} (\mathfrak{X} \times \square^n)|,
$$
the geometric realizations (see Remark \ref{remk:geometric realize} below) of the cubical spaces.\qed
\end{defn}
Thus we deduce the tower of spaces
$$
 \cdots \to \mathcal{G}^{q+1} (\mathfrak{X}) \to \mathcal{G}^q (\mathfrak{X}) \to \cdots \to \mathcal{G}^0 (\mathfrak{X}) \to \mathcal{G} ^{\square} (\mathfrak{X}).
$$

\medskip

\begin{remk}\label{remk:geometric realize}
We briefly mention that the geometric realization of a cubical space is defined as follows: recall that a cubical space (for us, a space is a spectrum) is a functor $\textbf{Cube} ^{\op} \to \Spt$. The category of cubical spaces is denoted by $\textbf{Cube}^{\op} \Spt$. 

Consider the $1$-simplex $\Delta ^1$, and for $n \geq 1$, let $(\Delta ^1)^n:= \Delta ^1 \times \cdots \times \Delta ^1$. When $n=0$, let $(\Delta ^1)^0$ be 
the zero simplex $\Delta ^0$. For the simplicial set 
$(\Delta ^1)^n_+:= (\Delta ^1)^n \coprod \Delta ^0$, consider the infinite suspension $n \mapsto C_n:=\Sigma^{\infty} (\Delta ^1)^n_+$. This gives a co-cubical spectrum denoted by 
$$C_-: \textbf{Cube} \to \Spt, \ \ n \mapsto C_n.$$

This induces the enriched Yoneda embedding
$$h : \textbf{Cube} \to \textbf{Cube}^{\op} \Spt$$
$$ n \mapsto ( h_n: \textbf{Cube}^{\op} \to \Spt)$$
where $h_n (m) = \Hom_{\Spt} (C_m, C_n).$

The geometric realization $| \alpha|$ of a cubical space $\alpha$ is then defined as the object evaluated with the left Kan extension $| - |$ of $C_{-}$ via $h$
$$
\xymatrix{ 
\textbf{Cube} \ar[d] ^h \ar[r]^{C_{-}} & \Spt \\
\textbf{Cube}^{\op} \Spt. \ar@{-->}[ru] _{|\cdot |} &}
$$
More specifically, 
$$| \alpha|=  \underset{ h_n \downarrow \alpha}{\rm colim}\  \Sigma^{\infty} (\Delta ^1)^n_+,
$$
which is the colimit over the comma category.
\qed
\end{remk}

\begin{lem}\label{lem:homotopy restricted}
For each $n \geq 0$, the morphism $\mathcal{G} (\widehat{X}) \to \mathcal{G} (\widehat{X} \times \square^n)$ induced by the projection $\widehat{X} \times \square^n \to \widehat{X}$ is a weak-equivalence.

In particular, the diagonal morphism $\mathcal{G} (\widehat{X}) \to \mathcal{G}^{\square} (\widehat{X})$ to the geometric realization is also a weak-equivalence.
\end{lem}
\begin{proof}
Under the localization theorem for $G$-theory, it is enough to consider the affine case $\widehat{X} = \Spf (A)$. Here, the projection corresponds to the injective homomorphism $A \hookrightarrow A \{ y_1, \cdots, y_n \}$ into the restricted formal power series ring. Here, the induced map of the $K$-spaces is a weak-equivalence by D. Quillen \cite[Theorem 7 of \S 6, p.120]{Quillen}.
\end{proof}

\subsection{A moving lemma: the statement}\label{sec:moving state}

Let $Y$ be a quasi-projective scheme. In \S \ref{sec:moving state}, we consider the special circumstance when there is a sequence of closed immersions $Y \hookrightarrow X_1 \hookrightarrow X_2$, where $X_1, X_2$ are equidimensional smooth $k$-schemes. They induce closed immersions $Y \hookrightarrow \widehat{X}_1 \overset{\widehat{\iota}}{\hookrightarrow} \widehat{X}_2$ of formal schemes, where $\widehat{X}_i$ is the completion of $X_i$ along $Y$.

\medskip

The following Definition \ref{defn:cnv X1} is similar to Definition \ref{defn:codim n sheaf}, but with an additional condition of proper intersections with respect to $\widehat{X}_1$ (cf. Definition \ref{defn:CG}-(4)):

\begin{defn}\label{defn:cnv X1}
Let 
\begin{equation}\label{eqn:X1 subcat}
\mathcal{D}_{\coh, \{ \widehat{X}_1\}} ^q (\widehat{X}_2, n)  \subset \mathcal{D}_{\coh} ^q (\widehat{X}_2, n)
\end{equation}
be the triangulated subcategory generated by all coherent $\mathcal{O}_{\widehat{X}_2 \times \square^n}$-modules $\mathcal{F}$ of codimension $\geq q$ intersecting properly with $\widehat{X}_1 \times F$ for each face $F \subset \square^n$. More precisely, there is an affine open cover $\mathcal{U}$ of $|\widehat{X}_2|=|Y|$ such that for each $U \in \mathcal{U}$, we have $[ \mathcal{F}|_{\widehat{X}_2|_U \times \square^n}] \in z^{ \geq q} _{\{ \widehat{X}_1|_U \}} (\widehat{X}_2|_U, n)$. As we did in Lemma \ref{lem:codim indep}, one checks that this notion is independent of the choice of $\mathcal{U}$. \qed
\end{defn}

Note that for any $\mathcal{F} \in \mathcal{D}_{\coh,  \{ \widehat{X}_1 \}} ^q (\widehat{X}_2,n)$, we have the pull-back $\mathbf{L}\widehat{\iota} ^* \mathcal{F} \in \mathcal{D}_{\coh} ^q (\widehat{X}_1,n )$ via the closed immersion $\widehat{\iota}: \widehat{X}_1 \hookrightarrow \widehat{X}_2$ (cf. \cite[Proposition 5.1.1]{P general}), and it induces the ``Gysin pull-back" functor
\begin{equation}\label{eqn:Gysin}
\mathbf{L}\widehat{\iota}^*: \mathcal{D}_{\coh,  \{ \widehat{X}_1 \}} ^q (\widehat{X}_2,n) \to  \mathcal{D}_{\coh} ^q (\widehat{X}_1,n).
\end{equation}
This \eqref{eqn:Gysin} gives the induced morphism of the Waldhausen $K$-spaces
\begin{equation}\label{eqn:Gysin2}
\widehat{\iota}^* : \mathcal{K} (\mathcal{D}_{\coh, \{ \widehat{X}_1 \}} ^q (\widehat{X}_2, n)) \to  \mathcal{K} (\mathcal{D}_{\coh} ^q (\widehat{X}_1,n)).
\end{equation}

To show that $\widehat{\iota}^*$ in \eqref{eqn:Gysin2} is also defined on $\mathcal{K} (\mathcal{D}^q _{\rm coh} (\widehat{X}_2, n))$ in the homotopy category, we want to check  whether the inclusion functor \eqref{eqn:X1 subcat} induces a weak-equivalence of the respective $K$-spaces. The following is a key, that we would like to consider as a moving lemma:

\begin{prop}\label{prop:D ess}
The inclusion functor \eqref{eqn:X1 subcat} is essentially surjective.
\end{prop}

The following \S \ref{sec:moving ci} and \S \ref{sec:moving general} are devoted to the proof of Proposition \ref{prop:D ess}.

\subsection{A moving lemma: the affine and complete intersection case}\label{sec:moving ci}

We continue to follow the notations and the assumptions of \S \ref{sec:moving state}. Since $X_1$ and $X_2$ are smooth, the closed immersion $X_1 \hookrightarrow X_2$ is an l.c.i.~morphism. 

\medskip

In \S \ref{sec:moving ci}, we prove Corollary \ref{cor:D ess local}, which is Proposition \ref{prop:D ess} in the special case when $Y$ is a connected affine $k$-scheme of finite type, and $X_1 \hookrightarrow X_2$ is a complete intersection. Using this and an argument resembling Zariski descent, in \S \ref{sec:moving general} we prove Proposition \ref{prop:D ess} in general.

\medskip

Let $\widehat{X}_1 = \Spf (A)$ for a regular noetherian $k$-domain $A$ complete with respect to an ideal $J \subset A$,  such that for $B:= A/J$, we have $Y= \Spec (B)$. When $r$ is the codimension of $\widehat{X}_1 \hookrightarrow \widehat{X}_2$, we have $\widehat{X}_2 = \Spf (A [[\un{t}]])$ for a set $\un{t}:= \{ t_1, \cdots, t_r \}$ of indeterminates. If $r=0$, there is nothing to consider, so we suppose $r \geq 1$.

\medskip

A basic intuition is to use a version of translation in the formal power series setting. We remark that for cycles, an analogous situation is studied in \cite[\S 5]{P general}. In what follows, some of the lemmas are very similar to some in \emph{ibid.} We give arguments independently for self-containedness.

\medskip

When we have rings $R_1 \subset R_2$ and a subset $S \subset R_1$, we use the notational conventions that $(S)_{R_2}$ is the ideal of $R_2$ generated by $S$.

\begin{defn} We introduce the following notations.
\begin{enumerate}
\item Let $A_0:= A$ and let $A_i := A[[t_1, \cdots, t_i]]$ for $1 \leq i \leq r$. They give a sequence of inclusions
$$A \subset A_1 \subset A_2 \subset \cdots \subset A_r.$$

\item Consider the ideal $J_i:= (J, t_1, \cdots, t_i )_{A_{i}} \subset A_i$ generated by $J, t_1, \cdots, t_i$. By convention $J_0:= J$. The ring $A_i$ is complete with respect to the ideal $J_i$, so we form the formal scheme $\widehat{X}_{1, i}:= \Spf (A_i)$. Here, $\widehat{X}_{1, 0} = \widehat{X}_1$ and $\widehat{X}_{1, r} = \widehat{X}_{2}$. 

Regarding $A_i$ as the quotient $A_{i+1}/ (t_{i+1})$, we have closed immersions
$$\widehat{X}_1 \subset \widehat{X}_{1,1} \subset \cdots \subset \widehat{X}_{1, r-1} \subset \widehat{X}_2.$$

\item Consider the ideal $(J_i) \subset A_{i+1}$, generated by $J_i$ in the bigger ring $A_{i+1}$. 

\item Define the set $\mathbb{V}_1= (J_0)$, and  the sets $\mathbb{V}_{i+1} := \mathbb{V}_i \times (J_{i})$ for $1 \leq i \leq r-1$, where $\times$ is the Cartesian product of sets.

\item For each $\un{c} = (c_1, \cdots, c_r) \in \mathbb{V}_r$, we consider the automorphism of $A[[\un{t}]]$ given by the translations $t_i \mapsto t_i + c_i$ for all $1 \leq i \leq r$. Note that $c_i \in (J_{i-1})$, and it may not necessarily be in $J_{i-1}$.

 Since $A[[\un{t}]]$ is complete with respect to the ideal $J_r =(J, (\un{t}))$, the above gives the induced automorphism of $\widehat{X}_2= \Spf (A[[\un{t}]])$
$$
\psi_{\un{c}}: \widehat{X}_2 \to \widehat{X}_2.
$$
This automorphism in turn induces the automorphism
$$
\psi_{\un{c}}: \widehat{X}_2 \times \square^n \to \widehat{X}_2 \times \square^n.
$$
We will use the above notations repeatedly until the end of the proof of Proposition \ref{prop:D ess}.
\qed
\end{enumerate}
\end{defn}

The automorphism $\psi_{\un{c}}$ induces pull-backs of cycles on $\widehat{X}_2 \times \square^n$. We first observe the following (cf. \cite[\S 5]{P general}):

\begin{lem}\label{lem:arbitrary translation}
Let $\un{c} \in \mathbb{V}_r$. Let $\mathfrak{Z} \in z^q (\widehat{X}_2, n)$. Then we have $\psi_{\un{c}} ^* (\mathfrak{Z}) \in z^q (\widehat{X}_2, n)$.
\end{lem}

\begin{proof}
We may assume that $\mathfrak{Z}$ is integral. We check the conditions (\textbf{GP}) and (\textbf{SF}) for $\psi_{\un{c}} ^* (\mathfrak{Z})$. 

Let $F \subset \square^n$ be a face. Since $\psi_{\un{c}}$ is an isomorphism, it preserves dimensions. Since we have 
$$
\psi_{\un{c}} ^* (\mathfrak{Z}) \cap (\widehat{X}_2 \times F) = \psi_{\un{c}} ^* (\mathfrak{Z} \cap (\widehat{X}_2 \times F)),
$$
we immediately deduce the condition (\textbf{GP}) for $\psi_{\un{c}} ^* (\mathfrak{Z})$ from that of $\mathfrak{Z}$.

\medskip

The condition (\textbf{SF}) for $\psi_{\un{c}} ^* (\mathfrak{Z})$ holds immediately, because all of $t_i, t_i + c_i, c_i$ belong to the largest ideal of definition.
\end{proof}

For a given coherent sheaf $\mathcal{F}$ on $\widehat{X}_2\times \square^n $ of codimension $\geq q$ whose associated cycle $[\mathcal{F}]$ belongs to $z^{\geq q} (\widehat{X}_2, n)$, we want to know whether we can have the additional proper intersection property with $\widehat{X}_1 \times F$ over the faces $F \subset \square_k ^n$ for the translated sheaf $\psi_{\un{c}} ^* \mathcal{F}$ for a suitable choice of $\un{c} \in \mathbb{V}_r$. The cardinality of $J$ matters in doing so. The trivial case is:

\begin{lem}\label{lem:J finite}
Suppose $|J| < \infty$. Then
\begin{enumerate}
\item $J=0$ and 
\item $Y= \widehat{X}_1 = X_1$ and $Y$ is smooth over $k$.
\end{enumerate}
\end{lem}

\begin{proof}

We have $Y= \Spec (B)$ and $\widehat{X}_1 = \Spf (A)$ with $A/J = B$. Write $X_1 = \Spec (A_0)$ for a smooth $k$-algebra $A_0$. Let $I \subset A_0$ be the ideal such that $A_0/I = B$. By definition, we have $A= \varprojlim_m A_0 / I^m$, $J= \widehat{I}$, and $A/J = A_0/I = B$.

\medskip

Since $J$ is finite, the descending chain of ideals 
$$
 \cdots \subset J^3 \subset J^ 2 \subset J
 $$
is stationary, so that there is some $N \geq 1$ such that $J^N = J^{N+1} = \cdots$. 

Since $A$ is $J$-adically complete, we have 
$$
A= \varprojlim_m A/ J^m = A/J^N.
$$
 This implies that $J^N = 0$. But $Y$ is connected so that $A$ is a regular $k$-domain. In particular, it has no nonzero nilpotent element. Hence $J= 0$, proving (1).

\medskip 

That $J=0$ implies $0 = \widehat{I}$. Hence $I=0$ and $A= A_0 = B$. This means that $Y = \widehat{X}_1 = X_1$. Because $X_1$ is smooth over $k$ and $Y= X_1$, the scheme $Y$ is smooth over $k$, proving (2).
\end{proof}

When $J$ is infinite, the technical key is Lemma \ref{lem:translation} below. The argument is somewhat involved. Analogous ideas are also to be used in \cite[\S 5]{P general}, in a slightly different situation. 

Let's begin with the following basic result from commutative algebra (see e.g. Atiyah-MacDonald \cite[Proposition 1.11-i), p.8]{AM}) needed in its proof.

\begin{lem}[Prime avoidance]\label{lem:prime avoid}
Let $A$ be a commutative ring with unity. Let $J \subset A$ be an ideal, and let $I_1, \cdots, I_N \subset A$ be prime ideals such that $J \not \subset I_i$ for all $1 \leq i \leq N$. Then $J \not \subset \bigcup_{i=1} ^N I_i$.
\end{lem}

The following, which is used in Subcases 1-2 and 1-3 of the proof of Lemma \ref{lem:translation}, needs the prime avoidance:

\begin{lem}\label{lem:bad t prime avoid}
Let $A$ be an integral domain. For a subset $S \subset A[[t]]$, let $(S) \subset A[[t]]$ be the ideal generated by $S$. Suppose that for an ideal $J \subset A$, the ideal $(J)\subset A[[t]]$ is proper.

Let $I_1, \cdots, I_N \subset A[[t]]$ be prime ideals such that $(J, t) \not \subset I_i$.

Suppose $t \in I_{i_j}$ for some indices $ 1 \leq i_1< \cdots < i_u \leq N$. Then we have 
\begin{enumerate}
\item $(J) \not \subset \bigcup_{j=1} ^u I_{i_j}$.
\item $|(J) \setminus  ( (J) \cap (\bigcup_{j=1} ^u I_{i_j}))| = \infty.$
\item For each $c \in (J) \setminus ( (J) \cap (\bigcup_{j=1} ^u I_{i_j}))$, we have $t + c \not \in I_{i_j}$ for all $1 \leq j \leq u$.
\end{enumerate}
\end{lem}

\begin{proof}
(1) For each $1 \leq j \leq u$, we are given that $t \in I_{i_j}$. If $(J) \subset I_{i_j}$, then we would have $(J, t) \subset I_{i_j}$, contradicting the given assumption that $(J, t) \not \subset I_{i_j}$. Thus we deduce that $(J) \not \subset I_{i_j}$. By the prime avoidance (Lemma \ref{lem:prime avoid}), we deduce that $(J) \not \subset \bigcup_{j=1} ^u I_{i_j}$, proving (1).

\medskip

(2) Since the set is nonempty by (1), choose any $c \in (J) \setminus  ( (J) \cap ( \bigcup_{j=1} ^u I_{i_j}))$. This is nonzero. Since $c \in (J)$ and $(J)$ is an ideal, we have $\{ c, c^2, \cdots \} \subset (J)$. 

We claim that the powers $c^1, c^2, \cdots, $ are all distinct. Indeed, suppose $c^p = c^{p'}$ for some distinct positive integers $p < p'$. This gives $c^p (1- c^{p'-p}) = 0$ in $A[[t]]$. Since $A[[t]]$ is an integral domain and $c \not = 0$, we deduce that $1- c^{p'-p} = 0$. In particular $c$ is a unit in $A[[t]]$. However, we are given that $(J)$ is a proper ideal of $A[[t]]$ so that $(J)$ contains no unit, a contradiction. In particular, the set $\{ c, c^2, c^3 , \cdots \}$ is infinite.

\medskip

If any one of $c^p$ is $\in \bigcup_{j=1} ^u I_{i_j}$ for some $p \geq 1$, then $c^p \in I_{i_{j}}$ for some $j$. But $I_{i_j}$ is a prime ideal so that it implies $c \in I_{i_j}$, which contradicts our choice of $c$. Hence the infinite set $\{ c^1, c^2, \cdots \}$ is disjoint from $\bigcup_{j=1}^u I_{i_j}$, i.e. it is contained in the set $(J) \setminus ( (J) \cap ( \bigcup_{j=1} ^u  I_{i_j}))$. This proves (2).

\medskip

(3) Toward contradiction, suppose $t + c \in I_{i_j}$ for some $1 \leq j \leq u$. Since $t \in I_{i_j}$ by our given assumption, we have $c= (t+ c) - t \in I_{i_j}$. But, this contradicts our choice of $c$ that $c \not \in I_{i_j}$. Hence $t+ c \not \in I_{i_j}$ for all $1 \leq j \leq u$, proving (3).
\end{proof}

The following is to be used in Subcase 1-3 of the proof of Lemma \ref{lem:translation}:

\begin{lem}\label{lem:pigeon hole}
Let $A$ be an integral domain complete with respect to an ideal $J \subset A$. 
Let $I \subset A[[t]]$ be a prime ideal such that $t \not \in I$. Let $c \in (J)$ be any nonzero member. 
\begin{enumerate}
\item If $c \in I$, then $t+ c^p \not \in I$ for all integers $p \geq 1$.
\item If $c \not \in I$, then there is at most one integer $p \geq 1$ such that $t+ c^p \in I$.
\end{enumerate}
\end{lem}

\begin{proof}
(1) Suppose $c \in I$. Then for all $p \geq 1$, we have $c^p \in I$. If $t+ c^p \in I$, then $t = (t+ c^p ) - c^p  \in I$, which contradicts that $t \not \in I$. Hence we must have $t + c^p \not \in I$ for all $p \geq 1$, proving (1).

\medskip

(2) Now suppose $c \not \in I$. Toward contradiction, suppose for two positive integers $p < p'$, we have $t + c^p, t+ c^{p'} \in I$. This implies that 
$$
(t + c^p) - ( t+ c^{p'} )= c^{p} ( 1- c^{p'-p}) \in I.
$$
Since $I$ is a prime ideal such that $c \not \in I$, we have $1- c^a \in I_i$, where $a:= p'-p$. 

Since $A$ is complete with respect to $J$, the ring $A[[t]]$ is complete with respect to $(J, t)$. Since $c \in (J) \subset (J, t)$, we have $1+ c^a + c^{2a} + \cdots \in A[[t]]$. Thus multiplying it to the element $1-c^a$ of the ideal $I$, we deduce that
$$
(1 + c^a + c^{2a} + \cdots ) ( 1- c^a) = 1 \in I,
$$
which is a contradiction because $I \subset A[[t]]$ is a prime ideal, thus proper. This proves (2).
\end{proof}

\begin{lem}\label{lem:translation PP}
Let $A$ and $J$ be as the above. Suppose $|J|= \infty$. Consider the ring $A[[t]]$ for a variable $t$. Let $(J)\subset A[[t]]$ be the ideal generated by $J$ in $A[[t]]$. This is proper.

Let $I_1, \cdots, I_N \subset A[[t]]$ be prime ideals such that $(J, t) \not \subset I_i$ for all $1 \leq i \leq N$. Then there exists some $c \in (J)$ such that $t+ c \not \in I_i$ for all $1 \leq i \leq N$.
\end{lem}

\begin{proof}
There are a few cases to consider. The easiest case is:

\textbf{Case 1:} Suppose $t \not \in I_i$ for all $1 \leq i \leq N$.

In this case, we may take $c=0$ to deduce the desired conclusion that $t+ c = t \not \in I_i$ for all $1 \leq i \leq N$.

\medskip

Here is the opposite case:

\textbf{Case 2:} Suppose $t \in I_i$ for all $1 \leq i \leq N$.

In this case, applying Lemma \ref{lem:bad t prime avoid}, we can find some $c \in (J) \setminus ( ( J) \cap ( \bigcup_{i=1} ^N I_i))$ such that $t + c \not \in I_i$ for all $1 \leq i \leq N$. This answers the lemma in this case.

\medskip

Here is the ``mixed" case:

\textbf{Case 3:} Suppose that $t \in I_i$ for some indices $i$, while we have $t \not \in I_{i'}$ for some other indices $i'$. 

After relabeling them, if necessary, we may assume that there is some positive integer $1 \leq s < N$ such that we have
$$
t\not \in I_i, \ \ \mbox{ for } 1 \leq i \leq s, \ \ \ \mbox{ and } t \in I_i, \ \ \mbox{ for } s+1 \leq i \leq N.
$$

We apply Lemma \ref{lem:bad t prime avoid} to the ideals $I_{s+1}, \cdots, I_N$. Let $\tilde{J}:= (J) \setminus ((J) \cap \bigcup_{i={s+1}} ^N I_i)$.

Pick $c_0 \in \tilde{J}$. By Lemma \ref{lem:bad t prime avoid}-(3), we have $ t + c_0 ^p \not \in I_i$ for $s+1 \leq i \leq N$ and $p \geq 1$. 

On the other hand, by Lemma \ref{lem:pigeon hole} there exists a finite subset $B \subset \mathbb{N}$ of size $|B|\leq s$ such that for all $p \in \mathbb{N} \setminus B$, we have $t + c_0 ^p \not \in I_i$ for all $1 \leq i \leq s$.

Thus for all $p \in \mathbb{N} \setminus B$ and for all $1 \leq i \leq N$, we have proven that $t + c_0 ^p \not \in I_i$.

This completes the proof of the lemma.
\end{proof}

\begin{lem}\label{lem:translation general}
Suppose $|J|=\infty$.

Suppose we are given finitely many integral cycles $\mathfrak{Z}_i \in z^{q_i} (\widehat{X}_2, n_i)$ for some integers $q_i, n_i \geq 0$ over the indices $1 \leq i \leq N$. 

Then there exists $\un{c} = (c_1, \cdots, c_r) \in \mathbb{V}_r$ such that $\psi_{\un{c}} ^* (\mathfrak{Z}_i) \in z^{q_i} _{\{ \widehat{X}_1 \}} (\widehat{X}_2, n_i)$ over all $1 \leq i \leq N$.
\end{lem}

\begin{proof}
By Lemma \ref{lem:arbitrary translation}, for any $\un{c} \in \mathbb{V}_r$, we already have $\psi_{\un{c}} ^* (\mathfrak{Z}_i) \in z^{q_i} (\widehat{X}_2, n_i)$. It remains to show that for a suitable $\un{c} \in \mathbb{V}_r$, each cycle $\psi_{\un{c}}^* (\mathfrak{Z}_i)$ intersects $\widehat{X}_1 \times F$ properly for all faces $F \subset \square_k ^n$.

\medskip

We prove it by induction on the codimension $r$ of $\widehat{X}_1$ in $\widehat{X}_2$.

\medskip

\textbf{Step 1:} Consider the case when $r=1$. Note that $\widehat{X}_2 \times \square^{n_i} = \Spf (A [[ t_1]] \{ y_1, \cdots, y_{n_i} \}).$

Let $I_i \subset A[[t_1]]\{ y_1, \cdots, y_{n_i} \}$ be the prime ideal of the integral cycle $\mathfrak{Z}_i$ over the indices $1 \leq i \leq N$. 

Let $F \subset \square ^{n_i}$ be a face. For the intersection $\mathfrak{Z}_i \cap (\widehat{X}_2 \times F)$, its integral components are given by a finite collection of prime ideals $I_{i, j}^F \subset A[[t_1]] \widehat{\otimes}_k \Gamma (F)$, where $\Gamma (F)$ is the coordinate ring of $F$. They all satisfy $(J, t_1) _{A[[t_1]]} \not \subset I_{i,j}^F$
by the condition (\textbf{SF}) in definition \ref{defn:HCG}.

\medskip

To obtain proper intersections with $\widehat{X}_1 \times F$ after a translation, we want some $c \in (J)$ such that $t _1+ c \not \in I_{i,j} ^F$ for all $i, j$ and $F$ at the same time. This can be achieved if the contracted prime ideals
$$
I_{i,j} ^{F,0}:= A[[t_1]] \cap I_{i,j}^F \ \ \mbox{ in }  A[[t_1]]
$$
over all $i, j , F$ have the property that $t_1 + c \not \in I_{i,j} ^{F, 0}$. Here we note that we still have $(J, t_1) _{A[[t_1]]} \not \subset I_{i,j} ^{F, 0}$
by the condition (\textbf{SF}) in definition \ref{defn:HCG}. 

Since $\{ I_{i,j} ^{F, 0} \}_{i, j, F}$ is a finite collection of prime ideals in $A[[t_1]]$, where none of them contains $(J, t_1)_{A[[t_1]]}$, we deduce from Lemma \ref{lem:translation PP} that there does exist some $c \in (J)$ such that $t_1 + c \not \in I_{i,j} ^{F, 0}$ for all $i, j, F$. 

Thus for such $c \in (J)$, the translations $\psi_{\un{c}} ^* (\mathfrak{Z}_i)$ intersect properly with $\widehat{X}_1 \times F$ for all $1 \leq i \leq N$ and all faces $F \subset \square_k ^n$. This proves the lemma for $r=1$.

\medskip

\textbf{Step 2:} Now suppose $r \geq 2$. Suppose that the lemma holds when the codimension of $\widehat{X}_1$ in $\widehat{X}_2$ is $\leq r-1$. The ring $A_{r-1}:= A[[t_1, \cdots, t_{r-1}]]$ is complete with respect to $J_{r-1}$, and $\widehat{X}_{1, r-1}= \Spf (A_{r-1})$ is a closed formal subscheme of codimension $1$ in $\widehat{X}_2$, given by the ideal generated by $t_r$.

For the given cycles $\mathfrak{Z}_i$ for $1 \leq i \leq N$, by Step $1$ applied to $\widehat{X}_{1, r-1} \subset \widehat{X}_2$, there exists some $c_r \in (J_{r-1})$ such that for all faces $F \subset \square^{n_i}$, 
$$\psi_{t_r \mapsto t_r + c_r} ^* (\mathfrak{Z}_i) \ \ \mbox{ and } \ \ \widehat{X}_{1, r-1} \times F$$
intersect properly. 

\medskip

For $1 \leq i \leq N$, let $\mathfrak{Z}'_i$ be the cycle associated to the intersection 
$$\psi_{t_r \mapsto t_r + c_r} ^* (\mathfrak{Z}_i) \cap ( \widehat{X}_{1, r-1} \times \square ^{n_i}).
$$ Since $\widehat{X}_1 \subset \widehat{X}_{1, r-1}$ is of codimension $r-1$, given by the ideal generated by $t_1, \cdots, t_{r-1}$, by the induction hypothesis we can find some $\un{c}'  = (c_1, \cdots, c_{r-1}) \in \mathbb{V}_{r-1}$ such that for all faces $F \subset \square^{n_i}$
$$
\psi_{\un{c}'} ^* (\mathfrak{Z}_i') \ \ \mbox{ and } \ \ \widehat{X}_1 \times F
$$
intersect properly. 

\medskip 

Combining the above choices of $\un{c}' \in \mathbb{V}_{r-1}$ and $c_r \in (J_{r-1})$, let $\un{c}:= (\un{c}', c_r) \in \mathbb{V}_r$. By construction, $\psi_{\un{c}} ^* (\mathfrak{Z}_i)$ and $\widehat{X}_1 \times F$ intersect properly for all faces $F \subset \square^{n_i}$ and all $1 \leq i \leq N$.

This completes the proof of the lemma.
\end{proof}

\begin{lem}\label{lem:translation}
Suppose $|J|= \infty$.

If $\mathcal{F}$ is a coherent $\mathcal{O}_{\widehat{X}_2 \times \square^n}$-module such that $[\mathcal{F}]$ belongs to the group $z^{\geq q} (\widehat{X}_2, n)$, then there exists $\un{c} = (c_1, \cdots, c_r) \in \mathbb{V}_r$ such that for the translated coherent sheaf $\psi_{\un{c}} ^* (\mathcal{F})$, the associated cycle $[\psi_{\un{c}} ^* (\mathcal{F})] \in z^{\geq q} _{\{ \widehat{X}_1 \}} (\widehat{X}_2,n )$.
\end{lem}

\begin{proof}
Write $[\mathcal{F}] = \sum_{i=1} ^N m_i \mathfrak{Z}_i$ for some integers $m_i >0$ and finitely many distinct integral cycles $\mathfrak{Z}_i \in z^{q_i} (\widehat{X}_2, n)$, for some integers $q_i \geq q$ and $n_1= \cdots= n_N = n$.

We apply Lemma \ref{lem:translation general} to those $\mathfrak{Z}_i$. Indeed by this lemma, we have some $\un{c} \in \mathbb{V}_r$ such that $\psi_{\un{c}} ^* (\mathfrak{Z}_i) \in z^{q_i} _{\{\widehat{X}_1 \} }(\widehat{X}_2, n)$ for all $1 \leq i \leq N$.

On the other hand, since $\psi_{\un{c}}$ is an isomorphism given by a translation, we have $[ \psi_{\un{c}} ^* (\mathcal{F})] = \sum_{i=1} ^N m_i \psi_{\un{c}} ^* (\mathfrak{Z}_i)$, which is in $z^{\geq q} _{\{ \widehat{X}_1 \}} (\widehat{X}_2,n )$ as desired.
\end{proof}

\begin{cor}\label{cor:D ess local}
Let $Y$ be a connected affine $k$-scheme of finite type, with closed immersions $Y \hookrightarrow X_1 \hookrightarrow X_2$ into equidimensional smooth $k$-schemes $X_1, X_2$ such that the immersion $X_1 \hookrightarrow X_2$ is a complete intersection.

Then Proposition \ref{prop:D ess} holds in this case, i.e. the inclusion functor
\begin{equation}\label{eqn:D ess local 0}
\mathcal{D}_{\coh, \{ \widehat{X}_1\}} ^q (\widehat{X}_2, n) \hookrightarrow \mathcal{D}_{\coh}^q (\widehat{X}_2, n)
\end{equation}
is essentially surjective.
\end{cor}

\begin{proof}
Write $Y= \Spec (B)$, $\widehat{X}_1= \Spf (A)$, and $\widehat{X}_2= \Spf (A[[\un{t}]])$, where $B= A/J$ for an ideal $J \subset A$ as  above. If $X_1= X_2$, then there is nothing to prove. So, we suppose $X_1 \subsetneq X_2$ and $r:= {\rm codim}_{X_2} X_1 \geq 1$.

Let $\mathcal{F}$ be a coherent $\mathcal{O}_{\widehat{X}_2\times \square^n}$-module of codimension $\geq q$, such that its associated cycle $[\mathcal{F}]$ belongs to $z^{\geq q} (\widehat{X}_2, n)$.

\medskip

\textbf{Case 1:} Suppose $|J|< \infty$. By Lemma \ref{lem:J finite}, we have $J=0$ and $Y= \widehat{X}_1= X_1$. 

Our requirement is that $[\mathcal{F}] \in z^{\geq q} (\widehat{X}_2,n)$ so that each component already intersects properly with $Y \times F$, which is $\widehat{X}_1 \times F$, for all faces $F \subset \square^n$. Thus in this case the functor \eqref{eqn:D ess local 0} is the identity and there is nothing to show.

\medskip

\textbf{Case 2:} Now suppose $|J|= \infty$. By Lemma \ref{lem:translation}, there exists $\un{c} \in \mathbb{V}_r$ such that $\psi_{\un{c}} ^* \mathcal{F}$ is a coherent $\mathcal{O}_{\widehat{X}_2 \times \square^n}$-module that intersects $\widehat{X}_1 \times F$ properly for all faces $F \subset \square^n$. But, $\psi_{\un{c}}$ is an automorphism of $\widehat{X}_2 \times \square^n$ so that $\psi_{\un{c}} ^* \mathcal{F}$ is isomorphic to $\mathcal{F}$ as $\mathcal{O}_{\widehat{X}_2 \times \square^n}$-modules.

Thus after replacing $\mathcal{F}$ by $\psi_{\un{c}} ^* \mathcal{F}$, which is in $\mathcal{D}_{\coh, \{ \widehat{X}_1 \}} ^q (\widehat{X}_2,n)$, we see that the inclusion \eqref{eqn:D ess local 0} is essentially surjective. This proves the corollary.
\end{proof}

The above Corollary \ref{cor:D ess local} answers part of Proposition \ref{prop:D ess} when $Y$ is connected affine and $X_1 \hookrightarrow X_2$ is a complete intersection.

\subsection{A moving lemma: the general case}\label{sec:moving general}
The argument via translations given in \S \ref{sec:moving ci} does not extend to a general quasi-projective $Y$ and a general closed immersion $X_1 \hookrightarrow X_2$ between smooth $k$-schemes. This may not be a complete intersection, but it is always a l.c.i., so a version of local-to-global Zariski descent-type gluing may offer a way-out. 

In \S \ref{sec:moving general}, we finish the proof of Proposition \ref{prop:D ess} using such an argument.

\begin{proof}[Proof of Proposition \ref{prop:D ess}]
If $Y$ is not connected, then we can argue for each connected component separately. Hence we may assume $Y$ is connected.

Since the morphism $X_1 \hookrightarrow X_2$ is l.c.i., there is a finite affine open cover $\mathcal{V}=\{ V_1, \cdots, V_N \}$ of $X_2$ such that $X_1 \cap V_i \hookrightarrow X_2 \cap V_i$ is a complete intersection for each $1 \leq i \leq N$.

Let $U_i:= Y \cap V_i$. We regard $U_i$ as an open subset of $Y$ as well as an open subscheme of $Y$. The collection $\mathcal{U}:= \{ U_1, \cdots, U_N\}$ gives an affine open cover of $Y$. We may assume each $U_i$ is connected.

Recall that $\widehat{X}_{\ell}$ is a ringed space whose underlying topological space $|\widehat{X}_{\ell}|$ is exactly equal to that of $Y$. So, when $U \subset Y$ is an open subset, recall $\widehat{X}_{\ell}|_U$ means the open formal subscheme $(U, \mathcal{O}_{\widehat{X}_{\ell}}|_U)$. For $\ell=1,2$ and $1 \leq i \leq N$, note that $\widehat{X}_{\ell} |_{U_i} $ is equal to the completion of $X_{\ell} \cap V_i$ along $U_i$. 

By Corollary \ref{cor:D ess local} applied to each $1 \leq i \leq N$, we see that the inclusion functor
$$
\mathcal{D}_{\coh, \{ \widehat{X}_1 |_{U_i} \}} ^q (\widehat{X}_2 |_{U_i},n) \hookrightarrow \mathcal{D}_{\coh}^q (\widehat{X}_2 |_{U_i},n)
$$
is essentially surjective. To prove Proposition \ref{prop:D ess} by induction, it is enough to show that if the proposition holds for two (not necessarily affine) open subsets of $Y$ that cover $Y$, then the proposition holds for $Y$ as well.

\medskip

So, suppose for two open sets $V_1, V_2 \subset X_2$, with $U_i = Y \cap V_i$, the proposition holds for $(U_{\ell}, \widehat{X}_1|_{U_{\ell}}, \widehat{X}_2|_{U_{\ell}})$, $\ell=1,2$, and $Y= U_1 \cup U_2$ so that $\widehat{X}_2= \widehat{X}_2|_{U_1 \cup U_2}$. 

Let $\mathcal{F}$ be a coherent sheaf in $\mathcal{D}_{\coh} ^q (\widehat{X}_2,n)$. We prove that for some $\mathcal{F}' \in \mathcal{D}_{\coh, \{ \widehat{X}_1 \}} ^q (\widehat{X}_2,n)$, there is a quasi-isomorphism between $\mathcal{F}$ and $\mathcal{F}'$. 

\medskip

In what follows, for each open subset $U \subset Y$, as a shorthand let us write $\mathcal{F}|_U$ instead of the bulky notation $\mathcal{F}|_{ \widehat{X}_2|_U \times \square^n}$, for notational simplicity.

By the given assumptions, for $\ell=1,2$ there are quasi-isomorphisms
$$
\alpha_{\ell}: \mathcal{F}|_{U_{\ell}} \overset{\sim}{\to} \mathcal{M}_{\ell}
$$
for some pseudo-coherent complexes $\mathcal{M}_{\ell} \in \mathcal{D}_{\coh, \{ \widehat{X}_1 |_{U_{\ell}}\}} ^q (\widehat{X}_2 |_{U_{\ell}},n)$. By Lemma \ref{lem:coh=perfect}, each $\mathcal{M}_{\ell}$ is a perfect complex on $\widehat{X}_2|_{U_{\ell} \times \square^n}$.
Let $j_\ell$ (resp. $j_{12}$) denote the open immersion 
$\widehat{X}_2|_{U_{\ell}}  \times \square^n \hookrightarrow \widehat{X}_2  \times \square^n $ (resp.
$\widehat{X}_2 |_{U_1 \cap U_2}   \times \square^n \hookrightarrow \widehat{X}_2 \times \square^n $).
Since $\mathcal{F}$ is a $\mathcal{O}_{\widehat{X}_2 \times \square^n }$-module, we then have a short exact sequence of $\mathcal{O}_{\widehat{X}_2 \times \square^n }$-modules
$$
0 \to  j_{12!} \mathcal{F}|_{U_1 \cap U_2} \to   j_{1!} \mathcal{F}|_{U_1} 
\oplus j_{2!} \mathcal{F} |_{U_2} \to  \mathcal{F}  \to 0.
$$ 
We can map it to the corresponding distinguished triangle in the derived category of $\mathcal{O}_{\widehat{X}_2 \times \square^n }$-modules. 

For an open immersion $j$, since $(j_!, j^*)$ is an adjoint pair (see e.g. SGA ${\rm IV}_1$ \cite[Exp. IV, 11.3.3]{SGA4-1} and SGA ${\rm IV}_2$ \cite[Exp. V, 1.3.1]{SGA4-2}), we have a morphism
\begin{equation}\label{eqn:D ess triangles}
\xymatrix{
j_{12!} \mathcal{F}|_{U_1 \cap U_2} \ar[d] ^{  \alpha_1 |_{U_1 \cap U_2}}  \ar[r] &  j_{1!} \mathcal{F}|_{U_1} \oplus j_{2!} \mathcal{F} |_{U_2} \ar[d] ^{ (\alpha_1 , \alpha_2)}  \ar[r] &  \mathcal{F} \ar[r] \ar@{-->}[d] ^{\gamma} &( j_{12!} \mathcal{F}|_{U_1 \cap U_2}) [1] \ar[d] ^{\alpha_1 |_{U_1 \cap U_2} [1]} \\
 j_{12!} \mathcal{M}_1 |_{U_1 \cap U_2} 
 \ar[r] ^{\beta \ \ \ \ } 
 & j_{1!} \mathcal{M}_1 |_{U_1} \oplus j_{2!} \mathcal{M}_2 |_{U_2} \ar[r] & \mathcal{F}' \ar[r] &( j_{12!} \mathcal{M} |_{U_1 \cap U_2}) [1],}
\end{equation}
of distinguished triangles in the derived category of $\mathcal{O}_{\widehat{X}_2  \times \square^n }$-modules, with $\beta = (\iota,\alpha_2 \circ \alpha_1 ^{-1}|_{U_1 \cap U_2})$, where
$\iota$ the canonical map induced by the counit of the adjunction
$(j_!, j^*)$, and
$$
\mathcal{F}':= {\rm Cone} (\beta).
$$
Notice that the morphism $\gamma$ exists by the axiom \textbf{TR3} of triangulated categories (see A. Neeman \cite{Neeman}). Since $\mathcal{F}'$ restricted over $U_{\ell}$ is just $\mathcal{M}_{\ell} $ for $\ell=1,2$, we see that $\mathcal{F}'$ is a perfect complex. By construction we see that $\mathcal{F}' \in \mathcal{D}_{\coh, \{ \widehat{X}_1 \}} ^q (\widehat{X}_2,n)$. Since the first two vertical morphisms of \eqref{eqn:D ess triangles} are quasi-isomorphisms, so is the morphism $\gamma$. This shows that the inclusion functor \eqref{eqn:D ess local 0} is essentially surjective, finishing the proof of the proposition.
\end{proof}

\begin{remk}
Note that when $\mathcal{F}|_{U}$ is quasi-coherent, the sheaf $j_! \mathcal{F}_{|U}$ is not necessarily quasi-coherent in general. But this is not a problem here; we use implicitly the fact that the algebraic $K$-groups constructed from various kinds of derived categories are all equivalent. See Thomason-Trobaugh \cite[Lemmas 3.5-3.7, p.313]{TT}.

We also remark that the above construction of $\mathcal{F}' = {\rm Cone} (\beta)$ is somewhat similar to the argument around \cite[(3.20.4.2), p.326]{TT}, where a perfect complex is obtained by gluing two perfect complexes defined on open subsets.
\qed
\end{remk}

\begin{cor}\label{cor:Gysin cnv}
Let $Y$ be a quasi-projective $k$-scheme. Let $Y \hookrightarrow X_1 \overset{\iota}{\hookrightarrow} X_2$ be closed immersions, where $X_1, X_2$ are equidimensional smooth $k$-schemes. Let $\widehat{X}_i$ be the completion of $X_i$ along $Y$.

Then the zigzag of functors
$$
\mathcal{D}_{\coh} ^q (\widehat{X}_2,n) \overset{\mathfrak{i}}{\hookleftarrow} \mathcal{D}_{\coh, \{ \widehat{X}_1 \}} ^q (\widehat{X}_2,n ) \overset{\widehat{\iota}^\ast}{\to} \mathcal{D}_{\coh}^q (\widehat{X}_1,n),$$
induces morphisms of the induced Waldhausen $K$-spaces
\begin{equation}\label{eqn:Gysin cnv 0}
\mathcal{G}^q (\widehat{X}_2,n) \overset{\mathfrak{i}}{\leftarrow} \mathcal{G}^q _{\{ \widehat{X}_1 \}} (\widehat{X}_2,n)\overset{\widehat{\iota}^*}{\to} \mathcal{G}^q (\widehat{X}_1,n),
\end{equation}
where the first arrow $\mathfrak{i}$ is a weak-equivalence. In particular, we have the Gysin morphisms of spaces in the homotopy category
\begin{equation}\label{eqn:Gysin cnv 1}
\tuborg
\widehat{\iota}^*: \mathcal{G}^q (\widehat{X}_2,n) \to \mathcal{G}^q (\widehat{X}_1,n), \\
\widehat{\iota}^*: \mathcal{G}^q (\widehat{X}_2) \to \mathcal{G}^q (\widehat{X}_1),
\sluttuborg
\end{equation}
where $\mathcal{G}^q (-)$ are as in Definition \ref{defn:g geometric}.
\end{cor}

\begin{proof}
By Proposition \ref{prop:D ess}, the inclusion functor $\mathfrak{i}: \mathcal{D}_{\coh, \{ \widehat{X}_1 \}} ^q (\widehat{X}_2,n) \hookrightarrow \mathcal{D}_{\coh} ^q (\widehat{X}_2,n)$ in \eqref{eqn:X1 subcat} is essentially surjective.
In addition, the functors $\mathfrak{i}$, $\widehat{\iota}^\ast$ commute
with the respective faces and degeneracies, see Remark \ref{cubicalfaces}.

Hence in the homotopy category, we obtain the morphisms \eqref{eqn:Gysin cnv 0} for each $n\geq 0$, and by \cite[Theorems 1.9.1, 1.9.8, p.263, 271]{TT},  $\mathfrak{i}$ 
 induces a weak-equivalence of their respective $K$-spaces, i.e.~the first arrow of \eqref{eqn:Gysin cnv 0} is a weak-equivalence. Thus we deduce the first morphism of \eqref{eqn:Gysin cnv 1}. 
 
 The morphisms in \eqref{eqn:Gysin cnv 0} over $n \geq 0$ form morphisms of cubical spaces, and taking their respective geometric realizations, we deduce the second morphism of \eqref{eqn:Gysin cnv 1}.
\end{proof}

\subsection{Some general pull-backs}\label{sec:moving pull-back}
In \S \ref{sec:moving pull-back}, we generalize Corollary \ref{cor:Gysin cnv} a bit. The extension we obtain is given as Theorem \ref{thm:moving pull-back} below.

Note that we have:

\begin{lem}\label{lem:prodcomp}
For $i=1,2$, let $Y_i$ be quasi-projective $k$-schemes. Choose any closed immersions $Y_i \hookrightarrow X_i$ into smooth $k$-schemes, and let $\widehat{X}_i$ be the completion of $X_i$ along $Y_i$ for $i=1,2$. We also consider $Y_1 \times_k Y_2$ as a closed subscheme of $X_1 \times_k X_2$, and let $\widehat{X_1 \times_k X_2}$ be the completion of $X_1 \times_k X_2$ along $Y_1 \times Y_2$. 

Then we have a natural isomorphism of formal schemes
\begin{equation}\label{eqn:prodcomp0}
 \widehat{X}_1 \times_k \widehat{X}_2 \simeq \widehat{X_1 \times_k X_2}.
 \end{equation}
\end{lem}

\begin{proof}

The question is local, so we may assume $Y_i$ and $X_i$ are affine.

For $i=1,2$, write $Y_i = \Spec (B_i)$, $X_i = \Spec (A_i)$. We have natural surjections $A_i \to B_i$ and let $I_i$ be the kernels. Let $A_3:= A_1 \otimes_k A_2$, $B_3:= B_1 \otimes_k B_2$, so that $X_1 \times_k X_2 = \Spec (A_3)$ and $Y_1 \times_k Y_2 = \Spec (B_3)$. 

We have
\begin{equation}
B_3  = \frac{A_1}{I_1} \otimes_k \frac{A_2}{I_2} \simeq \frac{ A_1 \otimes_k A_2}{ I_1 \otimes_k A_2 + A_1 \otimes_k I_2},
\end{equation}
and let $I_3:= I_1 \otimes_k A_2 + A_1 \otimes_k I_2$. For $1 \leq i \leq 3$, we have $\widehat{X}_i = \Spf (\widehat{A}_i)$, where $\widehat{A}_i:= \varprojlim_n \frac{A_i}{ I_i ^n}.$

The product $\widehat{X}_1 \times \widehat{X}_2$ of the formal schemes is the formal spectrum of the completed tensor product $\widehat{A}_1 \widehat{\otimes}_k \widehat{A}_2$. This is given also by
\begin{equation}\label{eqn:prodcomp1}
\widehat{A}_1 \widehat{\otimes}_k \widehat{A}_2=\varprojlim_{m,n} \frac{A_1}{ I_1 ^m} \otimes_k \frac{ A_2}{I_2 ^n} \simeq \varprojlim_{m,n} \frac{ A_1 \otimes_k A_2}{ I_1 ^m \otimes_k A_2 + A_1 \otimes_k I_2 ^n}.
\end{equation}

For integers $m, n \geq 1$, if $N \geq m+n$, by the binomial theorem
\begin{equation}\label{eqn:prodcomp2-1}
I_3 ^N = (I_1 \otimes_k A_2 + A_1 \otimes_k I_2)^N \subset I_1 ^m \otimes_k A_2 + A_1 \otimes_k I_2 ^n .
\end{equation}
Conversely for an integer $N \geq 1$, we have
\begin{equation}\label{eqn:prodcomp2-2}
I_1 ^N \otimes_k A_2 + A_1 \otimes_k I_2 ^N \subset (I_1 \otimes_k A_2 + A_1 \otimes_k I_2)^N = I_3 ^N.
\end{equation}

By \eqref{eqn:prodcomp2-1} and \eqref{eqn:prodcomp2-2}, the two systems $\{ I_1 ^m \otimes_k A_2 + A_1 \otimes_k I_2 ^n \}_{m, n \geq 1}$ and $\{ I_3 ^N \}_{N \geq 1}$ define the same topology on $A_3$. Hence the ring $\widehat{A}_1 \widehat{\otimes}_k \widehat{A}_2$ of \eqref{eqn:prodcomp1} is equal to the ring $\varprojlim_N A_3/ I_3 ^N$
\end{proof}

The following generalization of Corollary \ref{cor:Gysin cnv} is used repeatedly for the rest of the article:

\begin{thm}\label{thm:moving pull-back}

Let $g: Y_1 \to Y_2$ be a morphism of quasi-projective $k$-schemes. Let $Y_i \hookrightarrow X_i$ be closed immersion into equidimensional smooth $k$-schemes for $i=1,2$. Let $\widehat{X}_i$ be the completion of $X_i$ along $Y_i$. Suppose there is a morphism $f: X_1 \to X_2$ such that the diagram
\begin{equation}\label{eqn:moving pull-back 0}
\xymatrix{ 
X_1 \ar[r]^f & X_2 \\
Y_1 \ar[r]^g \ar@{^{(}->}[u] & Y_2 \ar@{^{(}->}[u] 
}
\end{equation}
commutes. Then for each $q \geq 0$, $n\geq 0$, we have the induced morphisms in the homotopy category
\begin{equation}\label{eqn:f* final}
\tuborg
\widehat{f}^*: \mathcal{G}^q (\widehat{X}_2,n) \to \mathcal{G}^q (\widehat{X}_1,n), \\
\widehat{f}^*: \mathcal{G}^q (\widehat{X}_2) \to \mathcal{G}^q (\widehat{X}_1),
\sluttuborg
\end{equation}
where $\mathcal{G}^q (-)$ are as in Definition \ref{defn:g geometric}.

Furthermore they form the commutative diagrams
$$
\xymatrix{
\mathcal{G}^{q+1} (\widehat{X}_2,n) \ar[d] \ar[r] & \mathcal{G}^{q+1} (\widehat{X}_1,n) \ar[d] &&
\mathcal{G}^{q+1} (\widehat{X}_2) \ar[d] \ar[r] & \mathcal{G}^{q+1} (\widehat{X}_1) \ar[d]\\
\mathcal{G}^q (\widehat{X}_1,n) \ar[r] & \mathcal{G}^q (\widehat{X}_1,n), &&
\mathcal{G}^q (\widehat{X}_1) \ar[r] & \mathcal{G}^q (\widehat{X}_1)
}
$$
in the homotopy category, where the vertical arrows are as in
\eqref{eqn:higher n-tower}-\eqref{eqn:higher tower}.

\end{thm}

\begin{proof}
From the given diagram \eqref{eqn:moving pull-back 0}, we deduce the commutative diagram
$$
\xymatrix{ 
\widehat{X}_1 \ar[r]^{\widehat{f}}& \widehat{X}_2 \\
Y_1 \ar[r]^g \ar@{^{(}->}[u] & Y_2. \ar@{^{(}->}[u] 
}
$$
The diagram \eqref{eqn:moving pull-back 0} also induces the commutative diagram
$$
\xymatrix{
X_1 \ar@{^{(}->}[r] ^{gr_f \ \ \ } & X_1 \times X_2 \ar[r] ^{\ \ \ pr_2} & X_2 \\
Y_1 \ar@{^{(}->}[u] \ar@{^{(}->}[r]^{gr_g \ \ \ } & Y_1 \times Y_2  \ar@{^{(}->}[u] \ar[r] ^{\ \ \ pr_2} & Y_2. \ar@{^{(}->}[u]}
$$

Let $X_{12}:= X_1 \times X_2$ and let $\widehat{X}_{12}$ be the completion of $X_{12}$ along $Y_1 \times Y_2$. This is equal to $\widehat{X}_1 \times \widehat{X}_2$ by Lemma \ref{lem:prodcomp}.

We also have the closed immersions $Y_1\hookrightarrow X_1 \overset{\iota:= gr_f}{\hookrightarrow} X_1 \times X_2= X_{12}$. We let $\widehat{X}_{12}'$ be the completion of $X_{12}$ along $Y_1$. This is equal to the further completion $\alpha: \widehat{X}_{12}' \to \widehat{X}_{12}$ of $\widehat{X}_{12}$ along $Y_1$ as well. Thus we have the commutative diagram
\begin{equation}\label{eqn:moving pull-back1}
\begin{array}{c}
\xymatrix{
\widehat{X}_1 \ar@{^{(}->}[r] ^{\widehat{\iota}}\ar@/^1.5pc/[rrr] ^{\widehat{f}} & \widehat{X}_{12}' \ar[r] ^{\alpha} & \widehat{X}_{12} \ar[r] ^{\widehat{pr}_2} & \widehat{X}_2 \\
Y_1 \ar@{^{(}->}[u] \ar@{=}[r] & Y_1  \ar@{^{(}->}[u]   \ar@{^{(}->}[r]^{gr_g} & Y _1\times Y_2  \ar@{^{(}->}[u]  \ar[r] ^{\ \ \ pr_2} & Y_2.  \ar@{^{(}->}[u] }
\end{array}
\end{equation}
Since the morphisms $\alpha$ (see EGA I \cite[Corollaire (10.8.9), p.197]{EGA1} or \cite[Theorems 8.8 and 8.12, p.60-61]{Matsumura}) and $\widehat{pr}_2$ are flat, if a coherent sheaf is of codimension $\geq q$, so are its flat pull-backs. Thus we have the associated pull-backs
\begin{equation}\label{eqn:moving pull-back2}
\tuborg
\widehat{pr_2}^*: \mathcal{D}_{\coh} ^q (\widehat{X}_2,n) \to \mathcal{D}_{\coh} ^q (\widehat{X}_{12},n),\\
\alpha^*: \mathcal{D}_{\coh} ^q (\widehat{X}_{12},n) \to \mathcal{D}_{\coh} ^q (\widehat{X}_{12}',n),
\sluttuborg
\end{equation}
thus applying the Waldhausen constructions, we have morphisms of the Waldhausen $K$-spaces
\begin{equation}\label{eqn:alpha pr}
\tuborg  \widehat{pr}^*_2 : \mathcal{G}^q (\widehat{X}_2,n) \to \mathcal{G}^q (\widehat{X}_{12},n), \\
 \alpha^* :\mathcal{G}^q (\widehat{X}_{12},n) \to \mathcal{G}^q (\widehat{X}_{12}',n).\sluttuborg
\end{equation}
On the other hand, the morphism $\widehat{\iota}$ is obtained from the closed immersion $ X_1 \overset{\iota=gr_f}{\hookrightarrow} X_{12}$ so that by Corollary \ref{cor:Gysin cnv}, we have the pull-back in the homotopy category
\begin{equation}\label{eqn:Gysin final}
\widehat{\iota}^*: \mathcal{G}^q (\widehat{X}_{12}',n) \to \mathcal{G}^q (\widehat{X}_1,n).
\end{equation}
Composing all pull-backs in \eqref{eqn:alpha pr} and \eqref{eqn:Gysin final} in the correct order, we obtain the first desired pull-back of \eqref{eqn:f* final}
$$
\widehat{f}^*: \mathcal{G}^q (\widehat{X}_2, n) \to \mathcal{G}^q (\widehat{X}_1, n).
$$
This induces a morphism of cubical spaces in the homotopy category with the faces and degeneracies, see e.g. Remark \ref{cubicalfaces}. Thus taking the geometric realizations, we deduce the second morphism in \eqref{eqn:f* final} in the homotopy category.

The second part is apparent and we omit details.
\end{proof}

The above construction can be composed:

\begin{lem}\label{lem:moving pull-back composed}
Let $g_1: Y_1 \to Y_2$ and $g_2: Y_2 \to Y_3$ be morphisms of quasi-projective $k$-schemes. Suppose we have closed immersions $Y_i \hookrightarrow X_i$ into equidimensional smooth $k$-schemes for $i=1,2,3$, and there are morphisms $f_1: X_1 \to X_2$ and $f_2: X_2 \to X_3$ that form a commutative diagram
$$
\xymatrix{
X_1 \ar[r] ^{f_1} & X_2 \ar[r] ^{f_2} & X_3 \\
Y_1 \ar[r] ^{g_1} \ar@{^{(}->}[u] & Y_2 \ar[r] ^{g_2}  \ar@{^{(}->}[u]  & Y_3.  \ar@{^{(}->}[u]}
$$
Let $\widehat{X}_i$ be the completion of $X_i$ along $Y_i$ for $i=1,2,3$, and let 
$$
\xymatrix{
\widehat{X}_1 \ar[r] ^{\widehat{f}_1} & \widehat{X}_2 \ar[r] ^{\widehat{f}_2} & \widehat{X}_3 \\
Y_1 \ar[r] ^{g_1} \ar@{^{(}->}[u] & Y_2 \ar[r] ^{g_2}  \ar@{^{(}->}[u]  & Y_3.  \ar@{^{(}->}[u]}
$$
be the induced commutative diagram.

Then we have
\begin{equation}\label{eqn:pull-back transitive}
\tuborg
(\widehat{f}_2 \circ \widehat{f}_1)^* &= \widehat{f}_1 ^* \circ \widehat{f}_2^*: \mathcal{G}^q (\widehat{X}_3,n) \to \mathcal{G}^q (\widehat{X}_1,n), \\
(\widehat{f}_2 \circ \widehat{f}_1)^*  &= \widehat{f}_1 ^* \circ \widehat{f}_2^*: \mathcal{G}^q (\widehat{X}_3) \to \mathcal{G}^q (\widehat{X}_1), 
\sluttuborg
\end{equation}
in the homotopy category.
\end{lem}

\begin{proof}
Following the notational conventions in Theorem \ref{thm:moving pull-back} and its proof, we define $\widehat{X}_{12}, \widehat{X}_{12}', \widehat{X}_{13}, \widehat{X}_{13}', \widehat{X}_{23}, \widehat{X}_{23}'$ similarly. They form part of the commutative diagram 
$$
\xymatrix{
\widehat{X}_1 \ar[r]  \ar[drr] & \widehat{X}_{12}' \ar[r] & \widehat{X}_{12} \ar[r] & \widehat{X}_2 \ar[r] & \widehat{X}_{23}' \ar[r] & \widehat{X}_{23} \ar[r] & \widehat{X}_3\\
& & \widehat{X}_{13}' \ar[rr] && \widehat{X}_{13}. \ar[rru] & &}
$$
This diagram and the construction in Theorem \ref{thm:moving pull-back} imply \eqref{eqn:pull-back transitive}. We omit details.
\end{proof}

\section{The $K$-spaces under the mod equivalences}\label{sec:mod Y spaces}

For a while, let $Y$ be an affine $k$-scheme of finite type. Let $Y \hookrightarrow X$ be a closed immersion into an equidimensional smooth $k$-scheme, and let $\widehat{X}$ be the completion of $X$ along $Y$. 

In \S \ref{sec:mod Y spaces}, we construct a tower of ``$K$-spaces modulo $Y$" in the homotopy category, using ideas inspired by the derived Milnor patching of S. Landsburg \cite{Landsburg Duke}.

\subsection{A pushout and derived Milnor patching}\label{sec:derived Milnor}

We recall a few basic results around the Milnor patching needed in this article. The readers may find some analogies in the constructions given in \cite{P general}. We would like to give self-contained treatments here as much as we could, and some basic lemmas may be repeated here. First recall from D. Ferrand \cite[Th\'eor\`eme 5.4]{Ferrand}:

\begin{lem}\label{lem:Ferrand}
Let $X'$ be a scheme, $Y' \subset X'$ a closed subscheme, and $g: Y' \to Y$ is a finite morphism of schemes. Consider the pushout $X:= X' \coprod_{Y'} Y$ in the category of ringed spaces so that we have the co-Cartesian square
$$
\xymatrix{Y' \ar[r] ^g \ar[d] & Y \ar[d] ^u\\
X' \ar[r] ^f & X.}
$$

Suppose that $X'$ and $Y$ satisfy the following:
\begin{enumerate}
\item [{\rm (FA)}] Each finite subset of points is contained in an affine open subset.
\end{enumerate}

Then $X$ is a scheme satisfying ${\rm (FA)}$, the diagram is Cartesian as well, and the morphism $f$ is finite, while $u$ is a closed immersion.
\end{lem}

\begin{lem}\label{lem:double}
Let $Y$ be an affine $k$-scheme of finite type. Let $Y \hookrightarrow X$ be a closed immersion into an equidimensional smooth $k$-scheme, and let $\widehat{X}$ be the completion of $X$ along $Y$.

Then the push-out $D_{\widehat{X}} = \widehat{X} \coprod_Y \widehat{X}$, taken as a locally ringed space, is a noetherian formal scheme.
\end{lem}

\begin{proof}
Let $\mathcal{I}$ be the ideal sheaf of the closed immersion $Y \hookrightarrow \widehat{X}$. This is also an ideal of definition of the noetherian formal scheme $\widehat{X}$ as well. For each $m \geq 1$, let $Y_m \subset \widehat{X}$ be the closed subscheme defined by $\mathcal{I}^m$. 

Since $Y_m$ is affine, it satisfies the condition (FA) of Lemma \ref{lem:Ferrand}. In particular, the pushout $D_m:= Y_m \coprod_Y Y_m$ exists as a noetherian affine $k$-scheme. Then we take $D_{\widehat{X}} := \underset{m}{\colim} \  D_m$ in the category of noetherian formal schemes. We have $\widehat{X}= \underset{m}{ \colim}  \ Y_m$, and there are closed immersions $j_i: \widehat{X} \hookrightarrow D_{\widehat{X}}$ for $i=1,2$. Since colimits commute among themselves, we have $D_{\widehat{X}} = \widehat{X} \coprod_Y \widehat{X}$ as desired.
\end{proof}

We say that the above $D_{\widehat{X}}$ is the \emph{double} of $\widehat{X}$ along $Y$. We have the following version of the derived Milnor patching of S. Landsburg \cite{Landsburg Duke}:

\begin{lem}\label{lem:derived Milnor}
Let $Y$ be an affine $k$-scheme of finite type. Let $Y \hookrightarrow X$ be a closed immersion into an equidimensional smooth $k$-scheme, and let $\widehat{X}$ be the completion of $X$ along $Y$. Consider the double $D_{\widehat{X}}$ of Lemma \ref{lem:double} and the associated co-Cartesian diagram
$$
\xymatrix{
Y \ar[r]^{\iota_1} \ar[d] ^{\iota_2} & \widehat{X} \ar[d] ^{j_2} \\
\widehat{X} \ar[r] ^{j_1} & D_{\widehat{X}}.}
$$

Let $\mathcal{F}_1$ and $\mathcal{F}_2$ be perfect complexes on $\widehat{X}$ such that we have an isomorphism $\mathbf{L}\iota_1 ^* \mathcal{F}_1 \simeq \mathbf{L}\iota_2 ^* \mathcal{F}_2$. Then there exists a perfect complex $\tilde{\mathcal{F}}$ on $D_{\widehat{X}}$ such that $\mathbf{L} j_i ^* \tilde{ \mathcal{F} } \simeq \mathcal{F}_i$ for $i=1,2$.
\end{lem}

\begin{proof}
This essentially follows from the main theorem of \cite{Landsburg Duke}. We sketch the argument. Let $\mathcal{I}$ be the ideal sheaf of the closed immersion $Y \hookrightarrow \widehat{X}$. Let $Y_m \hookrightarrow \widehat{X}$ be the closed subscheme defined by $\mathcal{I}^m$ for integers $m \geq 1$. Let $\mathcal{F}_{i, m}$ be the derived pull-back of $\mathcal{F}_i$ to $Y_m$, for $i=1,2$. We let $\iota_{i,m}: Y \hookrightarrow Y_m$ be the closed immersion. 

The given condition $\mathbf{L}\iota_1 ^* \mathcal{F}_1 \simeq \mathbf{L}\iota_2 ^* \mathcal{F}_2$ implies that $\mathbf{L} \iota_{1,m} ^* \mathcal{F}_{1, m} \simeq \mathbf{L} \iota_{2,m} ^* \mathcal{F}_{2,m}$. Thus by the main theorem of \cite{Landsburg Duke}, there is a perfect complex $\tilde{\mathcal{F}}_m$ on $D_m:= Y_m \coprod_{Y} Y_m$ whose derived restrictions to $Y_m$ are quasi-isomorphic to $\mathcal{F}_{i, m}$ for $i=1,2$. For each closed immersion $\iota_{m/ (m+1)}: D_m \hookrightarrow D_{m+1}$, we have $\iota_{m/(m+1)}^* \tilde{\mathcal{F}}_{m+1} \simeq \tilde{\mathcal{F}}_{m}$. We then take $\tilde{\mathcal{F}} := \mathbf{R}  \underset{m}{\varprojlim} \  \tilde{\mathcal{F}}_m$, which is a perfect complex on $D_{\widehat{X}}$ by \cite[Lemma 0CQG]{stacks} (or B. Bhatt \cite[Lemma 4.2]{Bhatt}). One checks that this is a desired perfect complex.
\end{proof}

\subsection{A mod equivalent pair of sheaves}

For the cycles on formal schemes over $\widehat{X}$, in \cite[\S 3]{P general}, the notion of ``mod $Y$-equivalence" was defined. There, it was first defined for a pair $\mathcal{A}_1, \mathcal{A}_2$ of coherent $\mathcal{O}_{\widehat{X} \times \square^n}$-\emph{algebras}, and then we took the subgroup $\mathcal{M}^q (\widehat{X}, Y, n)$ of the differences of the associated cycles $[\mathcal{A}_1]-[\mathcal{A}_2]$. The quotient $z^q (\widehat{X}, n)/ \mathcal{M}^q (\widehat{X}, Y, n)$ defined the equivalence on cycles. 

This process used the language of the derived algebraic geometry in terms of $\infty$-categories. For some basics of derived algebraic geometry, one can look at J. Lurie \cite{Lurie} and B. To\"en \cite{Toen}, while for $\infty$-categories, we refer the reader to J. Lurie \cite{Lurie 2}.

\medskip

In this paper, we would like to try its $K$-theoretic counterpart. After some attempts, we realized that we should work with coherent \emph{modules} unlike in the case of cycles, while we can proceed without the language of derived algebraic geometry, and just stick to the classical homological language of derived categories, together with the derived Milnor patching recalled in \S \ref{sec:derived Milnor}.

We start with the following (cf. \cite[\S 3.2]{P general}):

\begin{defn}\label{defn:adm pair}
Let $Y$ be an affine $k$-scheme of finite type. Let $Y \hookrightarrow X$ be a closed immersion into an equidimensional smooth $k$-scheme, and let $\widehat{X}$ be the completion of $X$ along $Y$. Recall we had the notion of admissible coherent $\mathcal{O}_{\widehat{X}\times \square^n}$-modules in Definition \ref{defn:cnv higher}.

 Two admissible coherent $\mathcal O _{\widehat{X} \times \square^n}$-modules (resp. of codimension $ \geq q$) 
$\mathcal{M}_1, \mathcal{M}_2$ 
are said to be \emph{mod $Y$-equivalent} if we have an isomorphism
\begin{equation}\label{eqn:adm pair}
\mathcal{M}_1 \otimes _{\mathcal{O}_{\widehat{X} \times \square^n}} ^{\mathbf{L}} {\mathcal{O}_{Y \times \square^n}} \simeq \mathcal{M}_2 \otimes _{\mathcal{O}_{\widehat{X} \times \square^n}} ^{\mathbf{L}} {\mathcal{O}_{Y \times \square^n}}
\end{equation}
in the derived category of $\mathcal{O}_{Y \times \square^n}$-modules. Let $\mathcal{L} (\widehat{X}, Y, n)$ (resp. $\mathcal{L} ^{  \geq q} (\widehat{X}, Y, n)$) be the set of all pairs of mod $Y$-equivalent admissible coherent 
$\mathcal O _{\widehat{X} \times \square^n}$-modules  (resp. of codimension $\geq q$). \qed
\end{defn}

There are two notable points here. The first is that $\widehat{X} \times \square^n$ is a regular formal scheme, so that by Lemma \ref{lem:coh=perfect}, a coherent $\mathcal O _{\widehat{X} \times \square^n}$-module is a perfect complex on $\widehat{X} \times \square^n$. The second point is that \eqref{eqn:adm pair} is a quasi-isomorphism of the derived pull-backs to $Y \times \square^n$ of perfect complexes on $\mathcal{O}_{\widehat{X} \times \square^n}$, so that the derived Milnor patching discussed in Lemma \ref{lem:derived Milnor} applies. We state it as follows:

\begin{cor}\label{cor:pair derived Milnor}
Under the notations and the assumptions of Definition \ref{defn:adm pair}, there exists a perfect complex $\tilde{\mathcal{M}}$ on $D_{\widehat{X}} \times \square^n$ whose derived pull-backs $\mathbf{L} j_i ^* \tilde{\mathcal{M}}$ via the two closed immersions $j_i: \widehat{X} \times \square^n \hookrightarrow D_{\widehat{X}} \times \square^n$ for $i=1,2$ are quasi-isomorphic to $ \mathcal{M}_i$ for $i=1,2$.
\end{cor}

We collect all such $\tilde{\mathcal{M}}$:

\begin{defn}\label{defn:patched perfect}
Under the above notations and the assumptions:
\begin{enumerate}
\item Let $\tilde{\mathcal{L}}(D_{\widehat{X}}, Y, n)$ (resp. $\tilde{\mathcal{L}}^{\geq q} (D_{\widehat{X}}, Y, n)$) be the set of perfect complexes $\tilde{\mathcal{M}}$ (resp. of codimension $\geq q$) on $D_{\widehat{X}} \times \square^n$ obtained in Corollary \ref{cor:pair derived Milnor}. 

\item Let $\mathcal{V} (D_{\widehat{X}}, Y, n)$ (resp. $\mathcal{V}^q (D_{\widehat{X}}, Y, n)$) be the triangulated subcategory of $\mathcal{D}_{\coh} (D_{\widehat{X}} \times \square^n)$ generated by the perfect complexes in $\tilde{\mathcal{L}}(D_{\widehat{X}}, Y, n)$ (resp. $\tilde{\mathcal{L}}^{\geq q} (D_{\widehat{X}}, Y, n)$).

\item Let $\mathcal{S} (D_{\widehat{X}}, Y, n)$ (resp. $\mathcal{S} ^q (\widehat{X}, Y, n)$) be the Waldhausen $K$-space of the triangulated category $\mathcal{V} (\widehat{X}, Y, n)$ (resp. $\mathcal{V}^q (\widehat{X}, Y, n)$).
\qed
\end{enumerate}
\end{defn}

The two closed immersions $j_i : \widehat{X} \times \square^n \hookrightarrow D_{\widehat{X}} \times \square^n$ for $i=1,2$ induce the following commutative diagram of functors, where the pull-backs are the derived pull-backs of perfect complexes:
$$
\xymatrix{
\mathcal{V} ^{  q} (D_{\widehat{X}}, Y, n) \ar@2{->}[r] ^{j_1^*}_{j_2^*} \ar[d] &  \mathcal{D}_{\coh} ^q ( \widehat{X}, n) \ar[d]  \\
 \mathcal{V} (D_{\widehat{X}}, Y, n)  \ar@2{->}[r] ^{j_1^*}_{j_2^*} &  \mathcal{D}_{\coh} ( \widehat{X}, n).}
$$
This diagram in turn induces the commutative diagram of the spaces
\begin{equation}\label{eqn:hocoeq0}
\begin{array}{c}
\xymatrix{
\mathcal{S} ^{  q} (D_{\widehat{X}}, Y, n)  \ar@2{->}[r] ^{j_1^*}_{j_2^*} \ar[d] &  \mathcal{G} ^q ( \widehat{X}, n) \ar[d]  \\
 \mathcal{S} (D_{\widehat{X}}, Y, n)  \ar@2{->}[r] ^{j_1^*}_{j_2^*} &  \mathcal{G} ( \widehat{X}, n).}
 \end{array}
 \end{equation}

\begin{defn}\label{defn:K-sp mod Y}
Take the homotopy coequalizers of the horizontal maps of \eqref{eqn:hocoeq0} to define
$$
\mathcal{G}^q (\widehat{X}, Y, n):= {\rm hocoeq} \left( \mathcal{S} ^{ q} (D_{\widehat{X}}, Y, n) \rightrightarrows   \mathcal{G} ^q ( \widehat{X}, n) \right) ,
$$
$$
\mathcal{G} (\widehat{X}, Y, n):= {\rm hocoeq} \left( \mathcal{S}  (D_{\widehat{X}}, Y, n) \rightrightarrows   \mathcal{G} ( \widehat{X}, n) \right).
$$
We can regard them as the ``Waldhausen $K$-spaces of $\widehat{X}$ mod $Y$" in a sense.
\qed
\end{defn}

\begin{remk}  \label{rem. modY trivial}
If $Y$ is a smooth affine $k$-scheme of finite type, then we can take $X=Y$ in Lemma \ref{lem:derived Milnor}.

In this case, we claim that we have $\mathcal{G}^q (\widehat{X}, Y, n)=\mathcal{G} ^q ( Y, n)$.

 Indeed, in this case we get $\widehat X =Y$, so $D_{\widehat X}= \widehat{X} \coprod_Y \widehat{X} = Y$ and all the maps in the cartesian diagram in \eqref{lem:derived Milnor} are equal to the identity on $Y$.

For two admissible coherent $\mathcal{O}_{Y\times \square^n}$-modules $\mathcal{M}_1$ and $\mathcal{M}_2$, they are mod $Y$-equivalent in the sense of Definition \ref{defn:adm pair} if and only if $\mathcal{M}_1 \simeq \mathcal{M}_2$ as $\mathcal{O}_{Y \times \square^n}$-modules. 

Hence we have $\mathcal{V}^q (D_{\widehat{X}}, Y, n) = \mathcal{D}^q _{\coh} (Y, n)$. In particular, applying the Waldhausen $K$-spaces, we deduce that $\mathcal{S} ^{ q} (D_{\widehat{X}}, Y, n)=\mathcal{G} ^q ( Y, n)$.

Since $j_1^\ast =j_2 ^\ast ={\rm Id}$ in \eqref{eqn:hocoeq0}, we deduce that $\mathcal{G}^q (\widehat{X}, Y, n) = \mathcal{G}^q (Y, n)$, as desired. Similarly, we have $\mathcal{G} (\widehat{X}, Y, n) = \mathcal{G} (Y, n)$.

We will use this observation later when we compare our constructions of the filtrations on $K$-theory with the pre-existing ones in the literature when $Y$ is smooth. 
\qed
\end{remk}

\medskip

Coming back to Definition \ref{defn:K-sp mod Y}, by construction, we deduce the morphisms in the homotopy category
$$
\mathcal{G}^{q+1} (\widehat{X}, Y, n) \to \mathcal{G}^q (\widehat{X}, Y, n) \to \mathcal{G} (\widehat{X}, Y, n).
$$
We also note that there are natural morphisms in the homotopy category
$$
\mathcal{G}^q (\widehat{X} , n) \to \mathcal{G}^0( \widehat{X}, n) \to \mathcal{G} (\widehat{X} \times \square^n) \overset{\sim}{ \leftarrow} \mathcal{G} (\widehat{X}) \overset{\sim}{ \leftarrow }\mathcal{K} (\widehat{X}) \to \mathcal{K} (Y).
$$
Its composition with both of the arrows $j_i^*$ in \eqref{eqn:hocoeq0} are equal. Thus by the definition of the homotopy coequalizers, we deduce the commutative diagram of spaces in the homotopy category
$$\xymatrix{
\mathcal{G}^q (\widehat{X}, Y, n)  \ar[d] \ar[dr] &  \\
\mathcal{G}(\widehat{X}, Y, n) \ar[r] &  \mathcal{K} (Y).}
$$

These $\mathcal{G}^q (\widehat{X}, Y, n) $ and $\mathcal{G}(\widehat{X}, Y, n)$ over $n \geq 0$ form cubical spaces.

\begin{defn}\label{defn:mod Y space realize}
We define 
$$
\mathcal{G}^q (\widehat{X}, Y):= | \un{n}\mapsto \mathcal{G}^q (\widehat{X}, Y, n)|,
$$
 $$
 \mathcal{G} (\widehat{X}, Y):= | \un{n}\mapsto \mathcal{G} (\widehat{X}, Y, n)|,
 $$
 the geometric realizations.
 \qed
\end{defn}

By construction, we have the induced morphisms
$$
\mathcal{G}^{q+1} (\widehat{X}, Y) \to \mathcal{G}^q (\widehat{X}, Y) \to \mathcal{G} (\widehat{X}, Y) \to \mathcal{K} (Y)
$$
in the homotopy category.

\subsection{Moving lemma and general pull-backs}

Combining the above construction with the discussions in \S \ref{sec:cnv 1}, we would like to have the ``mod $Y$-equivalence" version of Theorem \ref{thm:moving pull-back}. This is eventually proven in Theorem \ref{thm:moving pull-back mod Y} below. Along the way, we go over part of the constructions in \S \ref{sec:cnv 1} to check that they respect the mod $Y$-equivalences. We discuss these first.

\medskip

The following is a minor variant of Definition \ref{defn:K-sp mod Y} for the triangulated subcategory $\mathcal{D}_{\coh, \{ \widehat{X}_1 \}} ^q (\widehat{X}_2, n) \subset \mathcal{D}^q _{\coh} (\widehat{X}_2, n)$ (see Definition \ref{defn:cnv X1}):

\begin{defn}\label{defn:K-sp mod Y var}
Take the homotopy coequalizer to define
$$
\mathcal{G}^q _{\{ \widehat{X}_1 \}} (\widehat{X}_2, Y, n):= {\rm hocoeq} \left( \mathcal{S} ^{ q} (D_{\widehat{X}_2}, Y, n) \rightrightarrows \mathcal{G} ^{ q} (\widehat{X}_2, n)  \overset{w.e.}{\leftarrow} \mathcal{G}_{\{\widehat{X}_1\}} ^{ q} (\widehat{X}_2, n) \right),
$$
where $w.e.$ is a weak-equivalence given by Corollary \ref{cor:Gysin cnv}. Similarly we also define
$\mathcal{G}_{\{ \widehat{X}_1 \}} (\widehat{X}_2, Y, n).
$

As before, we take the geometric realizations to define
$$
\mathcal{G}_{\{ \widehat{X}_1 \}} ^q (\widehat{X}_2, Y):= | \un{n} \mapsto \mathcal{G}^q _{\{ \widehat{X}_1 \}} (\widehat{X}_2, Y, n)|,
$$
and similarly we define $\mathcal{G}_{ \{ \widehat{X}_1 \}} (\widehat{X}_2, Y)$. 
\qed
\end{defn}

The following is the mod $Y$-version of Corollary \ref{cor:Gysin cnv} for the $K$-spaces mod $Y$: 

\begin{lem}\label{lem:moving-pull-back mod version}
Let $Y$ be an affine $k$-scheme of finite type. Suppose we have a sequence of closed immersions $Y \hookrightarrow X_1 \hookrightarrow X_2$, where $X_1, X_2$ are equidimensional smooth $k$-schemes. They induce closed immersions $Y \hookrightarrow \widehat{X}_1 \overset{\widehat{\iota}}{\hookrightarrow} \widehat{X}_2$, where $\widehat{X}_i$ is the completion of $X_i$ along $Y$. 

Then for $n \geq 0$, the zigzag of functors
\begin{equation}\label{eqn:mpb mod 1}
\mathcal{D}_{\coh} ^q (\widehat{X}_2,n) \overset{\mathfrak{i}}{\hookleftarrow} \mathcal{D}_{\coh, \{ \widehat{X}_1 \}} ^q (\widehat{X}_2,n ) \overset{\widehat{\iota}^\ast}{\to} \mathcal{D}_{\coh}^q (\widehat{X}_1,n),
\end{equation}
induces morphisms of the spaces
\begin{equation}\label{eqn:Gysin cnv mod 0}
\mathcal{G}^q (\widehat{X}_2, Y ,n) \overset{\mathfrak{i}}{\leftarrow} \mathcal{G}^q _{\{ \widehat{X}_1 \}} (\widehat{X}_2, Y ,n)\overset{\widehat{\iota}^*}{\to} \mathcal{G}^q (\widehat{X}_1, Y, n),
\end{equation}
where the first arrow $\mathfrak{i}$ is a weak-equivalence. Consequently, we have the Gysin morphisms of spaces in the homotopy category
\begin{equation}\label{eqn:Gysin cnv mod 1}
\tuborg
\widehat{\iota}^*: \mathcal{G}^q (\widehat{X}_2, Y,n) \to \mathcal{G}^q (\widehat{X}_1, Y,n), \ \ n \geq 0, \\
\widehat{\iota}^*: \mathcal{G}^q (\widehat{X}_2, Y) \to \mathcal{G}^q (\widehat{X}_1, Y).
\sluttuborg
\end{equation}
\end{lem}

\begin{proof}
By Proposition \ref{prop:D ess}, the first arrow $\mathfrak{i}$ of \eqref{eqn:mpb mod 1} is essentially surjective so that the morphism $\mathcal{G}^q _{\{ \widehat{X}_1 \}} (\widehat{X}_2, n) \overset{\mathfrak{i}}{\to} \mathcal{G}^q (\widehat{X}_2, n)$ is a weak-equivalence by \cite[Theorems 1.9.1, 1.9.8, p.263, p.271]{TT}. Thus the morphism $\mathcal{G} _{\{ \widehat{X}_1 \}} ^q (\widehat{X}_2, Y, n) \to \mathcal{G}^q (\widehat{X}_2, Y, n)$ (see Definition \ref{defn:K-sp mod Y var}) of the homotopy coequalizers must also be a weak-equivalence. This gives the first morphism in \eqref{eqn:Gysin cnv mod 1}.

Taking the geometric realizations of the cubical spaces over $n \geq 0$, we deduce the second morphism in \eqref{eqn:Gysin cnv mod 1}, and the lemma follows.
\end{proof}

We are now ready to prove the following:

\begin{thm}\label{thm:moving pull-back mod Y}
Let $g: Y_1 \to Y_2$ be a morphism of affine $k$-schemes of finite type. Let $Y_i \hookrightarrow X_i$ be a closed immersion into an equidimensional smooth $k$-scheme for $i=1,2$. Let $\widehat{X}_i$ be the completion of $X_i$ along $Y_i$. Suppose there is a morphism $f: X_1 \to X_2$ such that the diagram
$$
\xymatrix{
X_1 \ar[r] ^f  & X_2 \\
Y_1 \ar@{^{(}->}[u] \ar[r] ^g & Y_2  \ar@{^{(}->}[u]}
$$
commutes. Then for $q \geq 0$, $n\geq 0$, there exist the induced morphisms in the homotopy category
\begin{equation}\label{eqn:moving pull-back mod Y f}
\widehat{f}^*: \mathcal{G}^q (\widehat{X}_2 , Y_2,n) \to \mathcal{G}^q (\widehat{X}_1, Y_1,n),  \; \;
\widehat{f}^*: \mathcal{G}^q (\widehat{X}_2 , Y_2) \to \mathcal{G}^q (\widehat{X}_1, Y_1),
\end{equation}
where $\mathcal{G} ^q (\widehat{X}, Y, n)$ and $\mathcal{G}^q (\widehat{X} , Y)$ are as in Definitions \ref{defn:K-sp mod Y} and \ref{defn:mod Y space realize}.
\end{thm}

\begin{proof}
 We are going to repeat part of the argument of Theorem \ref{thm:moving pull-back}. We continue to use the diagrams and the notations there.

From the commutative diagram of formal schemes \eqref{eqn:moving pull-back1}, we deduced the pull-back functors in \eqref{eqn:moving pull-back2} 
\begin{equation}\label{eqn:moving pull-back3}
\tuborg
\widehat{pr}_2^*: \mathcal{D}_{\coh} ^q (\widehat{X}_2,n) \to \mathcal{D}_{\coh} ^q (\widehat{X}_{12},n),\\
\alpha^*: \mathcal{D}_{\coh} ^q (\widehat{X}_{12},n) \to \mathcal{D}_{\coh} ^q (\widehat{X}_{12}',n),
\sluttuborg
\end{equation}
and the morphisms of spaces in \eqref{eqn:alpha pr}
\begin{equation}\label{eqn:alpha pr2}
\tuborg  \widehat{pr}^*_2 : \mathcal{G}^q (\widehat{X}_2,n) \to \mathcal{G}^q (\widehat{X}_{12},n), \\
 \alpha^* :\mathcal{G}^q (\widehat{X}_{12},n) \to \mathcal{G}^q (\widehat{X}_{12}',n).\sluttuborg
\end{equation}

\medskip

We need to check that the functors in \eqref{eqn:moving pull-back3} respect the mod equivalences.

Suppose we are given a pair $(\mathcal{M}_1, \mathcal{M}_2)$ of coherent 
$\mathcal{O}_{\widehat{X}_2 \times \square^n}$-modules such that we have an isomorphism in the derived category of $\mathcal{O}_{Y_2 \times \square^n}$-modules
\begin{equation}\label{eqn:moving pull-back3 mod Y}
\mathcal{M}_1 \otimes_{\mathcal{O}_{\widehat{X}_2 \times \square^n}} ^{\mathbf{L}} \mathcal{O}_{Y_2 \times \square^n} \simeq \mathcal{M}_2 \otimes_{\mathcal{O}_{\widehat{X}_2 \times \square^n}} ^{\mathbf{L}} \mathcal{O}_{Y_2 \times \square^n}.
\end{equation}
Since $\widehat{pr}_2$ is flat, applying the flat pull-back $\widehat{pr}_2^*$ to \eqref{eqn:moving pull-back3 mod Y}, we deduce an isomorphism
$$
\widehat{pr}_2^* (\mathcal{M}_1) \otimes_{\widehat{pr}_2 ^* \mathcal{O}_{\widehat{X}_2 \times \square^n}}  ^{\mathbf{L}}  \widehat{pr}_2 ^* \mathcal{O}_{Y_2 \times \square^n} \simeq \widehat{pr}_2^* (\mathcal{M}_2) \otimes_{\widehat{pr}_2 ^* \mathcal{O}_{\widehat{X}_2 \times \square^n}} ^{\mathbf{L}}  \widehat{pr}_2 ^* \mathcal{O}_{Y_2 \times \square^n},
$$
which implies
\begin{equation}\label{eqn:moving pull-back3 mod Y 2}
\widehat{pr}_2^* (\mathcal{M}_1) \otimes_{\mathcal{O}_{\widehat{X}_{12} \times \square^n}} ^{\mathbf{L}} \widehat{pr}_2 ^* \mathcal{O}_{Y_2 \times \square^n} \simeq \widehat{pr}_2^* (\mathcal{M}_2) \otimes_{ \mathcal{O}_{\widehat{X}_{12} \times \square^n}} ^{\mathbf{L}}   \widehat{pr}_2 ^* \mathcal{O}_{Y_2 \times \square^n}
\end{equation}
via the natural morphism $\widehat{pr}_2 ^* \mathcal{O}_{\widehat{X}_2 \times \square^n} \to \mathcal{O}_{\widehat{X}_{12} \times \square^n}$. We also have the natural morphism $\widehat{pr}_2 ^* \mathcal{O}_{Y_2 \times \square^n} \to \mathcal{O}_{Y_{12} \times \square^n}$, where $Y_{12}= Y_1 \times Y_2$. Applying the derived functor $(-) \otimes_{\widehat{pr}_2 ^* \mathcal{O}_{Y_2 \times \square^n} } ^{\mathbf{L}}  \mathcal{O}_{Y_{12} \times \square^n}$ to \eqref{eqn:moving pull-back3 mod Y 2}, we deduce an isomorphism in the derived category of $\mathcal{O}_{Y_{12} \times \square^n}$-modules
$$
\widehat{pr}_2^* (\mathcal{M}_1) \otimes_{\mathcal{O}_{\widehat{X}_{12} \times \square^n}} ^{\mathbf{L}} \mathcal{O}_{Y_{12} \times \square^n} \simeq \widehat{pr}_2^* (\mathcal{M}_2) \otimes_{ \mathcal{O}_{\widehat{X}_{12} \times \square^n}} ^{\mathbf{L}}   \mathcal{O}_{Y_{12} \times \square^n}.
$$
Thus, $\widehat{pr}_2 ^*$ of \eqref{eqn:moving pull-back3} induces the set map $\mathcal{L}^{\geq q} (\widehat{X}_2, Y_2, n) \to \mathcal{L}^{\geq q} (\widehat{X}_{12}, Y_{12}, n)$, which in turn induces
$$
\widehat{pr}_2^*: \mathcal{V}^{ q} (D_{\widehat{X}_{2}}, Y_{2}, n) \to \mathcal{V} ^{ q} (D_{\widehat{X}_{12}}, Y_{12}, n).
$$
Applying the Waldhausen $K$-spaces, together with \eqref{eqn:alpha pr2}, we deduce
$$
\xymatrix{
\mathcal{S} ^{ q} (D_{\widehat{X}_2}, Y_2, n) \ar@2{->}[r]^{j_1 ^*} _{j_2^*}  \ar[d] ^{\widehat{pr}_2^*} & \mathcal{G}^q (\widehat{X}_2, n) \ar[d] ^{\widehat{pr}_2 ^*} \\
\mathcal{S} ^{ q} (D_{\widehat{X}_{12}}, Y_{12}, n) \ar@2{->}[r] ^{j_1 ^*} _{j_2^*} & \mathcal{G}^q (\widehat{X}_{12}, n),
}
$$
and taking the homotopy coequalizers of the horizontal arrows, for each $n \geq 0$, we deduce the morphism in the homotopy category 
\begin{equation}\label{eqn:moving pull-back mod Y pr_2 0}
\widehat{pr}_2 ^*: \mathcal{G}^q (\widehat{X}_2, Y_2, n) \to \mathcal{G}^q (\widehat{X}_{12}, Y_{12}, n).
\end{equation}
The above \eqref{eqn:moving pull-back mod Y pr_2 0} over all $n \geq 0$ forms a morphism of cubical objects in the homotopy category. Taking the geometric realizations, we deduce the morphism in the homotopy category
\begin{equation}\label{eqn:moving pull-back mod Y pr_2}
\widehat{pr}_2^*: \mathcal{G}^q (\widehat{X}_2, Y_2) \to \mathcal{G}^q (\widehat{X}_{12}, Y_{12}).
\end{equation}
Since the completion morphism $\alpha$ in \eqref{eqn:moving pull-back1} is flat, we similarly deduce
\begin{equation}\label{eqn:moving pull-back mod Y alpha}
\tuborg
\alpha^*:  \mathcal{G}^q (\widehat{X}_{12}, Y_{12},n) \to \mathcal{G}^q (\widehat{X}_{12}', Y_1,n), \\
\alpha^*:  \mathcal{G}^q (\widehat{X}_{12}, Y_{12}) \to \mathcal{G}^q (\widehat{X}_{12}', Y_1).
\sluttuborg
\end{equation}

On the other hand, by Lemma \ref{lem:moving-pull-back mod version}, from $\widehat{\iota}$ in \eqref{eqn:moving pull-back1}, we deduce the following morphisms in the homotopy category
\begin{equation}\label{eqn:moving pull-back mod Y iota}
\tuborg
\widehat{\iota}^*: \mathcal{G}^q (\widehat{X}_{12}' , Y_1, n) \to \mathcal{G}^q (\widehat{X}_1, Y_1,n), \\
\widehat{\iota}^*: \mathcal{G}^q (\widehat{X}_{12}' , Y_1) \to \mathcal{G}^q (\widehat{X}_1, Y_1).
\sluttuborg
\end{equation}

Composing \eqref{eqn:moving pull-back mod Y pr_2}, \eqref{eqn:moving pull-back mod Y alpha} and \eqref{eqn:moving pull-back mod Y iota}, we thus deduce the desired morphisms \eqref{eqn:moving pull-back mod Y f} as
$$
\tuborg
\widehat{f}^*=\widehat{pr}_2^* \circ \alpha^* \circ  \widehat{\iota}^* : \mathcal{G}^q (\widehat{X}_2, Y_2,n) \to \mathcal{G}^q (\widehat{X}_1, Y_1,n), \\
\widehat{f}^*=\widehat{pr}_2^* \circ \alpha^* \circ  \widehat{\iota}^* : \mathcal{G}^q (\widehat{X}_2, Y_2) \to \mathcal{G}^q (\widehat{X}_1, Y_1).
\sluttuborg
$$
This proves the theorem.
\end{proof}

We can compose the above morphisms, too:

\begin{lem}\label{lem:moving pull-back composed mod}
Let $g_1: Y_1 \to Y_2$ and $g_2: Y_2 \to Y_3$ be morphisms of affine $k$-schemes of finite type. Suppose we have closed immersions $Y_i \hookrightarrow X_i$ into equidimensional smooth affine $k$-schemes for $i=1,2,3$, and there are morphisms $f_1: X_1 \to X_2$ and $f_2: X_2 \to X_3$ that form a commutative diagram
$$
\xymatrix{
X_1 \ar[r] ^{f_1} & X_2 \ar[r] ^{f_2} & X_3 \\
Y_1 \ar[r] ^{g_1} \ar@{^{(}->}[u] & Y_2 \ar[r] ^{g_2}  \ar@{^{(}->}[u]  & Y_3.  \ar@{^{(}->}[u]}
$$
Let $\widehat{X}_i$ be the completion of $X_i$ along $Y_i$ for $i=1,2,3$, and let 
$$
\xymatrix{
\widehat{X}_1 \ar[r] ^{\widehat{f}_1} & \widehat{X}_2 \ar[r] ^{\widehat{f}_2} & \widehat{X}_3 \\
Y_1 \ar[r] ^{g_1} \ar@{^{(}->}[u] & Y_2 \ar[r] ^{g_2}  \ar@{^{(}->}[u]  & Y_3.  \ar@{^{(}->}[u]}
$$
be the induced commutative diagram.

Then we have
\begin{equation}\label{eqn:pull-back transitive mod}
\tuborg
(\widehat{f}_2 \circ \widehat{f}_1)^* &= \widehat{f}_1 ^* \circ \widehat{f}_2^*: \mathcal{G}^q (\widehat{X}_3, Y_3,n) \to \mathcal{G}^q (\widehat{X}_1, Y_1,n), \\
(\widehat{f}_2 \circ \widehat{f}_1)^* &= \widehat{f}_1 ^* \circ \widehat{f}_2^*: \mathcal{G}^q (\widehat{X}_3, Y_3) \to \mathcal{G}^q (\widehat{X}_1, Y_1)
\sluttuborg
\end{equation}
in the homotopy category.
\end{lem}

\begin{proof}
Once we know that morphisms in \eqref{eqn:pull-back transitive mod} exist by Theorem \ref{thm:moving pull-back mod Y}, the equality follows from Lemma \ref{lem:moving pull-back composed} immediately.
\end{proof}

 \section{The filtrations}\label{sec:Cech}

 In \S \ref{sec:Cech}, we use a version of the \v{C}ech machine applied to the mod version of the Waldhausen $K$-spaces of Definitions \ref{defn:K-sp mod Y} and \ref{defn:mod Y space realize}, to define the desired filtrations on the algebraic $K$-theory of an arbitrary scheme $Y \in \Sch_k$.
 
 For our purposes, the \v{C}ech hypercohomology machine for presheaves as in R. Thomason \cite[\S 1]{Thomason etale}, is a bit insufficient. We use a structured version of \v{C}ech covers, called \emph{systems of local embeddings}, minted by R. Hartshorne \cite{Hartshorne DR} in his studies of algebraic de Rham cohomology. Unlike the work of Thomason, we limit our discussions only to algebraic $K$-theory.

 \subsection{The extended \v{C}ech construction}
 
The following is a minor variant of the notion from \cite[Remark, p.28]{Hartshorne DR}. The version here is used in \cite{P general} as well:

\begin{defn}\label{defn:system}
Let $Y$ be a $k$-scheme of finite type. A finite collection $\mathcal{U}:= \{ (U_i, X_i)\}_{i \in \Lambda}$ consisting of 
\begin{enumerate}
\item the underlying quasi-affine open cover $|\mathcal{U}|:= \{ U_i \}_{i \in \Lambda}$ of $Y$, and
\item a closed immersion $U_i \hookrightarrow X_i$ into an equidimensional smooth $k$-scheme for each $i \in \Lambda$,
\end{enumerate}
is called a \emph{system of local embeddings} for $Y$.
\qed
\end{defn}

N.B. For the rest of the paper, we will use the systems $\mathcal{U}$ such that \emph{each $U_i$ is affine only}, though often we don't say so explicitly.

\medskip

For $p \geq 0$, let $I = (i_0, \cdots, i_{p})\in \Lambda^{p+1}$. We define $U_I:= U_{i_0} \cap \cdots \cap U_{i_{p}}$ and $X_I:= X_{i_0} \times \cdots \times X_{i_{p}}$. Consider the diagonal closed immersion $U_I \hookrightarrow X_I$. Let $\widehat{X}_I$ be the completion of $X_I$ along $U_I$. Since each $U_{i_j}$ is affine, so is $U_I$, because $Y$ is assumed to be separated over $k$.

For the (usual) open cover $|\mathcal{U}|$ of $Y$, the usual \v{C}ech cosimplicial space in the sense of R. Thomason \cite[\S 1]{Thomason etale} is defined to be:
\begin{equation}\label{eqn:usual cech K}
{\mathcal{K}} ^{\bullet} _{|\mathcal{U}|} := \left \{ \prod_{i_0 \in \Lambda} \mathcal{K} (U_{i_0}) \overset{\leftarrow}{ \mathrel{\substack{\textstyle\rightarrow\\[-0.6ex]
                      \textstyle\rightarrow }}}   \prod_{I \in \Lambda^2} \mathcal{K} (U_I)
 \overset{\leftarrow}{ \mathrel{\substack{\textstyle\rightarrow\\[-0.6ex]
                      \textstyle\rightarrow \\[-0.6ex]
                      \textstyle\rightarrow}} } \prod_{I \in \Lambda^3} \mathcal{K} (U_I) \cdots \right \}.
\end{equation}
Here, when $U \subset Y$ is open, $\mathcal{K}(U)$ is the usual Waldhausen algebraic $K$-theory space of the scheme $U$, i.e. it is the Waldhausen $K$-space of the triangulated category of perfect complexes on $U$.

\medskip

We want to have its parallel construction for $\mathcal{U}$, involving the completions $\widehat{X}_I$ over the multi-indices $I \in \Lambda^{p+1}$ and $p \geq 0$:

\medskip

For each $ 0 \leq j \leq p$, let $I_j ':= (i_0, \cdots, \check{i}_j, \cdots, i_{p}) \in \Lambda^p$, i.e. we omit $i_j$. The open inclusion $U_I \hookrightarrow U_{I_j'}$ and the projection $X_I \to X_{I_j'}$ form the commutative diagram
$$
\xymatrix{
X_I \ar[r] & X_{I_j'} \\
U_I \ar[r] \ar@{^{(}->}[u] & U_{I_j'}. \ar@{^{(}->}[u]
}
$$
By Theorem \ref{thm:moving pull-back mod Y}, this diagram induces the morphism
\begin{equation}\label{eqn:coface 1}
\mathcal{G}^q ( \widehat{X}_{I_j'}, U_{I_j'}) \to \mathcal{G}^q ( \widehat{X}_I, U_I)
\end{equation}
 in the homotopy category, where $\mathcal{G}^q (-,-)$ are as in Definition \ref{defn:mod Y space realize}.

Thus for $0 \leq j \leq p$, we have the coface maps in the homotopy category
\begin{equation}\label{eqn:coface 2}
\delta_j: \prod_{I \in \Lambda^p} \mathcal{G}^q (\widehat{X}_I, U_I) \to \prod_{I \in \Lambda^{p+1}} \mathcal{G}^q (\widehat{X}_I, U_I),
\end{equation}
defined as follows: for each component at $I \in \Lambda^{p+1}$ of the right hand side, $\delta_j$ is given first by taking the projection to the factor $\mathcal{G}^p (\widehat{X}_{I_j'}, U_{I_j'} )$ for $I_j' \in \Lambda^p$ of the left hand side, and then applying \eqref{eqn:coface 1}.

\medskip

We have the codegeneracies as well: for a given $I=(i_0, \cdots, i_{p-1}) \in \Lambda ^{p}$ and for each $0 \leq j \leq p$, consider $I_j '':= (i_0, \cdots, i_j, i_j, \cdots, i_{p-1}) \in \Lambda^{p+1}$, i.e. we insert a copy of $i_j$ at the $(j+1)$-th place, and shift the remaining ones to the next position. In this case $U_{I} = U_{I_j''}$, while we have the diagonal closed immersion $X_I \hookrightarrow X_{I_j''} $ induced from $X_{i_j} \hookrightarrow X_{i_j} \times X_{i_j}$. They form the commutative diagram
$$
\xymatrix{
X_{I} \ar@{^{(}->}[r] & X_{I_j ''} \\
U_I \ar@{=}[r] \ar@{^{(}->}[u] & U_{I_j''}. \ar@{^{(}->}[u]
}
$$
Thus by the Theorem \ref{thm:moving pull-back mod Y}, this diagram induces the morphism
\begin{equation}\label{eqn:codegeneracy 1}
\mathcal{G}^q (\widehat{X}_{I_j''}, U_{I_j ''}) \to \mathcal{G}^q (\widehat{X}_I, U_I)
\end{equation}
 in the homotopy category. Thus for $0 \leq j \leq p-1$, we have the codegeneracy maps
\begin{equation}\label{eqn:codegeneracy 2}
s_j: \prod_{I \in \Lambda^{p+1}} \mathcal{G} ^q (\widehat{X}_I, U_I) \to \prod_{I \in \Lambda^p} \mathcal{G}^q (\widehat{X}_I, U_I)
\end{equation}
 in the homotopy category, which is defined as follows: for the component at $I \in \Lambda^p$ of the right hand side, $s_j$ is given first by taking the projection to the factor $\mathcal{G}^q (\widehat{X}_{I_j ''}, U_{I_j ''})$ for $I_j'' \in \Lambda^{p+1}$ of the left hand side, and then applying \eqref{eqn:codegeneracy 1}.

One can check straightforwardly (but with tedious calculations) that the above cofaces in \eqref{eqn:coface 2} and codegeneracies \eqref{eqn:codegeneracy 2} indeed satisfy the axioms of a cosimplicial object in the homotopy category of spaces. So, we define:

\begin{defn}\label{defn:cosimplicial}
Let $Y\in \Sch_k$. We have the following cosimplicial space in the homotopy category
\begin{equation}\label{eqn:complete cech K}
\widehat{\mathcal{G}} ^{q, \bullet} _{\mathcal{U}} := \left \{ \prod_{i_0 \in \Lambda} \mathcal{G}^q (\widehat{X}_{i_0}, U_{i_0}) \overset{\leftarrow}{ \mathrel{\substack{\textstyle\rightarrow\\[-0.6ex]
                      \textstyle\rightarrow }}}   \prod_{I \in \Lambda^2} \mathcal{G}^q (\widehat{X}_I, U_I)
 \overset{\leftarrow}{ \mathrel{\substack{\textstyle\rightarrow\\[-0.6ex]
                      \textstyle\rightarrow \\[-0.6ex]
                      \textstyle\rightarrow}} } \prod_{I \in \Lambda^3} \mathcal{G}^q (\widehat{X}_I, U_I) \cdots \right \}.
\end{equation}

Let $p \geq 0$.
Note that for each $I \in \Lambda^{p+1}$, we have the closed immersion $U_I \hookrightarrow \widehat{X}_I$,
and it induces $\mathcal{G}^q (\widehat{X}_I, U_I )  \rightarrow \mathcal{K}(U_I)$ in the homotopy category. 
Comparing \eqref{eqn:usual cech K} and \eqref{eqn:complete cech K} over various $q \geq 0$, we obtain the induced morphisms of cosimplicial objects in the homotopy category
\begin{equation}\label{eqn:complete usual cech K}
\cdots \to \widehat{\mathcal{G}} ^{q, \bullet}_{\mathcal{U}} \to \widehat{\mathcal{G}} ^{q-1, \bullet}_{\mathcal{U}} \to \cdots \to \widehat{\mathcal{G}} ^{ 0, \bullet}_{\mathcal{U}} \to  \mathcal{K} ^{\bullet} _{|\mathcal{U}|},
\end{equation}
where the last $\mathcal{K}^{\bullet}_{|\mathcal{U}|}$ is as in \eqref{eqn:usual cech K}.
\qed
\end{defn}

\begin{defn}\label{defn:Cech hypercohomology}
Let $Y\in \Sch_k$ and let $\mathcal{U}$ be a system of local embeddings for $Y$. For the above cosimplicial objects $\widehat{\mathcal{G}} ^{q, \bullet}_{\mathcal{U}}$ in Definition \ref{defn:cosimplicial}, define the \emph{\v{C}ech hypercohomology space} with respect to the system $\mathcal{U}$ to be
$$
\check{\mathbb{H}}^{\cdot} (\mathcal{U}, \widehat{\mathcal{G}}^q):= \underset{\Delta}{\holim}\  \widehat{\mathcal{G}} ^{q, \bullet} _{\mathcal{U}}.
$$
Recall that for \eqref{eqn:usual cech K}, we similarly had (see R. Thomason \cite[(1.5), p.443]{Thomason etale})
$$
\check{\mathbb{H}}^{\cdot} (|\mathcal{U}|, \mathcal{K}):= \underset{\Delta}{\holim} \  \mathcal{K} ^{ \bullet} _{|\mathcal{U}|}.
$$
The morphisms \eqref{eqn:complete usual cech K} induce the tower
\begin{equation}\label{eqn:complete usual cech K hyper}
\cdots \to \check{\mathbb{H}}^{\cdot} (\mathcal{U},  \widehat{\mathcal{G}}^q) \to \check{\mathbb{H}}^{\cdot} (\mathcal{U},  \widehat{\mathcal{G}}^{q-1}) \to \cdots \to \check{\mathbb{H}}^{\cdot} (\mathcal{U},  \widehat{\mathcal{G}}^0) \to \check{\mathbb{H}}^{\cdot} (|\mathcal{U}|, \mathcal{K})
\end{equation}
in the homotopy category.
\qed
\end{defn}

 \begin{remk}\label{remk:Cech Zariski descent}
 Note that we have the natural morphism $\mathcal{K} (Y) \to \check{\mathbb{H}}^{\cdot} (|\mathcal{U}|, \mathcal{K})$, and it is a weak-equivalence. See Thomason-Trobaugh \cite[Theorem 8.4, Exercise 8.5-(a), pp.373-374]{TT}. 
 \qed
 \end{remk}
 
 \subsection{Refinements}

 For the systems of local embeddings in Definition \ref{defn:system}, we have the following notion of refinements. 
 This is more flexible than a similar notion considered in \cite[Remark, p.28]{Hartshorne DR}, and identical to the corresponding notion considered in \cite[\S 6]{P general}:
 
 \begin{defn}\label{defn:refine}
 Let $Y \in \Sch_k$. Let $\mathcal{U} = \{ (U_i, X_i) \}_{i \in \Lambda}$ and $\mathcal{V}= \{ (V_j, X_j ') \}_{ j \in \Lambda'}$ be systems of local embeddings for $Y$. We say that $\mathcal{V}$ is a refinement system of $\mathcal{U}$, if there is a set map $\lambda: \Lambda' \to \Lambda$ such that
 \begin{enumerate}
 \item for each $j \in \Lambda'$, we have $V_j \subset U_{\lambda (j)}$, and
 \item there is a morphism $f_j : X_j ' \to X_{\lambda (j)}$ that restricts to $V_j \hookrightarrow U_{\lambda (j)}.$ \qed
 \end{enumerate}
 \end{defn}
 
 Some of the basic lemmas in the subsection are also in \cite[\S 6]{P general}. We included their arguments here for self-containedness of this article.

 \begin{lem}\label{lem:comref}
 Let $Y \in \Sch_k$. Let $\mathcal{U} = \{ (U_i, X_i) \}_{i \in \Lambda}$ and $\mathcal{V}= \{ (V_j, X_j ') \}_{ j \in \Lambda'}$ be systems of local embeddings for $Y$. Suppose all $U_i$ and $V_j$ are affine open in $Y$. 
 
 Then there exists a common refinement system $\mathcal{W}$ of both $\mathcal{U}$ and $\mathcal{V}$.
 \end{lem}
 
 \begin{proof}
 
  Let $\Lambda'':= \Lambda \times \Lambda'$. For $(i, j) \in \Lambda''$, let $W_{ij}:= U_i \cap V_j$. Since $U_i, V_j$ are affine, so is $W_{ij}$. 
   
 Note that we have a closed immersion $U_i \hookrightarrow X_i$, and $W_{ij}$ is open in $U_i$. Since the subspace topology on $U_i$ from $X_i$ agrees with the topology of $U_i$, there is an open subscheme $X_{ij} \subset X_i$ such that we have the induced closed immersion $W_{ij} \hookrightarrow X_{ij}$. Similarly, $W_{ij}$ is open in $V_j$ as well, so that there is an open subscheme $X'_{ij} \subset X'_{j}$ such that we have the induced closed immersion $W_{ij} \hookrightarrow X'_{ij}$.

The two closed immersions $W_{ij} \hookrightarrow X_{ij}, X'_{ij}$ induce the diagonal closed immersion $W_{ij} \hookrightarrow X''_{ij}:= X_{ij} \times_k X'_{ij}$. Let $\mathcal{W}:= \{ (W_{ij}, X''_{ij} ) \}_{(i,j)\in \Lambda''}$. This is a system of local embeddings for $Y$.
 
 Consider the projection maps $\lambda, \lambda': \Lambda'' = \Lambda \times \Lambda' \to \Lambda, \Lambda'$ of the index sets, respectively. The corresponding morphisms
 $$
 \tuborg X''_{ij} = X_{ij} \times X_{ij}' \to X_{ij} \overset{op}{\hookrightarrow} X_i, \\
 X''_{ij}= X_{ij} \times X_{ij} ' \to  X_{ij}' \overset{op}{\hookrightarrow} X_j'
 \sluttuborg
 $$
 are given by projections followed by open immersions. Hence $\mathcal{W}$ is a refinement of both $\mathcal{U}$ and $\mathcal{V}$. 
 \end{proof}
 
 \begin{lem}\label{lem:refinement induce}
 Let $Y \in \Sch_k$. Let $\mathcal{U} = \{ (U_i, X_i) \}_{i \in \Lambda}$ and $\mathcal{V}= \{ (V_j, X_j ') \}_{ j \in \Lambda'}$ be systems of local embeddings for $Y$ such that $\mathcal{V}$ is a refinement of $\mathcal{U}$ for a set map $\lambda: \Lambda' \to \Lambda$.
 
 Then there exists a morphism of cosimplicial objects
 \begin{equation}\label{eqn:refinement map *}
 \lambda^*: \widehat{\mathcal{G}} ^{q, \bullet} _{\mathcal{U}} \to  \widehat{\mathcal{G}} ^{q, \bullet} _{\mathcal{V}}
\end{equation}
 in the homotopy category. In particular, we have the induced morphism of the \v{C}ech hypercohomology spaces in the homotopy category
 $$
 \lambda^*: \check{\mathbb{H}}^{\cdot} (\mathcal{U},  \widehat{\mathcal{G}}^q) \to  \check{\mathbb{H}}^{\cdot} (\mathcal{V},  \widehat{\mathcal{G}}^q).
 $$
 
 Furthermore, this is compatible with the morphisms $ \widehat{\mathcal{G}} ^{q, \bullet} _{ - } \to  \widehat{\mathcal{G}} ^{q-1, \bullet} _{ - } $, where $(-) = \mathcal{U}, \mathcal{V}$.
 \end{lem}
 
 \begin{proof}
 For $p \geq 0$, let $J \in ( \Lambda')^{p+1}$ so that $ \lambda (J) \in \Lambda^{p+1}$. Then we have the associated commutative diagram
 $$
 \xymatrix{ 
 X'_J \ar[r] & X_{\lambda (J)} \\
 V_J \ar[r] \ar@{^{(}->}[u] & U_{\lambda (J)}, \ar@{^{(}->}[u]
 }
 $$
 which induces (by Theorem \ref{thm:moving pull-back mod Y}) the morphism of spaces
 \begin{equation}\label{eqn:refinement map p}
 \mathcal{G} ^q (\widehat{X}_{\lambda (J)}, U_{\lambda (J)}) \to \mathcal{G}^q (\widehat{X}_J', V_J)
 \end{equation}
 in the homotopy category. They induce the morphism in the homotopy category
 $$
 \lambda_p ^*: \prod_{I \in \Lambda^{p+1}}  \mathcal{G} ^q (\widehat{X}_I, U_I ) \to \prod_{J  \in (\Lambda')^{p+1}} \mathcal{G}^q (\widehat{X}_J ', V_J),
 $$
which is defined as follows: for the $J$-th factor on the right hand side, first taking the projection to the factor at $\lambda (J) \in \Lambda^{p+1}$ on the left hand side, and then applying \eqref{eqn:refinement map p}. One checks that these morphisms over $p \geq 0$ are compatible with the cofaces and the codegeneracies. Thus we have the morphism \eqref{eqn:refinement map *} of cosimplicial objects.

The second part follows by applying the homotopy limits of \eqref{eqn:refinement map *} over $\Delta$ as in Definition \ref{defn:Cech hypercohomology}. The last part of the lemma is immediate and we omit details.
 \end{proof}
 
 \begin{lem}\label{lem:refinement well-defined}
  Let $Y \in \Sch_k$. Let $\mathcal{U} = \{ (U_i, X_i) \}_{i \in \Lambda}$ and $\mathcal{V}= \{ (V_j, X_j ') \}_{ j \in \Lambda'}$ be systems of local embeddings for $Y$ such that $\mathcal{V}$ is a refinement of $\mathcal{U}$ for two different set maps $\lambda, \lambda': \Lambda' \to \Lambda$.
  
  Then the maps $\lambda^*$ and $(\lambda' )^*$ are homotopic to each other.
 \end{lem}
 
 \begin{proof}
 This is standard. cf. \cite[Lemma 1.20, p.446]{Thomason etale}.
 \end{proof}

 \subsection{The filtrations}\label{subsec:cnv}
  
  Finally, in \S \ref{subsec:cnv}, we define two filtrations on the algebraic $K$-theory of $Y \in \Sch_k$. 
 
 \begin{defn} \label{def qstep}
 Let $Y \in \Sch_k$. Define in the homotopy category
 $$
 \widehat{\mathcal{G}} ^q_Y:= \underset{\mathcal{U}}{\hocolim} \ \check{\mathbb{H}}^{\cdot} (\mathcal{U}, \widehat{\mathcal{G}} ^q)
 $$
(see Definition \ref{defn:Cech hypercohomology} for $\check{\mathbb{H}}^{\cdot} (\mathcal{U}, \widehat{\mathcal{G}} ^q)$), where the homotopy colimit is taken over all systems $\mathcal{U}$ of local embeddings for $Y$. By Lemmas \ref{lem:comref}, \ref{lem:refinement induce} and \ref{lem:refinement well-defined}, the homotopy colimit over $\mathcal{U}$ is well-defined up to weak-equivalence.

\medskip

Similarly, for a fixed $n \geq 0$, if we use $\mathcal{G}^q (\widehat{X}_I, U_I,n)$ of Definition \ref{defn:K-sp mod Y} in \eqref{eqn:coface 1}-\eqref{eqn:complete cech K} instead of the geometric realization $\mathcal{G}^q (\widehat{X}_I, U_I)$, and modify accordingly the construction in Definitions \ref{defn:cosimplicial} and \ref{defn:Cech hypercohomology}, for $\mathcal{U}$, $q\geq 0$, $n\geq 0$ we obtain $ \check{\mathbb{H}}^{\cdot} (\mathcal{U}, \widehat{\mathcal{G}} ^q,n)$ as well as
$$
 \widehat{\mathcal{G}} ^q_{Y,n}:= \underset{\mathcal{U}}{\hocolim} \ \check{\mathbb{H}}^{\cdot} (\mathcal{U}, \widehat{\mathcal{G}} ^q,n)
 $$
in the homotopy category.
 \qed
 \end{defn}
 
Note that for each system $\mathcal{U}$ and a refinement $\mathcal{V}$ of $\mathcal{U}$, we have the following commutative diagram in the homotopy category
$$
\xymatrix{ 
\cdots \ar[r] & \check{\mathbb{H}}^{\cdot} (\mathcal{U}, \widehat{\mathcal{G}}^q) \ar[d] ^{\lambda^*} \ar[r] & \cdots \ar[r] & \check{\mathbb{H}}^{\cdot} (\mathcal{U}, \widehat{\mathcal{G}}^0) \ar[d] ^{\lambda^*} \ar[r] & \check{\mathbb{H}}^{\cdot} (| \mathcal{U}|, \mathcal{K}) \ar[d] ^{\lambda^*}  & \mathcal{K} (Y) \ar[l]_{\ \ \ w.e.}  \ar[dl]^{w.e.} \\
\cdots \ar[r] & \check{\mathbb{H}}^{\cdot} (\mathcal{V}, \widehat{\mathcal{G}}^q) \ar[r] & \cdots \ar[r] & \check{\mathbb{H}}^{\cdot} (\mathcal{V}, \widehat{\mathcal{G}}^0) \ar[r] & \check{\mathbb{H}}^{\cdot} (| \mathcal{V}|, \mathcal{K}) ,&}
$$
where $w.e.$ stands for weak-equivalences (Remark \ref{remk:Cech Zariski descent}).
So, after taking the homotopy colimits over $\mathcal{U}$ and inverting weak-equivalences, we deduce morphisms in the homotopy category
\begin{equation} \label{final mcnv fil}
\cdots \to  \widehat{\mathcal{G}} ^q_Y \to  \widehat{\mathcal{G}} ^{q-1} _Y\to \cdots \to \widehat{\mathcal{G}} ^0 _Y \to \mathcal{K} (Y).
\end{equation}
Similarly, if we repeat the above arguments with $\check{\mathbb{H}}^{\cdot} (\mathcal{U}, \widehat{\mathcal{G}} ^q,n)$ instead of $\check{\mathbb{H}}^{\cdot} (\mathcal{U}, \widehat{\mathcal{G}} ^q)$ for $n \geq 0$, we have
\begin{equation} \label{final uscnv fil}
\cdots \to \widehat{\mathcal{G}} ^q_{Y,n} \to   \widehat{\mathcal{G}} ^{q-1} _{Y,n}\to \cdots \to  \widehat{\mathcal{G}} ^0 _{Y,n} \to \mathcal{K} (Y).
\end{equation}
For \eqref{final uscnv fil}, we will be primarily interested in the case $n=0$.

\begin{defn}\label{defn:cnv cech}
Let $Y \in \Sch_k$ and let $q \geq 1$ be an integer. Define
$$
F^q_{\cnv} K_n (Y): = {\rm Im} \left(\pi_n (\widehat{\mathcal{G}} ^q _{Y,0}) 
\to \pi_n ( \mathcal{K} (Y)) \right)  = \underset{\mathcal{U}}{\varinjlim} \ {\rm Im} \left( \pi_n (\check{ \mathbb{H}} ( \mathcal{U}, \widehat{\mathcal{G}}^q,0)) \to \pi_n( \mathcal{K} (Y)) \right),
$$
and
$$
F^q_{\mcnv} K_n (Y):= {\rm Im} \left(\pi_n (\widehat{\mathcal{G}} ^q _Y) \to \pi_n ( \mathcal{K} (Y)) \right) = \underset{\mathcal{U}}{\varinjlim} \ {\rm Im} \left( \pi_n (\check{ \mathbb{H}} ( \mathcal{U}, \widehat{\mathcal{G}}^q)) \to \pi_n( \mathcal{K} (Y)) \right).
$$

For $q\leq 0$, we define
$$
F_{\cnv} ^q K_n (Y):= K_n (Y), \ \ \mbox{ and } \ \ 
 F_{\mcnv} ^q K_n (Y):= K_n (Y).
$$
By construction, they form decreasing filtrations
$$ 
K_n (Y) = F^{0} _{\cnv} K_n (Y)  \supset F^1 _{\cnv} K_n (Y)  \supset F^2 _{\cnv} K_n (Y) \supset \cdots,
$$
and
$$
K_n (Y) = F^{0} _{\mcnv} K_n (Y)  \supset F^1 _{\mcnv} K_n (Y) \supset F^2 _{\mcnv} K_n (Y) \supset \cdots,
$$
which we call \emph{the coniveau filtration} and \emph{the motivic coniveau filtration} on $K_n (Y)$, respectively.

By construction, these filtrations are independent of the choice of a particular system $\mathcal{U}$ of local embeddings for $Y$.
\qed
\end{defn}

\subsection{Consistency for the coniveau filtration}\label{sec:cnv sm comp}

In \S \ref{sec:cnv sm comp}, we check the consistency of the coniveau filtration $F_{\rm cnv} ^{\bullet}$ on $K_n (Y)$ in Definition \ref{defn:cnv cech}, when $Y$ is smooth.

\medskip

Consider the tower of spaces constructed by D. Quillen \cite[\S 7.5 Theorem 5.4]{Quillen}:
\begin{equation} \label{Quillen tower}
\cdots \to K (M_q(Y)) \to  K (M_{q-1}(Y))\to \cdots \to K (M_1(Y)) \to K (M_0(Y))=K (Y).
\end{equation}
where $M_q(Y)$ is the exact category consisting of coherent $\mathcal O _Y$-modules with the supports of codimension $\geq q$.

\begin{prop}\label{prop:sm cnv}
Let $Y$ be an equidimensional smooth $k$-scheme, and let $q\geq 1$. 

Then there exists a commutative diagram in the homotopy category 

\begin{align} \label{eqn:Quillen tower comp}
\begin{array}{c}
\xymatrix@C=1.5pc{
\cdots \ar[r] & K (M_q(Y)) \ar[r] \ar[d]&  K (M_{q-1}(Y)) \ar[r] \ar[d]& \cdots \ar[r]&  K (M_1(Y))  \ar[r] \ar[d]&K (Y) \ar[d]_-{w.e.}\\
\cdots \ar[r] & \widehat{\mathcal{G}} ^q_{Y,0} \ar[r]& \widehat{\mathcal{G}} ^{q-1} _{Y,0} \ar[r]& \cdots \ar[r]& \widehat{\mathcal{G}} ^1 _{Y,0} \ar[r]& \mathcal{K} (Y),}
 \end{array}
\end{align}
 where the vertical maps except the last one are weak-equivalences induced by the inclusions $M_q(Y)\subseteq \mathcal{D}_{\coh} ^q (Y)$.

In particular, the coniveau filtration in Definition \ref{defn:cnv cech} coincides with the classical coniveau filtration on $K_n (Y)$ of Quillen.
\end{prop}

\begin{proof}Since the classical coniveau filtration $F^q_{top} K_n (Y)$ on $K_n (Y)$ is defined \cite[\S 7.5 Theorem 5.4]{Quillen} as the image of the map $K_n(M_q(Y))\rightarrow K_n(Y)$ induced by applying $\pi_n$ to \eqref{Quillen tower}, it suffices to show the existence of the diagram \eqref{eqn:Quillen tower comp}, where the vertical maps are weak-equivalences.

Let $\mathcal{U} = \{ (U_i, X_i ) \}_{i \in \Lambda}$ be a system of local embeddings for $Y$, with each $U_i \subset Y$ is affine open. Since $Y$ is equidimensional and smooth over $k$, each open set $U_i \subset Y$ of the underlying open cover $| \mathcal{U}| = \{ U_i \}_{i \in \Lambda}$, is also equidimensional and smooth, so that the identity map $U_i \to U_i$ gives a closed immersion into an equidimensional smooth $k$-scheme. 

Hence the collection $\mathcal{U}_{0} := \{ (U_i, U_i) \}_{i \in \Lambda}$ is a system of local embeddings for $Y$, which is also a refinement of $\mathcal{U}$. Furthermore, $\widehat{U}_I = U_I$ for each $I \in \Lambda^{p+1}$, so by Remark \ref{rem. modY trivial} we obtain $\mathcal{G}^q (U_I, U_I, 0)=\mathcal{G}^q (U_I, 0)$ in Definition \ref{defn:K-sp mod Y}.  Thus, $\check{\mathbb{H}}^{\cdot} (\mathcal{U}_0, \widehat{\mathcal{G}}^q,0)= \check{\mathbb{H}}^{\cdot} (| \mathcal{U}|, \mathcal{G}^q (-, 0))$ for the presheaf on $Y_{\rm Zar}$ that assigns $U \mapsto \mathcal{G} ^q (U,0)$, where $\mathcal{G}^q (U,0)$ is as in \eqref{eqn:318-1} applied to the scheme $U$ seen as a formal scheme. 

\medskip
 
 Now, we consider the presheaf $U\mapsto K (M_q(U))$ on $Y_{\rm Zar}$. Observe that the natural  map $K (M_q(Y))\rightarrow \check{\mathbb{H}}^{\cdot} (|\mathcal{U}|, K (M_q(-)))$ is a weak-equivalence; this is standard, but one can argue also by combining the computation of Quillen \cite[(5.5) in Theorem 5.4 \S 7.5]{Quillen} and the descending induction argument of \cite[(10.3.13), (10.3.14), p.386]{TT}  applied to $\check{\mathbb{H}}^{\cdot}$ instead of $\mathbb{H}_{\rm Zar}^{\cdot}$. We shrink details.
 
 Therefore, it suffices to see that for each equidimensional smooth $k$-scheme $U$ the inclusions $M_q(U)\subseteq \mathcal{D}_{\coh} ^q (U)$ induce weak-equivalences $K (M_q(U))\rightarrow \mathcal{G}^q (U,0)$, which fit into a commutative diagram in the homotopy category:
 
 $$
\xymatrix{
\cdots \ar[r] & K (M_q(U)) \ar[r] \ar[d]& K (M_{q-1}(U)) \ar[r] \ar[d]& \cdots \ar[r]&  K (M_1(U))  \ar[r] \ar[d]&K (U) \ar[d]_{w.e.}\\
 \cdots \ar[r] & \mathcal{G}^q (U,0) \ar[r]& \mathcal{G}^{q-1} (U,0) \ar[r]& \cdots \ar[r]&  \mathcal{G}^1 (U,0) \ar[r]& \mathcal{K} (U).}
$$
 
 To check this, we observe that $\mathcal{D}_{\coh} ^q (U,0)=\mathcal{D}_{\coh} ^q (U)$ (see Remark \ref{D0-n equal}), and therefore the inclusion $M_q(U)\subseteq \mathcal{D}_{\coh} ^q (U)$ induces the desired weak-equivalence by \cite[Theorem 1.9.8, p.271]{TT}, \cite[Theorem 1.11.7, p.279]{TT}, and an argument parallel to \cite[Lemmas 3.11, 3.12, Corollary 3.13, pp.316-317]{TT} that compares the Quillen $K$-theory with the Waldhausen $K$-theory. The commutativity of the diagram is clear  since the horizontal maps are induced by the inclusion of the corresponding supports.
\end{proof}

\subsection{Consistency for the motivic coniveau filtration}\label{sec:mcnv sm comp}

Now in \S \ref{sec:mcnv sm comp}, we proceed to check the consistency when $Y$ is smooth for the motivic coniveau filtration $F_{\rm m.cnv} ^{\bullet}$ on $K_n (Y)$ of Definition \ref{defn:cnv cech}.

\medskip

First let's consider the cubical version of the homotopy coniveau tower of M. Levine \cite[\S 2.1]{Levine hcnv} for Quillen's
$K$-theory in what follows.

\begin{defn}\label{defn:S^q}
Let $n \geq 0$. 
Let $\mathcal S _Y^{[q]}(n)$ be the set of closed subsets $Z$ of $Y\times \square ^n$ of codimension $\geq q$,
such that
\begin{equation}\label{eqn:S^q 1}
\mathrm{codim} _{Y\times F}Z\cap(Y\times F)\geq q
\end{equation}
for each face $F \subset \square ^n$ (cf. Definition \ref{defn:HCG}-(\textbf{GP})).\qed
\end{defn}

\begin{remk}\label{remk:S^q z^q diff}
We note that there is a subtle difference between Definition \ref{defn:S^q} and Definition \ref{defn:HCG}-(\textbf{GP}). Recall that in the latter, when $Z$ in $Y \times \square^n$ has the codimension $c$, then we want
\begin{equation}\label{eqn:S^q z^q 1}
{\rm codim}_{Y \times F} ( Z \cap (Y \times F)) \geq c,
\end{equation}
for each face $F \subset \square^n$. However, when an irreducible set $Z$ is in $S^{[q]}_Y (n)$, and its codimension $c$ in $Y \times \square^n$ is $>q$, then we have \eqref{eqn:S^q 1}, but it does not imply \eqref{eqn:S^q z^q 1} in general. 

This is one reason why the proof of Lemma \ref{lem.compfin} later, is more complicated than one might hope at first sight.\qed
\end{remk}

\begin{defn}\label{defn:555}
Given $Z\in \mathcal S _Y^{[q]}(n)$, let $K ^Z (Y\times \square ^n)$ denote the homotopy fiber of the map $j^\ast: K (Y\times \square ^n)\rightarrow K ((Y\times \square ^n) \backslash Z)$ induced by the open immersion $j:(Y\times \square ^n)\backslash Z \rightarrow Y \times \square^n$, where $K(-)$ is Quillen's $K$-space functor.

Let $ K^{[q]}(Y,n)$ be the filtered homotopy colimit
$$
	K^{[q]}(Y,n)=\hocolim _{Z\in \mathcal S _Y^{[q]}(n)} K ^Z (Y\times \square ^n).
$$

By pulling back along the cofaces and codegeneracies of the co-cubical scheme $(\un{n}\mapsto Y\times \square ^n)$, we obtain a cubical space $(\un{n}\mapsto K^{[q]}(Y,n))$, and we will write
$$
 K^{[q]}(Y):= | \un{n} \mapsto K^{[q]} (Y, n)|
 $$
  for its (cubical) geometric realization, as in Remark \ref{remk:geometric realize}.\qed
\end{defn}

Via the inclusions of supports, we deduce the following tower of spaces

\begin{equation} \label{eqn:hconiveautower}
 \cdots \to K^{[q]} (Y) \to   K^{[q-1]} (Y)\to \cdots \to K^{[1]} (Y) \to K^{[0]} (Y)\cong K(Y)
\end{equation}
in the homotopy category, where the last weak-equivalence follows from the $\mathbb A ^1$-invariance of $K$-theory for smooth $k$-schemes.

In order to compare the cubical homotopy coniveau tower \eqref{eqn:hconiveautower} with the simplicial one in M. Levine \cite[\S 2.1]{Levine hcnv}, we consider the following intermediate ``simplicial-cubical" construction.

\begin{defn}\label{defn:553}
Let $q, m, n \geq 0$ 
Let $\mathcal S _Y^{q}(m,n)$ be the set of closed subsets $Z$ of $Y\times \Delta ^m \times \square ^n$ of codimension $\geq q$ such that
$$
\mathrm{codim} _{Y\times F_s \times F_c}Z\cap(Y\times F_s \times F_c)\geq q
$$
for each pair consists of a face $F_s \subset \Delta^m$ and a face $F_c \subset \square^n$.

Let $K^{q}(Y,m,n)$ be the filtered homotopy colimit
$$
	 K^{q}(Y,m,n)=\hocolim _{Z\in \mathcal S _Y^{q}(m,n)} K ^Z (Y\times \Delta ^m \times \square ^n),
$$
which gives a simplicial-cubical space. We will write 
$$ 
\tuborg K^{q,(\un{m})}(Y,n) =|\un{m}\mapsto K^{q}(Y,m,n)|, \\
K^{q,[\un{n}]}(Y,m) = |\un{n}\mapsto  K^{q}(Y,m,n)|,
\sluttuborg
$$
for the geometric realizations of the simplicial and the cubical spaces, respectively. \qed
\end{defn}

Thus, we obtain a cubical (resp. simplicial) space $(\un{n}\mapsto  K^{q,(\un{m})}(Y,n))$ (resp. $(\un{m}\mapsto K^{q,[\un{m}]}(Y,m))$), and since colimits commute among themselves we conclude that the corresponding geometric realizations are isomorphic to each other in the homotopy category:

$$
	|\un{n} \mapsto  K^{q,(\un{m})}(Y,n) | \cong	|\un{m}\mapsto K^{q,[\un{n}]}(Y,m) | := K^{q,(\un{m}),[\un{n}]}(Y).
$$

\begin{lem} \label{lem.simpcubcomp}
Let $Y$ be an equidimensional smooth $k$-scheme, and let $q,m,n\geq 0$ be integers.

Then there are weak-equivalences
\begin{equation} \label{eqn:simp-cub we}
	 K ^{[q]}(Y)\rightarrow K^{q,(\un{m}),[\un{n}]}(Y) \leftarrow K^{(q)}(Y),
\end{equation}
that are natural in $Y$, where $ K^{(q)}(Y)$ is the $q$-th term in the simplicial homotopy coniveau tower of M. Levine \cite[\S 2.1]{Levine hcnv}.
\end{lem}

\begin{proof}
We apply M. Levine \cite[Proposition 3.3.4]{Levine Chow} for $A=k$ to conclude that there is a weak-equivalence $K^{q,(\un{m})}(Y,n)\rightarrow K ^{(q)}(Y\times \square ^n)$, which is natural in $Y$. Then \cite[Theorem 3.3.5]{Levine Chow} implies that the degeneracies $K^{q,(\un{m})}(Y,n) \rightarrow K^{q,(\un{m})}(Y,n+1)$ are weak-equivalences for $n\geq 1$, so all the $n$-cubes in the cubical space $(\un{n}\mapsto K^{q,(\un{m})}(Y,n))$ are degenerate for $n\geq 1$. Hence the map from the $0$-cubes into the geometric realization gives the first desired weak-equivalence of \eqref{eqn:simp-cub we}
\begin{equation}\label{eqn:simpcubcomp1}
K^{(q)}(Y)= K^{q,(\un{m})}(Y,0)\rightarrow  |\un{n} \mapsto  K^{q,(\un{m})}(Y,n) | = K^{q,(\un{m}),[\un{n}]}(Y).
\end{equation}

To construct the remaining weak-equivalence of \eqref{eqn:simp-cub we}, we use the cubical analogue of \cite[Proposition 3.3.4, Theorem 3.3.5]{Levine Chow}, \emph{mutatis mutandis} and argue similarly. We shrink details. 

As a result, we deduce the map from the $0$-simplices into the geometric realization is the second desired weak-equivalence of \eqref{eqn:simp-cub we}:
\begin{equation}\label{eqn:simpcubcomp2}
K^{[q]}(Y)= K^{q,[\un{n}]}(Y,0)\rightarrow  |\un{m}\mapsto  K^{q,[\un{n}]}(Y,m) | = K^{q,(\un{m}),[\un{n}]}(Y).
\end{equation}
Combining the weak-equivalences in \eqref{eqn:simpcubcomp1} and \eqref{eqn:simpcubcomp2}, we have the lemma.
\end{proof}

We can now state the comparison result:

\begin{prop}\label{prop:sm m.cnv}
Let $Y$ be an equidimensional smooth $k$-scheme, $q\geq 1$. Then there exists a commutative diagram in the homotopy category 

\begin{equation} \label{eqn:coniveau tower comp}
\begin{array}{c}
\xymatrix{
\cdots \ar[r] & K^{[ q]} (Y) \ar[r] \ar[d]& K^{[ q-1]} (Y) \ar[r] \ar[d]& \cdots \ar[r]&   K^{[ 1]} (Y)  \ar[r] \ar[d]& K (Y) \ar[d]_{w.e.}\\
\cdots \ar[r] &  \widehat{\mathcal{G}} ^q_{Y} \ar[r]&  \widehat{\mathcal{G}} ^{q-1} _{Y} \ar[r]& \cdots \ar[r]&   \widehat{\mathcal{G}} ^1 _{Y} \ar[r]& \mathcal{K} (Y),}
 \end{array}
\end{equation}
where the vertical maps are weak-equivalences:

In particular, the motivic coniveau filtration in Definition \ref{defn:cnv cech} coincides with the homotopy coniveau filtration on $K_n (Y)$ of M. Levine \cite[\S 2.1]{Levine hcnv}.
\end{prop}

\begin{proof}
By Lemma \ref{lem.simpcubcomp}, it suffices to show the existence of the commutative diagram \eqref{eqn:coniveau tower comp}, where the vertical maps are weak-equivalences.

Let $\mathcal{U} = \{ (U_i, X_i ) \}_{i \in \Lambda}$ be a system of local embeddings for $Y$, where each $U_i$ is affine for $i\in \Lambda$. Since $Y$ is equidimensional and smooth over $k$, each open set $U_i \subset Y$ of the underlying open cover $| \mathcal{U}| = \{ U_i \}_{i \in \Lambda}$, is also equidimensional and smooth, so that the identity map $U_i \to U_i$ gives a closed immersion into an equidimensional smooth $k$-scheme. 

Hence the collection $\mathcal{U}_{0} := \{ (U_i, U_i) \}_{i \in \Lambda}$ is a system of local embeddings for $Y$, which is also a refinement of $\mathcal{U}$. Furthermore, $\widehat{U}_I = U_I$ for each $I \in \Lambda^{p+1}$, so by Remark \ref{rem. modY trivial}, we obtain $\mathcal{G}^q (U_I, U_I, n)=\mathcal{G}^q (U_I, n)$ for $n\geq 0$ and thus $\mathcal{G}^q (U_I, U_I) =\mathcal{G}^q (U_I)$.

Thus $\check{\mathbb{H}}^{\cdot} (\mathcal{U}_0, \widehat{\mathcal{G}}^q)= \check{\mathbb{H}}^{\cdot} (| \mathcal{U}|, \mathcal{G}^q)$, for the presheaf $U \mapsto \mathcal{G} ^q (U)$ on $Y_{\rm Zar}$. 
 
 Now, we consider the presheaf $U\mapsto K^{[q]} (U)$ on $Y_{\rm Zar}$. Observe that the natural map $K^{[q]} (Y)\rightarrow \check{\mathbb{H}}^{\cdot} (|\mathcal{U}|, K^{[q]} (-))$ is a weak-equivalence; this follows from Lemma \ref{lem.simpcubcomp} and M. Levine \cite[Theorem 7.1.1]{Levine hcnv}.
 
 Therefore, it suffices to see that for every smooth affine equidimensional $k$-scheme $U$ there exist weak-equivalences 
 $\mathcal{G}^q (U)\rightarrow \mathcal{K}^{[q]} (U)$ for $q \geq 0$, which fit into a commutative diagram in the homotopy category:
 
$$
 \xymatrix{
\cdots \ar[r] & K^{[q]} (U) \ar[r] \ar[d] & K^{[q-1]} (U) \ar[r] \ar[d] & \cdots \ar[r]& K^{[1]} (U)  \ar[r] \ar[d] &K (U)  \ar[d]^-{w.e.}\\
\cdots \ar[r] & \mathcal{G}^q (U) \ar[r] &  \mathcal{G}^{q-1} (U) \ar[r] & \cdots \ar[r]&  \mathcal{G}^1 (U) \ar[r] & \mathcal{K} (U).}
$$
 
 The existence of the vertical arrows follows by taking the geometric realizations for the weak-equivalences between the cubical spaces in Lemmas \ref{lem.compfin1} and \ref{lem.compfin} proven separately below. The commutativity of the diagram is clear since the horizontal maps are induced by inclusion of supports which are compatible with the weak-equivalences in Lemmas \ref{lem.compfin1} and \ref{lem.compfin}.
\end{proof}

The proof of the above Proposition \ref{prop:sm m.cnv} used the following Lemmas \ref{lem.compfin1} and \ref{lem.compfin}. We establish them in what follows.

\begin{defn}\label{defn:559}
Let $M_q(Y,n)$ be the exact category of coherent $\mathcal O _{Y\times \square ^n}$-modules with the supports in $\mathcal S _Y^{[q]}(n)$ defined in Definition \ref{defn:553}, and let $\mathcal D ^{[q]}_{\coh}(Y,n)\subseteq \mathcal D_{\coh}(Y\times \square ^n)$ be the full triangulated subcategory generated by $M_q(Y,n)$.\qed
\end{defn}

\begin{lem}  \label{lem.compfin1}
Let $Y$ be an equidimensional smooth $k$-scheme, and let $q$, $n\geq 0$.

Then there are weak-equivalences:
\begin{equation}\label{eqn:compfin1-1}
	 \mathcal{K} (\mathcal D ^{[q]}_{\coh}(Y,n))  \leftarrow K(M_q(Y,n))\rightarrow K^{[q]}(Y,n),
\end{equation}
that are natural in $Y$ and they form maps of cubical spaces as $n \geq 0$ varies, where the space $K^{[q]} (Y, n)$ is defined in Definition \ref{defn:555}.
\end{lem}

\begin{proof}
By Thomason-Trobaugh \cite[Theorem 1.9.8, p.271]{TT}, \cite[Theorem 1.11.7, p.279]{TT}, and an argument parallel to \cite[Lemmas 3.11, 3.12, Corollary 3.13, pp.316-317]{TT}, the inclusion $M_q(Y,n)\subseteq \mathcal{D}_{\coh} ^{[q]} (Y,n)$ induces the left weak-equivalence of \eqref{eqn:compfin1-1}
$$
K(M_q(Y,n)) \overset{w.e.}{\rightarrow} \mathcal K (\mathcal D ^{[q]}_{\coh}(Y,n)).
$$
It is natural in $Y$ because so is the inclusion $M_q(Y,n)\subseteq \mathcal{D}_{\coh} ^{[q]} (Y,n)$. The compatibility with the faces and the degeneracies is immediate.

\medskip

For the remaining map of \eqref{eqn:compfin1-1}, we apply D. Quillen \cite[\S 2 (9), p.104]{Quillen} and \cite[\S 7 Proposition 3.1, p.127]{Quillen} to obtain a weak-equivalence
$$
K(M_q(Y,n)) \overset{w.e.}{\rightarrow} \colim _{Z\in \mathcal S _Y^{[q]}(n)}G(Z).
$$
However, $G(Z)=K^Z(Y\times \square ^n)$ by the localization theorem for $G$-theory in D. Quillen \cite[\S 5 Theorem 5, p.113]{Quillen}. Thus, there is a weak-equivalence
\begin{equation}\label{eqn:compfin1-2}
K(M_q(Y,n)) \overset{w.e.}{\rightarrow} \colim _{Z\in \mathcal S _Y^{[q]}(n)}K^Z (Y\times \square ^n).
\end{equation}
Since this is a filtered colimit, it follows from D. Dugger \cite[Proposition 7.3]{Dugger} that the right hand side of \eqref{eqn:compfin1-2} computes the filtered homotopy colimit
$$
\hocolim _{Z\in \mathcal S _Y^{[q]}(n)} K ^Z (Y\times \square ^n)=K^{[q]}(Y,n).
$$
This construction is also natural in $Y$ and it is compatible with the faces and the degeneracies.
\end{proof}

\begin{defn}
Let $M'_{q}(Y,n)$ be the exact category of coherent $\mathcal O _{Y\times \square ^n}$-modules that are admissible of codimension $\geq q$ (see Definition \ref{defn:cnv higher}). Since $Y$ is smooth, there is an inclusion $M'_q(Y,n)\subseteq M_q (Y,n)$ into the category of Definition \ref{defn:559}. This induces a fully faithful triangulated functor $i_n:\mathcal D_{\coh}^q(Y,n) \rightarrow \mathcal D_{\coh}^{[q]}(Y,n)$.
\qed
\end{defn}

\begin{lem} \label{lem.compfin}
Let $U$ be an equidimensional smooth affine $k$-scheme, and let $q$, $n\geq 0$. Let $\Spec (R) = U \times \square^n$.

Then the triangulated functor $i_n:\mathcal D_{\coh}^q(U,n) \rightarrow \mathcal D_{\coh}^{[q]}(U,n)$ is an equivalence of categories, so it induces a weak-equivalence
$$
	\iota _n :\mathcal G ^q(U,n)\rightarrow \mathcal K (\mathcal D_{\coh}^{[q]}(U,n)).
$$
Furthermore, it is natural in $U$ and forms a map of cubical spaces as $n \geq 0$ varies. 
\end{lem}

\begin{proof} 
The morphism $\iota_n$ is natural in $U$ because the triangulated functor $i_n$ is induced by the inclusion $M'_q(U,n)\subseteq M_q(U,n)$, which is clearly natural. The compatibility with the faces and the degeneracies for $\iota_n$ is immediate.

It follows from Thomason-Trobaugh \cite[Theorems 1.9.1, 1.9.8, p.263, p.271]{TT} that $\iota _n$ is a weak-equivalence provided that $i_n$ is an equivalence of categories. Since $i_n$ is fully faithful, it only remains to check that $i_n$ is essentially surjective on objects.

\medskip

Recall (Definition \ref{defn:559}) that $\mathcal D_{\coh}^{[q]}(U,n)$ is the full triangulated subcategory of $\mathcal D _{\coh}(U\times \square ^n)$ generated by coherent $\mathcal O _{U\times \square ^n}$-modules with the supports in $\mathcal S _U^{[q]}(n)$. Therefore if $q=0$, we indeed have the equality $\mathcal D_{\coh}^q(U,n)=\mathcal D_{\coh}^{[q]}(U,n)$, thus $\mathcal{G}^0 (U, n) = \mathcal{K} (\mathcal{D}_{\coh} ^{[0]} (U, n))$.

\medskip

We now assume $q \geq 1$. 

Since every coherent $\mathcal O _{U\times \square ^n}$-module $\mathcal M$ admits a finite decreasing filtration $\mathcal M_i\subset \mathcal M$ such that each successive quotient $\mathcal M_i/\mathcal M_{i+1}\cong \mathcal O_{Z_i}$ for an irreducible reduced closed subscheme $Z_i$ of $U\times \square ^n$, where the set of such $Z_i$ is uniquely determined counting multiplicities. Thus in order to show the required essential surjectivity, it is enough to verify the following two claims:

\begin{enumerate}
\item \label{lem.compfin.a}
Suppose $Z \in  \mathcal S _U^{[q]}(n)$ is irreducible. Then there exists a closed subscheme $Z'$ of $Y\times \square ^n$ such that $Z\subseteq Z'$ and $[Z']\in z^{\geq q} (U, n)$.
\item \label{lem.compfin.b}
Let $Z\subseteq Z'$ be irreducible reduced closed subschemes of $U\times \square ^n$ such that $[Z']\in z^{\geq q} (U, n)$.  Then $\mathcal O_Z \in \mathcal{D}_{\coh} ^q (U,n)$. (N.B. We don't claim $[Z] \in z^{\geq q} (U, n)$. See Remark \ref{remk:S^q z^q diff}.)
\end{enumerate}

\medskip

We show that the two claims hold.

\eqref{lem.compfin.a}:  Since $Z \in \mathcal S _U^{[q]}(n)$, $Z$ is a closed subset of $U\times \square ^n$ of codimension $c\geq q$, such that
\begin{equation}  \label{eq.faceinterq}
\mathrm{codim} _{U\times F}Z\cap(U\times F)\geq q\geq 1 \tag{IF(q)}
\end{equation}
for each face $F \subset \square ^n$.  If $c=q$, by definition $Z$ intersects $U\times F$ properly for each face $F$ of $\square ^n$, so $[Z]\in z^{\geq q} (U, n)$. 

In case $c\geq q+1$, by induction we will construct $f_1, \cdots , f_q \in R$ such that $Z\subseteq V_i=V(f_1, \cdots , f_i)$ and $[V_i] \in z^{\geq i} (U, n)$  for $1\leq i\leq q$. (N.B. Here, the codimension of $V_i$ should be necessarily $i$.)

 Let $J\subset R$ be an ideal such that $V(J)=Z$ and let $\{ \mathfrak P _{0,r}\in \Spec (R) \}_{r\in C_0}$ be the finite set of prime ideals such that $\{ V(\mathfrak P _{0,r})\}_{r\in C_0}$ is the set of irreducible components in the collection of closed subsets $\{ U\times F  \ | \ F  \mbox{ is a face of } \square ^n\}$. By \eqref{eq.faceinterq}, we deduce that $J\not \subset \mathfrak P _{0,r}$ for all $r\in C_0$, thus the prime avoidance (Lemma \ref{lem:prime avoid}) implies that there is some $f_1\in J$ such that $f_1\not \in \mathfrak P _{0,r}$ for all $r\in C_0$.  Hence $V_1:=V(f_1)$ satisfies the required conditions.

We proceed the induction steps. Let $1 \leq i < q$. Assume that there are $f_1, \cdots , f_i\in R$, with the required properties. Let $\{ \mathfrak P _{i,r}\in \Spec (R) \}_{r\in C_i}$ be the finite set of prime ideals such that $\{ V(\mathfrak P _{i,r})\}_{r\in C_i}$ is the set of irreducible components in the collection of closed subsets $\{ V_i\cap (U\times F) \ | \ F \mbox{ is a face of } \square ^n\}$. Since $\mathrm{codim} _{U\times F}V_i\cap(U\times F)=i<q$, we deduce by \eqref{eq.faceinterq} that $J \not \subset \mathfrak P _{i,r}$ for all $r\in C_i$, thus the prime avoidance implies that there is some $f_{i+1}\in J$ such that $f_{i+1}\not \in \mathfrak P _{i,r}$ for all $r\in C_i$.  Hence $V_{i+1}:=V(f_1,\cdots , f_{i+1})$ satisfies the required conditions. This proves \eqref{lem.compfin.a}.

\medskip

\eqref{lem.compfin.b}: If $Z=Z'$, then the claim holds by definition, so we may assume $Z\subsetneq Z'$.

Let $0\neq \mathcal I \subset \mathcal O_{Z'}$ be the ideal sheaf such that $\mathcal O_{Z'}/\mathcal I \cong \mathcal O_Z$. Consider the short exact sequence 
\begin{equation}\label{eqn:tri1}
0\rightarrow \mathcal I \rightarrow \mathcal O_{Z'}\rightarrow \mathcal O_Z \rightarrow 0,
\end{equation}
and observe that $\mathcal O_{Z'} \in \mathcal{D}_{\coh} ^q (U,n)$ since we are given that $[Z']\in z^{\geq q} (U, n)$. 

We first claim that $[\mathcal I ]\in z^{\geq q} (U, n)$. 

To prove it, we switch to affine notations: let $\mathfrak{P} \in \Spec (R)$ be such that $\Spec (R/\mathfrak{P}) = Z'$, and let $\mathfrak{Q} \in \Spec ( R/ \mathfrak{P})$ be such that $\Spec (R/ \mathfrak{Q}) = Z$. In particular, $\widetilde{\mathfrak{Q}} = \mathcal{I}$. Then $R/ \mathfrak{P}$ is an integral domain since $Z'$ is reduced and irreducible. Thus $\mathrm{Ass} (R / \mathfrak{P}) = \{ (0) \}$, the zero ideal of $R/ \mathfrak{P}$. Since we have $0 \not = \mathfrak{Q} \subset R/ \mathfrak{P}$, we have $\emptyset \not = \mathrm{Ass} (\mathfrak{Q}) \subset \mathrm{Ass} (R/ \mathfrak{P}) = \{ (0)\}$. In particular, $\mathrm{Ass} (\mathfrak{Q}) = \{ (0)\}$.

Since $(0) \subset R/ \mathfrak{P}$ corresponds to $Z'$, this means $[ \mathcal{I}] = [ \widetilde{\mathfrak{Q}}] = [ Z' ]$, which is in $z^{ \geq q} (U, n)$, proving the claim.

This implies in particular that $\mathcal{I} \in \mathcal{D}_{\coh} ^q (U, n)$. Since $\mathcal{I}, \mathcal{O}_{Z'} \in \mathcal{D}_{\coh} ^q (U, n)$, from the exact sequence \eqref{eqn:tri1}, we deduce that $\mathcal{O}_Z \in \mathcal{D}_{\coh} ^q (U, n)$ (see e.g. A. Neeman \cite[Remark 1.5.2]{Neeman}). This proves (2), and finishes the proof of the lemma.
\end{proof}

\subsection{Functoriality}
We want to check that the two filtrations we defined in Definition \ref{defn:cnv cech} are functorial on $\Sch_k$. We consider the following, which is also used in \cite[\S 6]{P general}. 
This also generalizes the notion of refinement of systems of local embeddings:

 \begin{defn}\label{defn:funrefine}
 Let $g: Y_1 \to Y_2$ be a morphism in $\Sch_k$. Let $\mathcal{U}= \{ (U_i, X_i)\}_{i \in \Lambda}$ be a system of local embeddings for $Y_2$.
 
 We say that a system $\mathcal{V}= \{ (V_j, X'_j)\}_{j \in \Lambda'}$ of local embeddings for $Y_1$ is \emph{associated to $\mathcal{U}$ by $g$}, if there is a set map $\lambda: \Lambda' \to \Lambda$ such that for each $j \in \Lambda'$, we have $g (V_j) \subset U_{\lambda (j)}$, and there is a morphism $f_j: X_j' \to X_{\lambda (j)}$ such that $f_j|_{V_j} = g|_{V_j}$.
 \qed
  \end{defn}
  
  Such systems do exist:
  
\begin{lem}\label{lem:ass system}
Let $g: Y_1 \to Y_2$ be a morphism in $\Sch_k$. Let $\mathcal{U} = \{ (U_i, X_i) \}_{i \in \Lambda}$ be an arbitrary system of local embeddings for $Y_2$. 

Then there exists a system $\mathcal{V}$ of local embeddings for $Y_1$ associated to $\mathcal{U}$ by $g$.
\end{lem}

\begin{proof}
For each $i \in \Lambda$, the open subset $g^{-1} (U_i) \subset Y_1$ may not be affine, but we can find a finite affine open cover $\{ V_{ij} \}_{j \in \Lambda_i'}$ of $g^{-1} (U_i)$ for a finite set $\Lambda_i'$. Let $\Lambda':= \coprod_{ i \in \Lambda} \Lambda_i '$. Let $\lambda: \Lambda' \to \Lambda$ be the projection set map given by sending $j \in \Lambda_i'$ to $i$.

 For each $V_{ij}$, choose a closed immersion $\iota_{ij}: V_{ij} \hookrightarrow Z_{ij}$ into an equidimensional smooth $k$-scheme. Let $X_{ij}' := Z_{ij} \times X_i$, and consider the  closed immersion
 $$ 
 V_{ij} \overset{ gr_{g_{ij}}}{\hookrightarrow} V_{ij} \times U_i \hookrightarrow Z_{ij} \times  X_i = X_{ij}',
 $$
 where we let $g_{ij} := g|_{V_{ij}}$. Then $\mathcal{V}:= \{ ( V_{ij}, X'_{ij}) \}_{i,j}$ is a system of local embeddings for $Y_1$. For the projection $f_{ij}: X_{ij}' = Z_{ij} \times X_i \to X_i$, we have $f_{ij}|_{V_{ij}} = g_{ij} = g|_{V_{ij}}$. Thus $\mathcal{V}$ is associated to $\mathcal{U}$ by $g$.
\end{proof}
 
 \begin{lem}\label{lem:Cech functor}
 Let $g: Y_1 \to Y_2$ be a morphism in $\Sch_k$ and let $\mathcal{U}= \{ (U_i, X_i)\}_{ i \in \Lambda}$ be a system of local embeddings for $Y_2$. Let $\mathcal{V}= \{ (V_j, X_j')\}_{j \in \Lambda'}$ be a system of local embeddings for $Y_1$, that is associated to $\mathcal{U}$ by $g$ as in Lemma \ref{lem:ass system}.
 
 Then there exists a morphism of cosimplicial objects
 \begin{equation}\label{eqn:Cech functoriality 0-1}
 g_{\mathcal{U}, \mathcal{V}}^* : \widehat{\mathcal{G}} ^q _{\mathcal{U}} \to \widehat{\mathcal{G}} ^q _{\mathcal{V}}
 \end{equation}
 in the homotopy category. In particular, we have a morphism
 \begin{equation}\label{eqn:Cech functoriality 0-2}
 g_{\mathcal{U}, \mathcal{V}}^* : \check{\mathbb{H}}^{\cdot} (\mathcal{U}, \widehat{\mathcal{G}} ^q) \to  \check{\mathbb{H}}^{\cdot} (\mathcal{V}, \widehat{\mathcal{G}} ^q) 
 \end{equation}
 in the homotopy category.
 \end{lem}
 
 \begin{proof}
 Let $p \geq 0$ be an integer. Let $J \in (\Lambda')^{p+1}$. Then we have the commutative diagram
 $$
 \xymatrix{
 X_J ' \ar[r]^{f_J} & X_{\lambda (J)} \\
 V_J \ar[r] ^{g|_{V_J}} \ar@{^{(}->}[u] & U_{\lambda (J)}.  \ar@{^{(}->}[u] }
 $$
 By Theorem \ref{thm:moving pull-back mod Y}, this diagram induces a morphism
 \begin{equation}\label{eqn:Cech functoriality 1}
 \mathcal{G}^q (\widehat{X}_{\lambda (J)}, U_{\lambda (J)}) \to \mathcal{G}^q (\widehat{X}_J ', V_J)
 \end{equation}
  in the homotopy category.
 These induce morphisms in the homotopy category for $p \geq 0$
 \begin{equation}\label{eqn:Cech functoriality 2}
 \prod_{I \in \Lambda^{p+1}} \mathcal{G}^q (\widehat{X}_I, U_I ) \to \prod_{J \in (\Lambda')^{p+1}} \mathcal{G}^q (\widehat{X}_J', V_J)
 \end{equation}
 defined as follows: for the $J$-th factor on the right hand side of \eqref{eqn:Cech functoriality 2}, we first project the left hand side to the factor over $\lambda (J) \in \Lambda^{p+1}$, and then apply \eqref{eqn:Cech functoriality 1}. This gives \eqref{eqn:Cech functoriality 2}. That they are compatible with the cofaces and codegeneracies are immediate, and we thus have a morphism of the cosimplicial objects \eqref{eqn:Cech functoriality 0-1}. By taking the homotopy limits over $\Delta$, we deduce \eqref{eqn:Cech functoriality 0-2}.
 \end{proof}
 
 We want to take homotopy colimits over $\mathcal{V}$ and $\mathcal{U}$ in \eqref{eqn:Cech functoriality 0-2}. For this, we need:
 
 \begin{lem}\label{lem:strrefine}
 Let $g: Y_1 \to Y_2$ be a morphism in $\Sch_k$. Let $\mathcal{U} =\{ (U_i, X_i) \}_{i \in \Lambda}$ be a system of local embeddings for $Y_2$ and let $\mathcal{V} = \{ (V_j, X_j ') \}_{ j \in \Lambda'}$ be a system for $Y_1$.
 
 Then there exists a system $\mathcal{V}'$ of local embeddings for $Y_1$, which is 
 \begin{enumerate}
 \item a refinement of $\mathcal{V}$ and
 \item associated to $\mathcal{U}$ by $g$.
 \end{enumerate}
 \end{lem}
 
 \begin{proof}
 First choose a system $\mathcal{W}$ of local embeddings for $Y_1$, that is associated to $\mathcal{U}$ by $g$, by Lemma \ref{lem:ass system}. 
 
 If the system $\mathcal{W}$ is already a refinement of $\mathcal{V}$, then take $\mathcal{V}'= \mathcal{W}$. 
 
 If not, then by Lemma \ref{lem:comref}, find a common refinement $\mathcal{V}'$ of both systems $\mathcal{V}$ and $\mathcal{W}$ for $Y_1$. Note that when $\mathcal{W}$ is associated to $\mathcal{U}$ by $g$, then so is any refinement of $\mathcal{W}$. Hence $\mathcal{V}'$ satisfies the above two conditions of the lemma. 
 \end{proof}
 
 \begin{lem}\label{lem:double system compatible}
 Let $g: Y_1 \to Y_2$ be a morphism in $\Sch_k$. Let $\mathcal{U}, \mathcal{U}'$ be systems of local embeddings for $Y_2$ and $\mathcal{V}, \mathcal{V}'$ be systems for $Y_1$, such that 
 \begin{enumerate}
 \item $\mathcal{V}$ is associated to $\mathcal{U}$ by $g$,
 \item $\mathcal{V}'$ is associated to $\mathcal{U}'$ by $g$,
 \item $\mathcal{U}'$ is a refinement of $\mathcal{U}$,
 \item $\mathcal{V}'$ is a refinement of $\mathcal{V}$.
 \end{enumerate}
 Then the following diagram
 $$
 \xymatrix{
 \check{\mathbb{H}}^{\cdot} (\mathcal{U}, \widehat{\mathcal{G}}^q) \ar[r] ^{g_{\mathcal{U}, \mathcal{V}}^*} \ar[d] &  \check{\mathbb{H}}^{\cdot} (\mathcal{V}, \widehat{\mathcal{G}}^q)  \ar[d] \\
  \check{\mathbb{H}}^{\cdot} (\mathcal{U}', \widehat{\mathcal{G}}^q) \ar[r] ^{g_{\mathcal{U}', \mathcal{V}'}^*} &  \check{\mathbb{H}}^{\cdot} (\mathcal{V}', \widehat{\mathcal{G}}^q) }
  $$
 in the homotopy category of spaces commutes, where the horizontal maps are supplied by Lemma \ref{lem:Cech functor} and the vertical maps are the refinement morphisms of Lemma \ref{lem:refinement induce}.
 \end{lem}
 
 \begin{proof}
 It is elementary. We omit it.
 \end{proof}
 
 For a morphism $g: Y_1 \to Y_2$ in $\Sch_k$, choose a system $\mathcal{U}$ of local embeddings for $Y_2$ and choose a system $\mathcal{V}$ for $Y_1$ associated to $\mathcal{U}$ by $g$ by Lemma \ref{lem:ass system}, so that we have the map $g_{\mathcal{U}, \mathcal{V}}: \check{\mathbb{H}}^{\cdot} (\mathcal{U}, \widehat{\mathcal{G}}^q) \to \check{\mathbb{H}}^{\cdot} (\mathcal{V}, \widehat{\mathcal{G}}^q)$. 
 
Note that a refinement $\mathcal{V}'$ of $\mathcal{V}$ again gives a system for $Y_1$ associated to $\mathcal{U}$ by $g$. Hence we have the induced morphism in the homotopy category
\begin{equation}\label{eqn:g_U}
g_{\mathcal{U}, \cdot}^*:= \underset{\mathcal{V}}{\hocolim} \ g_{\mathcal{U}, \mathcal{V}}^*: \check{\mathbb{H}}^{\cdot} (\mathcal{U}, \widehat{\mathcal{G}}^q) \to \widehat{\mathcal{G}}_{Y_1} ^q.
\end{equation}

On the other hand, by Lemmas \ref{lem:strrefine} and \ref{lem:double system compatible}, for a refinement $\mathcal{U}'$ of 
$\mathcal{U}$, the morphisms \eqref{eqn:g_U} for $\mathcal{U}$ and $\mathcal{U}'$ satisfy the commutative diagram
$$
\xymatrix{
\check{\mathbb{H}}^{\cdot} (\mathcal{U}, \widehat{\mathcal{G}} ^q)  \ar[rr]^{g_{\mathcal{U}, \cdot}^*} \ar[d] & & \widehat{\mathcal{G}}^q _{Y_1}  \\
\check{\mathbb{H}}^{\cdot} (\mathcal{U}', \widehat{\mathcal{G}}^q), \ar[rru] _{g_{\mathcal{U}', \cdot}^*} & & }
$$
where the vertical arrow is the morphism induced by the refinement. Hence this induces
\begin{equation}\label{eqn:pull-back}
g^*:=\underset{\mathcal{U}}{\hocolim}  \ g_{\mathcal{U},\cdot}^*: \widehat{\mathcal{G}}_{Y_2} ^q \to \widehat{\mathcal{G}}_{Y_1} ^q,
\end{equation}
which fits into the following commutative diagram in the homotopy category
\begin{align} \label{eqn:pull-back mtower}
\begin{array}{c}
\xymatrix{
\cdots \ar[r] & \widehat{\mathcal{G}} ^q_{Y_2} \ar[r] \ar[d]_-{g^\ast}& \widehat{\mathcal{G}} ^{q-1} _{Y_2} \ar[r] \ar[d]_-{g^\ast}& \cdots \ar[r] & \widehat{\mathcal{G}} ^0 _{Y_2} \ar[r] \ar[d]_-{g^\ast}&  \mathcal{K} (Y_2) \ar[d]_-{g^\ast}\\
\cdots \ar[r] & \widehat{\mathcal{G}} ^q_{Y_1} \ar[r]&  \widehat{\mathcal{G}} ^{q-1} _{Y_1} \ar[r]& \cdots \ar[r]& \widehat{\mathcal{G}} ^0 _{Y_1} \ar[r]& \mathcal{K} (Y_1),}
\end{array}
\end{align}
because the maps in \eqref{eqn:Cech functoriality 1}-\eqref{eqn:Cech functoriality 2} commute with the corresponding maps $\mathcal{G}^q (\widehat{X}_I, U_I )\rightarrow  \mathcal{G}^{q-1} (\widehat{X}_I, U_I )$ induced by the inclusion of supports $\mathcal{D}_{\coh} ^q (\mathfrak X,n)\subseteq \mathcal{D}_{\coh} ^{q-1} (\mathfrak X,n)$.

Notice that the arguments in Lemmas \ref{lem:Cech functor}, \ref{lem:double system compatible} apply \emph{mutatis mutandis} for the tower \eqref{final uscnv fil}, thus for each $n\geq 0$ we obtain the following commutative diagram in the homotopy category:

\begin{align} \label{eqn:pull-back tower}
\begin{array}{c}
\xymatrix{
\cdots \ar[r] & \widehat{\mathcal{G}} ^q_{Y_2,n} \ar[r] \ar[d]_-{g^\ast}& \widehat{\mathcal{G}} ^{q-1} _{Y_2,n} \ar[r] \ar[d]_-{g^\ast}& \cdots \ar[r] & \widehat{\mathcal{G}} ^0 _{Y_2,n} \ar[r] \ar[d]_-{g^\ast}&  \mathcal{K} (Y_2) \ar[d]_-{g^\ast}\\
\cdots \ar[r] & \widehat{\mathcal{G}} ^q_{Y_1,n} \ar[r]&  \widehat{\mathcal{G}} ^{q-1} _{Y_1,n} \ar[r]& \cdots \ar[r]& \widehat{\mathcal{G}} ^0 _{Y_1,n} \ar[r]& \mathcal{K} (Y_1).}
\end{array}
\end{align}

\begin{defn}\label{defn:pb filtrations}
 The diagram \eqref{eqn:pull-back tower} with $n=0$ and the diagram \eqref{eqn:pull-back mtower}, respectively, induce the homomorphisms
 $$
 \tuborg 
  g^*: F_{\cnv} ^q K_n (Y_2) \to F_{\cnv} ^q K_n (Y_1), \\
 g^*: F_{\mcnv} ^q K_n (Y_2) \to F_{\mcnv} ^q K_n (Y_1),
\sluttuborg
 $$
 of abelian groups defined in Definition \ref{defn:cnv cech}. One notes that they give homomorphisms of filtered abelian groups.
 \qed
 \end{defn}
 
 \medskip
 
 To see the above $g^*$ in Definition \ref{defn:pb filtrations} is functorial, let $g_1: Y_1 \to Y_2$ and $g_2: Y_2 \to Y_3$ be two morphisms. A question to check is whether $(g_2 \circ g_1)^* = g_1^* \circ g_2^*$. To see this, first choose an arbitrary system $\mathcal{U}_3$ of local embeddings for $Y_3$. By Lemma \ref{lem:ass system}, there exists a system $\mathcal{U}_2$ for $Y_2$ associated to $\mathcal{U}_3$ by $g_2$, and applying Lemma \ref{lem:ass system} again, there exists a system $\mathcal{U}_1$ for $Y_1$ associated to $\mathcal{U}_2$ by $g_1$. By definition, one sees that $\mathcal{U}_1$ is associated to $\mathcal{U}_3$ by $g_2 \circ g_1$.
 
It follows from Lemma \ref{lem:moving pull-back composed} that $ (g_2 \circ g_1)_{\mathcal{U}_1, \mathcal{U}_3} ^* = (g_1)_{\mathcal{U}_1, \mathcal{U}_2} ^* \circ (g_2)_{\mathcal{U}_2, \mathcal{U}_3}^*$. Thus that $(g_2 \circ g_1)^* = g_1^* \circ g_2^*$ follows by taking homotopy colimits over $\mathcal{U}_1,\mathcal{U}_2, \mathcal{U}_3$, consecutively. We leave out details.

These discussions summarize into the following:

\begin{thm}\label{thm:tower functoriality}
The towers \eqref{final mcnv fil}, \eqref{final uscnv fil} in the homotopy category are functorial. More precisely, we have the following:
\begin{enumerate}
\item For each morphism $g: Y_1 \to Y_2 \in \Sch_k$, we have the 
commutative diagrams of towers in the homotopy category for $n\geq 0$
$$
\xymatrix{
\widehat{\mathcal{G}} ^\bullet _{Y_2,n } \ar[r] \ar[d]_-{g^\ast} &  \mathcal{K} (Y_2) \ar[d]_-{g^\ast} & \widehat{\mathcal{G}} ^\bullet _{Y_2} \ar[r] \ar[d]_-{g^\ast} &  \mathcal{K} (Y_2) \ar[d]_-{g^\ast}\\
\widehat{\mathcal{G}} ^\bullet _{Y_1,n} \ar[r]& \mathcal{K} (Y_1), & \widehat{\mathcal{G}} ^\bullet _{Y_1} \ar[r]& \mathcal{K} (Y_1).}
$$

\item For two morphisms $g_1: Y_1 \to Y_2$ and $g_2: Y_2 \to Y_3$ in $\Sch_k$, we have the equality in the homotopy category for $n\geq 0$
$$
\tuborg
(g_2 \circ g_1)^* = g_1 ^* \circ g_2 ^*: \widehat{\mathcal{G}} ^\bullet _{Y_3,n}  \to \widehat{\mathcal{G}} ^\bullet _{Y_1,n}, \\
(g_2 \circ g_1)^* = g_1 ^* \circ g_2 ^*: \widehat{\mathcal{G}} ^\bullet _{Y_3}  \to \widehat{\mathcal{G}} ^\bullet _{Y_1}.
 \sluttuborg
$$
\end{enumerate}

\end{thm}

Taking homotopy groups we conclude:

\begin{thm}\label{thm:cech functoriality}
The coniveau filtration (resp. motivic coniveau filtration) of Definition \ref{defn:cnv cech} on the algebraic $K$-theory of schemes in  $\Sch_k$ is functorial. More precisely, we have the following:
\begin{enumerate}
\item For each morphism $g: Y_1 \to Y_2 \in \Sch_k$, we have homomorphisms of filtered abelian groups
$$
\tuborg
g^*: F_{\cnv} ^{\bullet} K_n (Y_2) \to F_{\cnv} ^{\bullet} K_n (Y_1), \\
g^*: F_{\mcnv} ^{\bullet} K_n (Y_2) \to F_{\mcnv} ^{\bullet} K_n (Y_1).
\sluttuborg
$$
\item For two morphisms $g_1: Y_1 \to Y_2$ and $g_2: Y_2 \to Y_3$ in $\Sch_k$, we have the equalities
$$
\tuborg
(g_2 \circ g_1)^* &= g_1 ^* \circ g_2 ^*: F_{\cnv} ^{\bullet} K_n (Y_3) \to F_{\cnv} ^{\bullet} K_n (Y_1), \\
(g_2 \circ g_1)^* &= g_1 ^* \circ g_2 ^*: F_{\mcnv} ^{\bullet} 
K_n (Y_3) \to F_{\mcnv} ^{\bullet} K_n (Y_1).
\sluttuborg
$$
\end{enumerate}

\end{thm}

Recall that for a filtered abelian group $F^{\bullet} A$, its $n$-th graded associated is the group $gr_F^n A:= F^n A/ F^{n+1} A$. Its direct sum is written $gr_F ^{\bullet} A: = \bigoplus_n gr_F ^n A$. From Theorem \ref{thm:cech functoriality}, we deduce immediately:

\begin{thm}\label{thm:cech gr functoriality}
The graded associated of the coniveau filtration (resp. motivic coniveau filtration) on the algebraic $K$-theory of schemes in $\Sch_k$ is functorial. More precisely, we have the following:
\begin{enumerate}
\item For each morphism $g: Y_1 \to Y_2 \in \Sch_k$, we have homomorphisms of filtered abelian groups
$$
\tuborg
g^*: gr_{\cnv} ^{\bullet} K_n (Y_2) \to gr_{\cnv} ^{\bullet} K_n (Y_1), \\
g^*: gr_{\mcnv} ^{\bullet} K_n (Y_2) 
\to gr_{\mcnv} ^{\bullet} K_n (Y_1).
\sluttuborg
$$
\item For two morphisms $g_1: Y_1 \to Y_2$ and $g_2: Y_2 \to Y_3$ in $\Sch_k$, we have the equalities
$$
\tuborg
(g_2 \circ g_1)^*= g_1 ^* \circ g_2 ^*: gr_{\cnv} ^{\bullet} K_n (Y_3) \to gr_{\cnv} ^{\bullet} K_n (Y_1), \\
(g_2 \circ g_1)^*= g_1 ^* \circ g_2 ^*: gr_{\mcnv} ^{\bullet}  K_n (Y_3) \to gr_{\mcnv} ^{\bullet} K_n (Y_1).
\sluttuborg
$$
\end{enumerate}
\end{thm}

\begin{remk}

Note that we can apply Theorem \ref{thm:cech functoriality} to morphisms $g: Y_1 \to Y_2$ between smooth irreducible quasi-projective $k$-schemes. We saw that the classical coniveau filtration and our new coniveau filtration coincide (Proposition \ref{prop:sm cnv}) for each $Y_j$, $j=1,2$, so that the theorem shows 
that the pull-back $g^*: K_i (Y_2) \to K_i (Y_1)$ respects the classical coniveau filtration.

This is already known, e.g. see H. Gillet \cite[Theorem 83, Lemma 84, p.283]{Gillet}, which is based on the technique of deformation to the normal cone and $\mathbb{A}^1$-invariance of the algebraic $K$-theory of regular schemes.

Our proof given in this article via formal schemes and a generalized \v{C}ech machine offers a new argument even for the classical smooth case, and it does not explicitly rely on the $\mathbb{A}^1$-invariance.\qed
\end{remk}

More about the connections between cycles in terms of the \emph{new} higher Chow groups of \cite{P general} and the above motivic coniveau filtration on $K_i (Y)$ will be pursued in follow-up papers.

 \bigskip

\noindent\emph{Acknowledgments.} 
Some initial ideas were conceived when the authors met at the Special Session on Algebraic Geometry of the third PRIMA Congress at Oaxaca, M\'exico in the summer 2017, and also at the ICM 2018 Satellite Conference on $K$-theory at the University of Buenos Aires in Argentina. 
The authors thank the organizers of the conferences, Guillermo Corti{\~{n}}as, Pedro L. del \'Angel, Kenichiro Kimura, and Gisela Tartaglia for their kind invitations, that allowed the opportunities.

This work was supported by Samsung Science and Technology Foundation under Project Number SSTF-BA2102-03.

\end{document}